\documentclass[11pt]{amsart} \textwidth=14.5cm \oddsidemargin=1cm
\usepackage{amsmath,wasysym}
\usepackage{amsxtra}
\usepackage{amscd}
\usepackage{amsthm}
\usepackage{amsfonts}
\usepackage{amssymb}
\usepackage{eucal}
\usepackage{hyperref,mathrsfs}
\usepackage{tikz}
\usepackage{tikz-cd}
\usepackage{stmaryrd}
\usepackage[all]{xy}
\usepackage{bbm} % \mathbbm
\usepackage{upgreek}
\usepackage{extarrows}
\allowdisplaybreaks

\hypersetup{colorlinks,linkcolor=blue, urlcolor=blue,
	citecolor=blue}

\newtheorem{thm}{Theorem} [section]
\newtheorem{lem}[thm]{Lemma}
\newtheorem{cor}[thm]{Corollary}
\newtheorem{prop}[thm]{Proposition}

\theoremstyle{definition}

\newtheorem{example}[thm]{Example}

\newtheorem{rem}[thm]{Remark}

\newtheorem*{thmA}{Main Result 1}
\newtheorem*{thmB}{Main Result 2}
\newtheorem*{thmC}{Main Result 3}

\numberwithin{equation}{section}

\newcommand{\Bm}{\mathfrak{B}}
\newcommand{\C}{\mathbb{C}}
\newcommand{\D}{\mathcal{D}}

\newcommand{\End}{{\mathrm{End}}}

\newcommand{\ff}{\mathbf{f}}
\newcommand{\Fm}{\mathfrak{F}}
\newcommand{\g}{\mathbf{g}}
\newcommand{\Hom}{\mathrm{Hom}}
\newcommand{\id}{\mathbbm1}
\newcommand{\K}{\mathbb{K}}
\newcommand{\N}{\mathbb{N}}

\newcommand{\Ob}{\mathscr{O}}
\newcommand{\rank}{\mathrm{rank}}

\newcommand{\Z}{\mathbb{Z}}
\newcommand{\Mat}{\mathrm{Mat}}
\newcommand{\nc}{\newcommand}
\nc{\browntext}[1]{\textcolor{brown}{#1}}
\nc{\greentext}[1]{\textcolor{green}{#1}}
\nc{\redtext}[1]{\textcolor{red}{#1}}
\nc{\bluetext}[1]{\textcolor{blue}{#1}}
\nc{\brown}[1]{\browntext{#1}}
\nc{\green}[1]{\greentext{#1}}
\nc{\red}[1]{\redtext{#1}}
\nc{\blue}[1]{\bluetext{#1}}

\title[Construction of affine quantum Schur algebras and $\mathrm{i}$quantum groups]
{K-theoretic construction of affine quantum Schur algebras and $\mathrm{i}$quantum groups with unequal parameters}

\author[Li Luo]{Li Luo}
\author[Zheming Xu]{Zheming Xu}
\author[Yang Yang]{Yang Yang}
\address{School of Mathematical Sciences, Key Laboratory of MEA (Ministry of Education) \& Shanghai Key Laboratory of PMMP, East China Normal University, Shanghai 200241, China}
\address{Department of Mathematics, Pohang University of Science and Technology, Pohang 790-784, Republic of Korea}
\email{lluo@math.ecnu.edu.cn (Luo)}
\email{zhemingxu@postech.ac.kr (Xu)}
\email{52275500011@stu.ecnu.edu.cn (Yang)}

\begin{document}

\begin{abstract}
We present an equivariant K-theoretic construction of affine quantum Schur algebras of type C with three parameters by taking advantage of Kato's exotic framework. This enables the derivation of an equivariant K-theoretic realization of affine iquantum groups of type AIII with three parameters. Additionally, in Appendix A, we provide an equivariant K-theoretic construction of affine quantum Schur algebras of types B/C/F/G with two parameters, following Antor's recent work on affine Hecke algebras with two parameters.
\end{abstract}

\maketitle

\tableofcontents

%-----------------------------------------------------
\section{Introduction}
%-----------------------------------------------------
\subsection{Background}
There are two different geometric realizations of the affine Hecke algebra $\widetilde{\mathbf{H}}$ associated with an algebraic group $G$. One is due to Iwahori and Matsumoto (cf. \cite{IM65}), saying that $\widetilde{\mathbf{H}}$ is isomorphic to the convolution algebra of $I$-bi-invariant functions on $G(\mathbb{Q}_p)$ with compact support, where $\mathbb{Q}_p$ is the $p$-adic field and $I$ is an Iwahori subgroup of $G(\mathbb{Q}_p)$. The other is initiated by Kazhdan and Lusztig \cite{Lu85,KL85,KL87}, showing that $\widetilde{\mathbf{H}}$ is isomorphic to the equivariant $K$-group of the Steinberg variety attached to the Langlands dual of $G$. The isomorphism between these two geometric constructions is called the Langlands reciprocity, whose categorification was given by Bezrukavnikov in \cite{Be16}.

The above two geometric realizations have been carried out for affine quantum Schur algebras (of type A) and further for affine quantum $\mathfrak{gl}_n$ by using $n$-step flag varieties instead of complete flag varieties.
See Lusztig and Du-Fu's work \cite{Lu99, DF15} for the Iwahori-Matsumoto-type realization of affine quantum Schur algebras and affine quantum $\mathfrak{gl}_n$, which is the affine version of the influential work \cite{BLM90} due to Beilinson, Lusztig and MacPherson (abbr. BLM); and see Ginzburg and Vasserot's work \cite{GV93, Va98} for their equivariant K-theoretic counterparts.

At this stage, finite and affine type B/C/D analogs of BLM-type and equivariant K-theoretic constructions have also been obtained. In contrast to type A, it is not the Drinfeld-Jimbo quantum groups but rather iquantum groups that occur in these realizations. Historically, iquantum groups first appeared in Letzter's work \cite{Le99} on quantum symmetric pairs, and their affinization was later provided by Kolb \cite{Ko14}.
Thanks to Bao and Wang's work \cite{BW18} on a reformulation of the Kazhdan-Lusztig theory for types B/C/D based on a Schur-type duality, iquantum groups (especially of type AIII), as well as quantum Schur algebras (especially of types B/C/D), have begun to attract attention. 
These quantum algebras exhibit geometric structures as rich as those of type A. Bao, Kujawa, Li and Wang \cite{BKLW18} gave a geometric construction of the quantum Schur algebras of finite type B/C and hence of the quasi-split iquantum groups of type AIII as subalgebras of their inverse limits, which generalizes the BLM construction of quantum Schur algebras and quantum $\mathfrak{gl}_n$ to type B/C. A BLM-type realization of quantum Schur algebras of affine type C and quasi-split iquantum groups of affine type AIII was obtained in \cite{FLLLW20, FLLLW23}, while an equivariant K-theoretic construction of quasi-split iquantum groups of affine type AIII was developed in \cite{SW26, LSX26} (see also \cite{FMX22} for finite type AIII). Furthermore, the first author and his collaborators \cite{LW22, CLW24} generalized the notion of $n$-step partial flag variety in terms of a finite subset $Q_\mathbf{f}$ consisting of orbits of Weyl groups on the (co)weight lattice for arbitrary finite/affine type, and then introduced a series of generalized quantum Schur algebras. For some special choices of $Q_\mathbf{f}$, this recovers the original finite/affine quantum Schur algebras of classical types. These generalized quantum Schur algebras also admit a BLM-type realization (see \cite{CLW24}) and an equivariant K-theoretic realization (see \cite{LXY26}), which form a Schur algebra analog of Langlands reciprocity.

For the non-simply-laced case (i.e. types B/C/F/G), Hecke algebras and quantum Schur algebras with unequal parameters arise naturally, yet their geometric architectures have not yet been fully investigated. The Iwahori-Matsumoto-type realization of Hecke algebras works only in very restricted cases, namely for certain specific integral powers of a single parameter. A significant breakthrough was made by Kato \cite{Ka09}, who provided an equivariant K-theoretic construction of affine Hecke algebras of type C with three independent parameters. Recently, Antor \cite{An25} established an equivariant K-theoretic construction of all affine Hecke algebras of types B/C/F/G with two independent parameters. The key innovation in these works is the introduction of an exotic representation, a direct sum of two or three submodules whose nonzero weight sets correspond to short and long roots, respectively. 

In this paper, we show that this exotic setup can be systematically adapted to realize affine quantum Schur algebras with unequal parameters as well. Concretely, we provide an equivariant K-theoretic construction of affine quantum Schur algebras of type C with three parameters and of types B/C/F/G with two parameters. Addressing this issue is not a mere technical extension. As mentioned above, affine type A quantum Schur algebras serve as a crucial bridge between affine quantum $\mathfrak{gl}_n$ and Hecke algebras; their geometric realization is a prerequisite for a deeper understanding of Langlands-type correspondences at the level of affine quantum $\mathfrak{gl}_n$. The same philosophy applies to the case of affine type C with three parameters. As a further application of our construction, we establish algebra homomorphisms from affine iquantum groups to affine quantum Schur algebras via the equivariant K-theoretic approach, which is a multi-parameter analog of the works \cite{SW26, LSX26}. 
Finally, we note that an algebraic BLM-type realization (using lengths of Weyl group elements instead of the dimensions of certain varieties) of the affine quantum Schur algebras of type C and the affine iquantum groups of type AIII with three parameters can be found in a companion paper \cite{LY27} (see also \cite{LL21} for finite quantum Schur algebras and iquantum groups of type AIII with two parameters).

\subsection{Main results}
%-------------------------------------
\subsubsection{Type C with three parameters}
Let $G=\mathrm{Sp}_{2d}(\C)$ be a complex symplectic group. Fix a maximal torus $T$ and a Borel subgroup $B$ of $G$ with $T\subseteq B$, and let $W=N_G(T)/T$ be the associated Weyl group (of type C). Let $\mathbb{V}=V_1\oplus V_1\oplus V_2$ be Kato's exotic representation of $G$, where $V_1$ is the standard module of $G$ and $V_2=\wedge^2V_1$.  The set $R$ of nonzero weights of $\mathbb{V}$ corresponds to the root system $\Pi$ of $G$ by a $W$-equivariant map $\Psi: \Pi\to R$.
In the K-theoretic construction with unequal parameters, the positive part $\mathbb{V}^+=\bigoplus_{\lambda\in\Psi(\Pi^+)}\mathbb{V}[\lambda]$ of $\mathbb{V}$ plays a role analogous to that of the nilradical subalgebra $\mathfrak{n}\subseteq\mathfrak{g}=\mathrm{Lie}(G)$ in the case of a single parameter. In this context, the cotangent bundle of the complete flag variety $G/B$ is replaced by $F=G\times ^{B}\mathbb{V}^+$. Let $\mathcal{N}$ be the exotic nilpotent cone, i.e. the image of the projection $F\to \mathbb{V}$, $(g,v)\mapsto g\cdot v$. This cone admits a $G\times (\C^\times)^3$-action. The exotic Steinberg variety is defined as $\mathbf{Z}=F\times_\mathcal{N} F$. Kato \cite{Ka09} proved that the affine Hecke algebra $\widetilde{\mathbb{H}}$ of type C with three parameters over $\mathbb{Z}[q_0^{\pm1},q_1^{\pm1},q_2^{\pm1}]$ is isomorphic to the equivariant K-group $K^{G\times (\C^\times)^3}(\mathbf{Z})$. 

Fix a finite set $\Lambda_\mathbf{f}$ of $W$-orbits in the weight lattice of $G$. The affine quantum Schur algebra with three parameters is defined as the endomorphism algebra $\widetilde{\mathbb{S}}_\mathbf{f}=\mathrm{End}_{\widetilde{\mathbb{H}}}(\bigoplus_{\mathbf{v}\in\Lambda_\mathbf{f}}x_\mathbf{v}\widetilde{\mathbb{H}})$, where $x_\mathbf{v}\in\widetilde{\mathbb{H}}$ is the $q$-symmetrizer associated with the $W$-orbit $\mathbf{v}$.
Each $W$-orbit $\mathbf{v}\in\Lambda_\mathbf{f}$ determines a root subsystem $\Pi_\mathbf{v}\subseteq\Pi$ and parabolic subgroups $P_\mathbf{v}\subseteq G$ and $W_\mathbf{v}\subseteq W$. Set $\mathbb{V}_\mathbf{v}^+=\bigoplus_{\lambda\in\Psi(\Pi^+)\backslash\Psi(\Pi_\mathbf{v}^+)}\mathbb{V}[\lambda]$, which is a $P_{\mathbf{v}}$-module. We introduce the exotic generalized Steinberg variety
$\mathbf{Z}_\mathbf{f}=\bigsqcup_{\mathbf{v},\mathbf{w}\in\Lambda_\mathbf{f}} F_\mathbf{v}\times_{\mathcal{N}} F_\mathbf{w}$, where $F_\mathbf{v}=G\times^{P_\mathbf{v}}\mathbb{V}_\mathbf{v}^+$.

\begin{thmA}[Theorem~\ref{main}]
The affine type C quantum Schur algebra $\widetilde{\mathbb{S}}_\mathbf{f}$ with three parameters is isomorphic to the equivariant K-group $K^{G\times (\C^\times)^3}(\mathbf{Z}_\mathbf{f})$.
\end{thmA}

We emphasize that this K-theoretic construction for quantum Schur algebras is not merely an imitation of the approach for affine Hecke algebras established in \cite{Ka09}.  In our case, we generally cannot obtain a good basis of $K^{G\times (\C^\times)^3}(\mathbf{Z}_\mathbf{f})$ through the cellular fibration lemma. Instead, we adopt the indirect approach introduced in our previous work \cite{LXY26} to address this issue.

\subsubsection{Affine iquantum groups of type AIII}
%=========================================================
When we take $\Lambda_\mathbf{f}$ as in \eqref{eq:lambdaj} and \eqref{eq:lambdai}, the exotic generalized Steinberg varieties are denoted by $\mathbf{Z}_n^\jmath$ and $\mathbf{Z}_n^\imath$, respectively. The associated algebra $\widetilde{\mathbb{S}}_\mathbf{f}\simeq K^{G\times (\C^\times)^3}(\mathbf{Z}_n^\imath)$ (resp., $K^{G\times (\C^\times)^3}(\mathbf{Z}_n^\jmath)$) is just the three-parameter affine quantum Schur algebra $\mathbb{S}_{n,d}^{\imath\imath}$ (resp., $\mathbb{S}_{n,d}^{\imath\jmath}$) studied in \cite{FLLLWW20, LY27}. These two cases are of central importance due to their close connections with the affine type AIII iquantum groups $\mathbb{U}^{\imath\imath}$ and $\mathbb{U}^{\imath\jmath}$. Thus, as a three-parameter analog of the work \cite{SW26, LSX26}, we derive the following equivariant K-theoretic realization of $\mathbb{U}^{\imath\imath}$ and $\mathbb{U}^{\imath\jmath}$, which is proven by checking the generating relations of the Drinfeld new presentations of $\mathbb{U}^{\imath\imath}$ and $\mathbb{U}^{\imath\jmath}$ given in \cite{LWZ24, LPWZ25}, respectively.
\begin{thmB}[Theorems~\ref{thm:ii} \& \ref{thm:ij}] 
We provide explicit formulations for the homomorphisms from the affine iquantum groups $\mathbb{U}^{\imath\imath}$ and  $\mathbb{U}^{\imath\jmath}$ to (a localization of) the equivariant K-groups $K^{G\times (\C^\times)^3}(\mathbf{Z}_n^\imath)$ and $K^{G\times (\C^\times)^3}(\mathbf{Z}_n^\jmath)$, respectively.
\end{thmB}

%As mentioned in \cite[\S3.3]{FLLLWW20}, some specializations of the parameters $q,q_0,q_1$ indicate different types, e.g. $q_0=q_1$ for affine type B with two parameters $q,q_1$, which can be further specialized to the case of a single parameter by letting $q_0=q_1=q$; $q_0=1$ and $q_1=q^2$ for affine type C with a single parameter; and $q_0=q_1=1$ for affine type D. This means that affine quantum Schur algebras and affine iquantum groups associated with algebraic groups other than $Sp_{2d}$ can be obtained at some specialization of those with three parameters.  

In contrast to Vasserot's work \cite{Va98} on type A, there is a non-closed $\mathrm{Sp}_{2d}$-orbit on the double ``$n$-step'' flag variety for type C when we consider the iquantum group $\mathbb{U}^{\imath\imath}$. Su and Wang \cite{SW26} treated this non-closed orbit using a localization method. In our exotic setting for the case of three parameters, we also need to apply localization to overcome this complication. For the case of a single parameter, one can instead take $G=\mathrm{O}_{2d}$ to realize $\mathbb{U}^{\imath\imath}$ and thereby avoid the localization method (cf. \cite[Appendix A]{LSX26}).

%---------------------------------------------------
\subsubsection{Types B/C/F/G with two parameters}
%===============================================================
Kato's exotic setup for the case of three parameters has been generalized to the case of two parameters by Antor in \cite{An25}, where $G$ is an algebraic group of types B/C/F/G over an algebraically closed field $\mathbf{k}$ with the special characteristic so that the adjoint module $\mathfrak{g}$ of $G$ has a submodule $\mathfrak{g}_s$ whose nonzero weights are exactly the short roots. Antor introduced the exotic representation as $\mathbb{V}=\mathfrak{g}_s\oplus \mathfrak{g}/\mathfrak{g}_s$, which enables the definition of related exotic objects such as the exotic flag variety $F$ and the exotic Steinberg variety $\mathbf{Z}$. As demonstrated in \cite{An25}, the two-parameter affine Hecke algebra of types B/C/F/G over $\mathbb{Z}[q_1^{\pm1},q_2^{\pm1}]$ is isomorphic to the equivariant K-group $K^{G\times (\mathbf{k}^\times)^2}(\mathbf{Z})$. Similarly, the generalized exotic Steinberg variety $\mathbf{Z}_\mathbf{f}$ and the two-parameter affine Schur algebra $\widetilde{\mathbb{S}}_\mathbf{f}$ of types B/C/F/G over $\mathbb{Z}[q_1^{\pm1},q_2^{\pm1}]$ can also be introduced for a finite set $\Lambda_\mathbf{f}$ of $W$-orbits. In Appendix A, we present a two-parameter counterpart of Main Result 1.

\begin{thmC}[Theorems~\ref{main2}]
The two-parameter affine quantum Schur algebra $\widetilde{\mathbb{S}}_\mathbf{f}$ of types B/C/F/G is isomorphic to the equivariant K-group $K^{G\times (\mathbf{k}^\times)^2}(\mathbf{Z}_\mathbf{f})$.
\end{thmC}

Of course, for types B and C, we can further consider the construction of two-parameter iquantum groups. Since they are special cases of the three-parameter ones and the construction methods are similar, we will not elaborate on this process in this paper.

%--------------------------
\subsection{Organization}
%----------------------------
Here is the layout of the paper.

In Section~\ref{sec:pre}, we recall some basic notions, including Kato's exotic geometric setting.
Section~\ref{sec:conssch} is devoted to the equivariant K-theoretic realization of affine quantum Schur algebras of type C with three parameters, while Section~\ref{sec:consiquan} is devoted to the equivariant K-theoretic realization of the quasi-split affine iquantum groups of type AIII with three parameters. 
In Appendix~\ref{Apped:A}, we employ Antor's exotic setting to give the equivariant K-theoretic realization of affine quantum Schur algebras of type B/C/F/G with two parameters. In Appendix~\ref{Apped:B}, we summarize some standard facts about algebraic equivariant K-theory for the convenience of readers.

\subsubsection*{Acknowledgments}
The work is partially supported by the National Key R\&D Program of China (No. 2024YFA1013802), the NSF of China (Nos. 12371028, 125B2001 and 12631002).

%====================================================
\section{Preliminaries}\label{sec:pre}
%======================================================
In this section, we recall combinatorial and geometric structures of the symplectic group and its Weyl group, particularly including Kato's exotic setting. We also give the definition of the Hecke and Schur algebras of affine type C with three parameters.

%-------------------------------------------------------
\subsection{Symplectic group and its Weyl group orbits}
%------------------------------------------------------
Throughout this paper (except for Appendix~\ref{Apped:A}), we take $G=\mathrm{Sp}_{2d}(\C)$ the complex symplectic group of rank $d$, which is a connected, simply connected, simple algebraic group over $\C$. 

Fix a maximal torus and a Borel subgroup $T\subseteq B\subseteq G$. Let $W=N_G(T)/T$ be the Weyl group of $G$ associated with $T$. 
Let $$\Pi=\{\pm\epsilon_{i}\pm\epsilon_{j}\}_{1\le i\neq j\le d}\cup\{\pm2\epsilon_{i}\}_{1\le i\le d}$$ be the root system of $G$ with simple roots $$\Delta=\{\alpha_1=\epsilon_1-\epsilon_2, \ldots, \alpha_{d-1}=\epsilon_{d-1}-\epsilon_d, \alpha_d=2\epsilon_d\}.$$
By abuse of notation, we also consider the Weyl group $W$ as the permutation group of the weight lattice $Q=\sum_{i=1}^d \mathbb{Z}\epsilon_i$
generated by simple reflections $s_1,\ldots,s_d$ associated with $\alpha_1,\ldots,\alpha_d$, respectively.

Take a finite $W$-invariant subset $Q_\ff\subseteq Q$, and denote
	\begin{align*}
		\Lambda =Q/W,\quad \Lambda_\ff =Q_\ff/W
	\end{align*} the sets of $W$-orbits.
 The choice of $Q_\ff$ is flexible and far from unique. Given $n\in\mathbb{N}$, the two most important choices are 
\begin{align*}
  Q_n^\jmath=\{\sum_{i=1}^da_i\epsilon_i \in Q \mid |a_i| \leq n,\ \forall i\} \quad \text{and}\quad
 Q_n^\imath=\{\sum_{i=1}^da_i\epsilon_i \in Q \mid 0\neq|a_i|\leq n,\ \forall i\},
\end{align*}
which are in close relation to iquantum groups of type AIII.
 In these two cases, the set $\Lambda_\ff$ can be identified, respectively, as
 \begin{align}\label{eq:lambdaj}
 \Lambda_{n,d}^\jmath&=\{\mathbf{v}=(v_1, v_2,\ldots,v_{2n+1}) \in \N^{2n+1}~|~v_i=v_{2n+2-i},\sum_{i=1}^{2n+1}v_i=2d\} \quad \text{and}\\
 \label{eq:lambdai}
 \Lambda_{n,d}^\imath&=\{\mathbf{v}=(v_1, v_2,\ldots,v_{2n})\in\N^{2n}~|~v_i=v_{2n+1-i},\sum_{i=1}^{2n}v_i=2d\},
 \end{align}
where $\mathbf{v} \in \Lambda_{n,d}^\imath$ represents the orbit consisting of the weights $\sum_{i=1}^d a_i\epsilon_i\in Q_n^\imath$ such that \begin{equation}\label{eq:vk}
v_k=\#\{1\leq i\leq d~|~|a_i|=n+1-k\}\quad  (k=1,2,\ldots,n), 
\end{equation}
and $\mathbf{v} \in \Lambda_{n,d}^\jmath$ represents the orbit consisting of the weights $\sum_{i=1}^d a_i\epsilon_i\in Q_n^{\jmath}$ with $v_1,\ldots v_n$ as in \eqref{eq:vk} and $v_{n+1}=2\#\{1\leq i\leq d~|~a_i=0\}$.

%--------------------
\subsection{Parabolic subgroups}
 %----------------------
 For each $W$-orbit $\mathbf{v} \in \Lambda$, there exists a unique anti-dominant weight $\mathbf{i}_\mathbf{v}\in\mathbf{v}$. Let
 $$\Delta_\mathbf{v}=\{\alpha_k \in \Delta~|~s_k\mathbf{i}_\mathbf{v} =\mathbf{i}_\mathbf{v}\},$$
 and $\Pi_\mathbf{v}=\Pi_\mathbf{v}^+ \sqcup\Pi_\mathbf{v}^-$ the associated root system.
 Let $W_\mathbf{v}$ be the parabolic subgroup of $W$ associated with $\Delta_\mathbf{v}$.
Denote
	$$\D_\mathbf{v}=\{w \in W  ~|~  \ell(vw)= \ell(w)-\ell(v),\ \forall v \in W_\mathbf{v}\}.$$
	Then $\D_\mathbf{v}$ (resp., $\D_\mathbf{v}^{-1}$) is the set of distinguished longest length
	right (resp., left) coset representatives of $W_\mathbf{v}$ in $W$.
Let $\D_{\mathbf{v}\mathbf{w}}=\D_\mathbf{v} \cap \D_\mathbf{w}^{-1}$ be the set of longest length double coset representatives of $W_\mathbf{v}\backslash W/W_\mathbf{w}$.
For $w \in W$, denote by $w_{\mathbf{v}\mathbf{w}}$ the unique longest element in $W_\mathbf{v} wW_\mathbf{w}$.

Denote by $P_{\mathbf{v}}$ the standard parabolic subgroup associated with $\Delta_{\mathbf{v}}$. For any $w\in W$, we denote
 \begin{equation*}
 W_{\mathbf{v}\mathbf{w}}^w=W_\mathbf{v} \cap w W_\mathbf{w} w^{-1} \quad \text{and} \quad P_{\mathbf{v}\mathbf{w}}^w=P_\mathbf{v} \cap w P_\mathbf{w} w^{-1}.
 \end{equation*} 
 The Weyl group of the Levi part of $P_{\mathbf{v}\mathbf{w}}^w$ is just $W_{\mathbf{v}\mathbf{w}}^w$.  %We remark that $W_{\mathbf{v}\mathbf{w}}^w=W_{\mathbf{v}\mathbf{w}}^{w'}$ if and only if $ P_{\mathbf{v}\mathbf{w}}^w= P_{\mathbf{v}\mathbf{w}}^{w'}$ if and only if $w$ and $w'$ are in the same double coset representatives of $W_\mathbf{v}\backslash W/W_\mathbf{w}$ (i.e., $W_\mathbf{v} wW_\mathbf{w}=W_\mathbf{v} w'W_\mathbf{w}$).

Let us fix a regular $W$-orbit $\mathbf{d}\in\Lambda$, e.g., 
$\mathbf{d}=(\underbrace{1,\ldots,1}_{2d})\in\Lambda_{d,d}^{\imath}$.
It is clear that $\Delta_\mathbf{d}=\emptyset$, $W_\mathbf{d}=\{\id\}$ and $P_\mathbf{d}=B$. 

When $\Delta_{\mathbf{v}}=\{\alpha_i\}$, we simply write $P_i=P_\mathbf{v}$.

%----------------------------------
\subsection{Flag varieties}
%-----------------------------------
Let $\Fm_\mathbf{v}=G/P_\mathbf{v}$ be the partial variety associated with $\mathbf{v}$, which carries a natural $G$-action. In particular, $\Fm_\mathbf{d}$ is just the complete flag variety $\Bm=G/B$.

For $w \in W$, denote by
$\Ob_{\mathbf{v},w,\mathbf{w}}= G\cdot(P_\mathbf{v}, wP_\mathbf{w})$ the diagonal $G$-orbit in $\Fm_\mathbf{v}\times\Fm_\mathbf{w}$. 
The orbits $\Ob_{\mathbf{v},w,\mathbf{w}}$ and $\Ob_{\mathbf{v},w',\mathbf{w}}$ coincide if and only if $W_\mathbf{v} wW_\mathbf{w}=W_\mathbf{v} w'W_\mathbf{w}$. We have a natural isomorphism $\Ob_{\mathbf{v},w,\mathbf{w}}\simeq G/P^w_{\mathbf{v}\mathbf{w}}$ via $g\cdot(P_\mathbf{v},wP_\mathbf{w})\mapsto gP^w_{\mathbf{v}\mathbf{w}}$.

Set $$\Fm_{\ff}=\bigsqcup_{\mathbf{v}\in\Lambda_{\ff}}\Fm_{\mathbf{v}}.$$ There is a one-to-one correspondence between the set $$\Xi_{\ff}=\{(\mathbf{v},w,\mathbf{w})~|~\mathbf{v},\mathbf{w}\in\Lambda_\ff,\ w\in\D_{\mathbf{v}\mathbf{w}}\}$$ and the set of all $G$-diagonal orbits on $\Fm_\ff\times\Fm_\ff=\bigsqcup_{\mathbf{v,w}\in\Lambda_{\ff}}\Fm_{\mathbf{v}}\times \Fm_{\mathbf{w}}$, sending $(\mathbf{v},w,\mathbf{w})$ to the orbit $\Ob_{\mathbf{v},w,\mathbf{w}}$. 
Actually,
$$\Fm_\mathbf{v}\times\Fm_\mathbf{w}=\bigsqcup_{w\in\D_{\mathbf{v}\mathbf{w}}} \Ob_{\mathbf{v},w,\mathbf{w}}.$$

For $(\mathbf{v},w,\mathbf{w})$ and $(\mathbf{v}',w',\mathbf{w}')\in\Xi_\mathbf{f}$, $\Ob_{\mathbf{v},w,\mathbf{w}}\subseteq\overline{\Ob}_{\mathbf{v}',w',\mathbf{w}'}$ if and only if $\mathbf{v}=\mathbf{v}', \mathbf{w}=\mathbf{w}'$ and $w\leq w'$ under the Bruhat order.

Let $V=\C^{2d}$ be the standard module of $G$, equipped with a non-degenerate skew-symmetric bilinear form. For each $\mathbf{v} \in \Lambda_{n,d}^\jmath$ (resp., $\Lambda_{n,d}^\imath$), we have a canonical $G$-equivariant isomorphism 
\begin{align*}
\Fm_\mathbf{v}\simeq \{(0=V_0\subseteq V_1\subseteq\cdots\subseteq V_{2n+1}=V)~|~V_i=V_{2n+1-i}^\perp,\ \dim V_i=\bar{v}_i, \ \forall i\}\\
(\text{resp., } \Fm_\mathbf{v}\simeq\{(0=V_0\subseteq V_1\subseteq\cdots\subseteq V_{2n}=V)~|~V_i=V_{2n-i}^\perp,\ \dim V_i=\bar{v}_i, \ \forall i\}),
\end{align*}
where 
\begin{equation}\label{def:tildevi}
\bar{v}_i=v_1+\cdots+v_i.
\end{equation}
Denote $\Fm_n^\jmath=\bigsqcup_{\mathbf{v}\in\Lambda_{n,d}^\jmath}\Fm_\mathbf{v}$ and $\Fm_n^\imath=\bigsqcup_{\mathbf{v}\in\Lambda_{n,d}^\imath}\Fm_\mathbf{v}$.
In these cases, the set $\Xi_{\ff}$ can be described, respectively, by the following sets of matrices: (cf. \cite[\S6]{BKLW18})
\begin{align*}
\Xi_{n,d}^\jmath&=\{A=(a_{ij})\in\mathrm{Mat}_{2n+1}(\mathbb{N})~|~\sum_{i,j=1}^{2n+1}{a_{ij}}=2d,\ a_{ij}=a_{2n+2-i,2n+2-j},\ \forall i,j\},\\
\Xi_{n,d}^\imath&=\{A=(a_{ij})\in\mathrm{Mat}_{2n}(\mathbb{N})~|~\sum_{i,j=1}^{2n}{a_{ij}}=2d,\ a_{ij}=a_{2n+1-i,2n+1-j},\ \forall i,j\}.
\end{align*} 
Precisely, we present $(\mathbf{v},w,\mathbf{w})\in\Xi_{\ff}$ by $A=(a_{ij})\in \Xi_{n,d}^\mathfrak{c}$ ($\mathfrak{c}=\imath,\jmath$)  such that \begin{align*}
\mathbf{v}&=\mathrm{ro}(A):=(\sum_{j=1}a_{1j},\sum_{j=1}a_{2j},\ldots), \\
\mathbf{w}&=\mathrm{co}(A):=(\sum_{i=1}a_{i1},\sum_{i=1}a_{i2},\ldots),
\end{align*}
and a condition about the unique minimal length representative in the double coset $W_\mathbf{v}wW_\mathbf{w}$. We refer to \cite{FLLLW23} for this condition, which is a bit wordy but unnecessary in this paper.
Moreover, for $A=(a_{ij}),\ B=(b_{ij})\in\Xi_{n,d}^\mathfrak{c}\ (\mathfrak{c}=\jmath,\imath)$, $\Ob_A\subseteq \overline{\Ob}_B$ if and only if
$$\mathrm{ro}(A)=\mathrm{ro}(B),\quad \mathrm{co}(A)=\mathrm{co}(B) \quad \text{and} \quad \sum_{r\leq i, s\geq j} a_{rs}\leq \sum_{r\leq i, s\geq j} b_{rs}, \quad \forall i<j.$$

%------------------------
\subsection{The exotic representation}
%--------------------------------
Recall that $V=\C^{2d}$ is the standard module of $G$.  
 Kato \cite{Ka09} introduced $$\mathbb{V}=V^{\oplus 2} \oplus \wedge^2V,$$ which is named the {\em exotic representation} of $G$. The set of nonzero weights of $\mathbb{V}$ is $$R=\{\pm\epsilon_i\pm\epsilon_j\}_{1\le i\neq j\le d}\cup\{\pm\epsilon_i\}_{1\le i\le d},$$ which corresponds one-to-one to the root system $\Pi$
by a $W$-equivariant map
\begin{equation}\label{def:Psi}
\Psi:\Pi\to R, \qquad \pm\epsilon_i\pm\epsilon_j\mapsto\pm\epsilon_i\pm\epsilon_j; \quad \pm2\epsilon_i\mapsto\pm\epsilon_i.
\end{equation}

 For $\mathbf{v}\in\Lambda$, set $$\mathbb{V}_{\mathbf{v}}^+=\bigoplus\limits_{\lambda\in\Psi(\Pi^+)\backslash\Psi(\Pi_{\mathbf{v}}^+)}\mathbb{V}[\lambda],$$
which is a $P_{\mathbf{v}}$-module. When $P_\mathbf{v}=B$ (resp., $P_{\mathbf{v}}=P_i$), we simply write $\mathbb{V}^+=\mathbb{V}_{\mathbf{v}}^+$ (resp., $\mathbb{V}^+_i=\mathbb{V}_{\mathbf{v}}^+$).

\begin{lem}\label{lem int}
  If $w \in \D_{\mathbf{v}\mathbf{w}}$ and $P_{\mathbf{u}}\subseteq P_{\mathbf{w}}$ for some $\mathbf{u,v,w}\in\Lambda_{n,d}^\jmath$, then 
  $$w(\Pi^{+}_{\mathbf{w}})\cap\Pi^+=\emptyset\quad \text{and}\quad \mathbb{V}_{\mathbf{v}}^{+}\cap w(\mathbb{V}^{+}_{\mathbf{w}})=\mathbb{V}_{\mathbf{v}}^{+}\cap w(\mathbb{V}^{+}_{\mathbf{u}}).$$
\end{lem}
\begin{proof}
If $w(\Pi^{+}_{\mathbf{w}})\cap\Pi^+\neq\emptyset$, then there exists a simple root $\alpha\in\Delta_{\mathbf{w}}$ such that $w(\alpha)$ is a positive root. Denote $w'=w\cdot s_{\alpha}\in W_{\mathbf{v}}wW_{\mathbf{w}}$. It is clear that $w'(\alpha)$ is a negative root. Now $\ell(w')=|w'(\Pi^+)\cap\Pi^{-}|=|(w(\Pi^+)\cap\Pi^{-})\bigsqcup\{w'(\alpha)\}|=\ell(w)+1$, which contradicts the condition $w \in \D_{\mathbf{v}\mathbf{w}}$. Thus, $w(\Pi^{+}_{\mathbf{w}})\cap\Pi^+=\emptyset$

Since the map $\Psi$ defined in \eqref{def:Psi} is $W$-equivariant, $w(\Pi^{+}_{\mathbf{w}})\cap\Pi^+=\emptyset$ implies $w(\Psi(\Pi^{+}_{\mathbf{w}}))\cap\Psi(\Pi^+)=\emptyset$. Note that $\mathbb{V}^{+}_{\mathbf{u}}=\mathbb{V}^{+}_{\mathbf{w}}\oplus\bigoplus\limits_{\lambda\in\Psi(\Pi^+_\mathbf{w})}\mathbb{V}^{+}_{\mathbf{u}}[\lambda]$. By the previous arguments, $\mathbb{V}^{+}_{\mathbf{v}}\cap w(\bigoplus\limits_{\lambda\in\Psi(\Pi^{+}_{\mathbf{w}})}\mathbb{V}^{+}_{\mathbf{u}}[\lambda])=0$. Hence $\mathbb{V}^{+}_{\mathbf{v}}\cap w(\mathbb{V}^{+}_{\mathbf{u}})=(\mathbb{V}^{+}_{\mathbf{v}}\cap w(\mathbb{V}^{+}_{\mathbf{w}}))\oplus(\mathbb{V}^{+}_{\mathbf{v}}\cap w(\bigoplus\limits_{\lambda\in\Psi(\Pi^{+}_{\mathbf{w}})}\mathbb{V}^{+}_{\mathbf{u}}[\lambda]))=\mathbb{V}^{+}_{\mathbf{v}}\cap w(\mathbb{V}^{+}_{\mathbf{w}})$ as desired.
\end{proof}

%

%------------------------------------------
 \subsection{Generalized exotic Steinberg varieties}
%------------------------------------------
For $\mathbf{v}\in\Lambda$,
denote $$F_{\mathbf{v}}=G\times^{P_{\mathbf{v}}} \mathbb{V}_{\mathbf{v}}^+.$$
There are two canonical maps $$\begin{tikzcd}
                        & F_\mathbf{v} \arrow[ld, "\pi_\mathbf{v}"'] \arrow[rd, "\mu_{\mathbf{v}}"] &            \\
\mathfrak{F}_\mathbf{v} &                                                                           & \mathbb{V}
\end{tikzcd}$$ where $\pi_\mathbf{v}$ is the canonical projection and $\mu_\mathbf{v}$ is the moment map sending $(g,v)\mapsto g\cdot v$.
The {\em exotic nilpotent cone} $\mathcal{N}\subset\mathbb{V}$ is defined as the image of $\mu_{\mathbf{d}}$; recall that $\mathbf{d}$ is a regular $W$-orbit.
Set $$F_{\ff}=\bigsqcup\limits_{\mathbf{v}\in\Lambda_{\ff}}F_{\mathbf{v}}.$$
We introduce the {\em generalized exotic Steinberg varieties} as follows:  $$\mathbf{Z}_{\mathbf{v}\mathbf{w}}=F_{\mathbf{v}}\times_{\mathcal{N}}F_{\mathbf{w}}\quad \text{and}\quad \mathbf{Z}_{\ff}=F_{\ff}\times_{\mathcal{N}}F_{\ff}=\bigsqcup\limits_{\mathbf{v},\mathbf{w}\in\Lambda_{\ff}}\mathbf{Z}_{\mathbf{v}\mathbf{w}}.$$

Let $\pi_{\mathbf{v}\mathbf{w}}: \mathbf{Z}_{\mathbf{v}\mathbf{w}}\to\Fm_{\mathbf{v}}\times\Fm_{\mathbf{w}}$ be the canonical projection. 
Denote $$\mathbf{T}_{\Ob_{\mathbf{v},w,\mathbf{w}}}=\pi^{-1}_{\mathbf{v}\mathbf{w}}(\Ob_{\mathbf{v},w,\mathbf{w}}) \quad\mbox{and}\quad \mathbf{T}_{\mathbf{v},w,\mathbf{w}}=\overline{\pi^{-1}_{\mathbf{v}\mathbf{w}}(\Ob_{\mathbf{v},w,\mathbf{w}})}.$$

When we take $\mathbf{v}=\mathbf{w}=\mathbf{d}$ a regular $W$-orbit, the above setup goes back to the one in \cite{Ka09}. In this case, We simply denote $F=F_{\mathbf{d}}$, $\mathbf{Z}=\mathbf{Z}_{\mathbf{dd}}=F\times_{\mathcal{N}}F$ and $\pi=\pi_{\mathbf{dd}}:\mathbf{Z}\to \Bm\times\Bm$.
By abuse of notation, we shall always omit the orbit $\mathbf{d}$ if it occurs in the subscript, just like $F, \mathbf{Z}$ and $\pi$. For example, we write $\mathbf{T}_{\Ob_{\mathbf{v},w}}$ and $\mathbf{T}_{\mathbf{v},w}$ for $\mathbf{T}_{\Ob_{\mathbf{v},w,\mathbf{d}}}$ and $\mathbf{T}_{\mathbf{v},w,\mathbf{d}}$, respectively; write $\mathbf{T}_{\Ob_{w}}$ and $\mathbf{T}_{w}$ for $\mathbf{T}_{\Ob_{\mathbf{d},w,\mathbf{d}}}$ and $\mathbf{T}_{\mathbf{d},w,\mathbf{d}}$, respectively.

\begin{lem}
 The variety $\mathbf{Z}_{\mathbf{v}\mathbf{w}}$ consists of $|\D_{\mathbf{v}\mathbf{w}}|$ irreducible components $\mathbf{T}_{\mathbf{v},w,\mathbf{w}}$ $( w\in\D_{\mathbf{v}\mathbf{w}})$.
\end{lem}
\begin{proof}
The proof is similar to that of \cite[Lemma~1.5]{Ka09}.
\end{proof}

Let $M_{i}$ ($i=1,2,3$) be smooth varieties. For closed subvarieties $N\subseteq M_1\times M_2$ and $N'\subseteq M_2\times M_{3}$, denote their set-theoretic composition by $$N\circ N'=\left\{(m_1,m_3)\in M_1\times M_{3}~\middle| \exists \ m_2 \in M_2, \ \text{such that} \ (m_1,m_2,m_3) \in N \times_{M_2} N'\right\}.$$% which is also a locally closed subvariety in $M_1\times M_{3}$; see \cite[\S2.7.5]{CG97} for details. We note that although $N$ and $N'$ are supposed to be closed in {\em loc.cit.}, the definition remains valid for locally closed subvarieties.

Let $p_{12}$  (resp., $p_{23}$): $\Fm_\mathbf{v}\times\Fm_{\mathbf{w}}\times\Fm_{\mathbf{u}}\to\Fm_{\mathbf{v}}\times\Fm_{\mathbf{w}}$ (resp., $\Fm_{\mathbf{w}}\times\Fm_{\mathbf{u}}$) be the natural projection.
Denote $F_{\mathbf{v}\mathbf{w}\mathbf{u}}=F_{\mathbf{v}}\times F_{\mathbf{w}}\times F_{\mathbf{u}}$.
Let $pr_{12}:F_{\mathbf{v}\mathbf{w}\mathbf{u}}\to F_{\mathbf{v}}\times F_{\mathbf{w}}$ and $pr_{23}:F_{\mathbf{v}\mathbf{w}\mathbf{u}}\to F_{\mathbf{w}}\times F_{\mathbf{u}}$ be natural projections.
\begin{lem}\label{inter trans}
    Let $\mathbf{v},\mathbf{w},\mathbf{u}\in\Lambda_\ff$ and $w\in\D_{\mathbf{v}\mathbf{w}}$. If $P_\mathbf{u}\subseteq P_{\mathbf{w}}$, then $pr_{12}^{-1}(\mathbf{T}_{\Ob_{\mathbf{v},w,\mathbf{w}}})$ and $pr_{23}^{-1}(\mathbf{T}_{\mathbf{w},\id,\mathbf{u}})$ intersect transversely in $F_{\mathbf{v}\mathbf{w}\mathbf{u}}$. Moreover, the natural projection $$pr'_{13}:pr_{12}^{-1}(\mathbf{T}_{\Ob_{\mathbf{v},w,\mathbf{w}}})\cap pr_{23}^{-1}(\mathbf{T}_{\mathbf{w},\id,\mathbf{u}})\to\mathbf{T}_{\Ob_{\mathbf{v},w,\mathbf{w}}}\circ\mathbf{T}_{\mathbf{w},\id,\mathbf{u}}=\mathbf{T}_{\Ob_{\mathbf{v},w,\mathbf{u}}}$$ is an isomorphism.
\end{lem}

\begin{proof}
  Consider the natural $G$-equivariant map $$p_{13}:p_{12}^{-1}(\Ob_{\mathbf{v},w,\mathbf{w}})\cap p_{23}^{-1}(\Ob_{\mathbf{w},\id,\mathbf{u}})\to\Ob_{\mathbf{v},w,\mathbf{w}}\circ\Ob_{\mathbf{w},\id,\mathbf{u}}=\Ob_{\mathbf{v},w,\mathbf{u}}.$$
  We claim that it is an isomorphism.
  In fact, for $x=(P_\mathbf{v},wP_\mathbf{u})\in\Ob_{\mathbf{v},w,\mathbf{u}}$, the fiber $p_{13}^{-1}(x)=\{P_{\mathbf{v}}\}\times(P_\mathbf{v} wP_\mathbf{w}/P_\mathbf{w}\cap wP_{\mathbf{w}}/P_{\mathbf{w}})\times\{wP_\mathbf{u}\}=\{(P_\mathbf{v},wP_\mathbf{w},wP_\mathbf{u})\}$ is a single point. So the claim follows from the $G$-equivariantness.
  
By the claim, we find that $p_{12}^{-1}(\Ob_{\mathbf{v},w,\mathbf{w}})\cap p_{23}^{-1}(\Ob_{\mathbf{w},\id,\mathbf{u}})$ is a $G$-orbit in $\Fm_{\mathbf{v}}\times\Fm_{\mathbf{w}}\times\Fm_{\mathbf{u}}$ with $(P_\mathbf{v},wP_\mathbf{w},wP_\mathbf{u})$ as its representative.
  Noting that $pr_{12}^{-1}(\mathbf{T}_{\Ob_{\mathbf{v},w,\mathbf{w}}})\cap pr_{23}^{-1}(\mathbf{T}_{\mathbf{w},\id,\mathbf{u}})$ is a $G$-vector bundle over $p_{12}^{-1}(\Ob_{\mathbf{v},w,\mathbf{w}})\cap p_{23}^{-1}(\Ob_{\mathbf{w},\id,\mathbf{u}})$,
  we have \begin{align*}
      &pr_{12}^{-1}(\mathbf{T}_{\Ob_{\mathbf{v},w,\mathbf{w}}})\cap pr_{23}^{-1}(\mathbf{T}_{\mathbf{w},\id,\mathbf{u}})\\&=\{(gP_\mathbf{v},gwP_{\mathbf{w}},gwP_\mathbf{u},v)\mid g\in G, v\in g(\mathbb{V}_{\mathbf{v}}^{+}\cap w(\mathbb{V}^{+}_{\mathbf{w}})\cap w(\mathbb{V}^{+}_{\mathbf{u}}))\}\\&\overset{\text{\tiny Lemma~\ref{lem int}}}{=}\{(gP_\mathbf{v},gwP_{\mathbf{w}},gwP_\mathbf{u},v)\mid g\in G,v\in g(\mathbb{V}_{\mathbf{v}}^{+}\cap w(\mathbb{V}^{+}_{\mathbf{u}}))\},
  \end{align*} that forces $pr'_{13}$ to be an isomorphism.

  To show that $pr_{12}^{-1}(\mathbf{T}_{\Ob_{\mathbf{v},w,\mathbf{w}}})$ and $pr_{23}^{-1}(\mathbf{T}_{\mathbf{w},\id,\mathbf{u}})$ intersect transversely, it suffices to show that they intersect transversely at the point $z=(P_\mathbf{v},wP_\mathbf{w},wP_\mathbf{u},v)$. Compute directly for the following tangent spaces:
  \begin{align*}
  T_{z}(F_{\mathbf{v}\mathbf{w}\mathbf{u}})&=\mathfrak{g}/\mathfrak{p}_{\mathbf{v}}\oplus\mathfrak{g}/\mathfrak{p}^{w}_{\mathbf{w}}\oplus\mathfrak{g}/\mathfrak{p}^{w}_{\mathbf{u}}\oplus\mathbb{V}^{+}_{\mathbf{v}}\oplus w(\mathbb{V}^{+}_{\mathbf{w}})\oplus w(\mathbb{V}^{+}_{\mathbf{u}}),
\\
  T_{z}(pr_{12}^{-1}(\mathbf{T}_{\Ob_{\mathbf{v},w,\mathbf{w}}}))&=
\{(x+\mathfrak{p}_{\mathbf{v}},x+\mathfrak{p}^{w}_{\mathbf{w}},y+\mathfrak{p}^{w}_{\mathbf{u}},v,v,u)\mid x,y\in\mathfrak{g},v\in\mathbb{V}^{+}_{\mathbf{v}}\cap w(\mathbb{V}^{+}_{\mathbf{w}}),u\in w(\mathbb{V}^{+}_{\mathbf{u}})\},
\\
T_{z}(pr_{23}^{-1}(\mathbf{T}_{\mathbf{w},\id,\mathbf{u}}))&=\{(x+\mathfrak{p}_{\mathbf{v}},y+\mathfrak{p}^{w}_{\mathbf{w}},y+\mathfrak{p}^{w}_{\mathbf{u}},v,u,u)\mid x,y\in\mathfrak{g},v\in\mathbb{V}^{+}_{\mathbf{v}},u\in w(\mathbb{V}^{+}_{\mathbf{w}})\cap w(\mathbb{V}^{+}_{\mathbf{u}})\}.
  \end{align*}
  Since $\mathbb{V}^{+}_{\mathbf{v}}\cap w(\mathbb{V}^{+}_{\mathbf{w}})+w(\mathbb{V}^{+}_{\mathbf{w}})\cap w(\mathbb{V}^{+}_{\mathbf{u}})=w(\mathbb{V}^+_{\mathbf{w}})$, we have $$T_{z}(pr_{12}^{-1}(\mathbf{T}_{\Ob_{\mathbf{v},w,\mathbf{w}}}))+T_{z}(pr_{23}^{-1}(\mathbf{T}_{\mathbf{w},\id,\mathbf{u}}))=T_{z}(F_{\mathbf{v}\mathbf{w}\mathbf{u}}).$$ Hence, $pr_{12}^{-1}(\mathbf{T}_{\Ob_{\mathbf{v},w,\mathbf{w}}})$ and $pr_{23}^{-1}(\mathbf{T}_{\mathbf{w},\id,\mathbf{u}})$ intersect transversely at $z$. The proof is complete.
\end{proof}
\begin{lem}\label{intr-2}
    Let $w$ be the unique shortest element in $W_\mathbf{v}w$ for $\mathbf{v}\in\Lambda_{\ff}$. The two $pr_{12}^{-1}(\mathbf{T}_{\mathbf{v},\id})$ and $pr_{23}^{-1}(\mathbf{T}_{\Ob_w})$ intersect transversely. Moreover, $$pr'_{13}:pr_{12}^{-1}(\mathbf{T}_{\mathbf{v},\id})\cap pr_{23}^{-1}(\mathbf{T}_{\Ob_w})\to\mathbf{T}_{\mathbf{v},\id}\circ\mathbf{T}_{\Ob_w}=\mathbf{T}_{\Ob_{\mathbf{v},w}}$$ is an isomorphism.
\end{lem}
\begin{proof}
    By the proof of \cite[Proposition 3.6]{LXY26}, one sees that  $$p_{13}:p^{-1}_{12}(\Ob_{\mathbf{v},\id})\cap p^{-1}_{23}(\Ob_w)\to\Ob_{\mathbf{v},\id}\circ\Ob_w=\Ob_{\mathbf{v},w}$$ is a $G$-equivariant isomorphism. The lemma follows from an argument similar to that of Lemma \ref{inter trans}.
\end{proof}

\subsection{Affine Hecke and Schur algebras}
Let $q_0,q_1,q_2$ be indeterminates. Denote $\mathcal{A}=\mathbb{Z}[q_0^{\pm 1},q_1^{\pm 1},q_2^{\pm 1}]$.

The affine Hecke algebra $\widetilde{\mathbb{H}}$ of type $\widetilde{C}_{d}$ with three parameters is the $\mathcal{A}$-algebra generated by $T_i$ ($1\le i\le d$) and $e^{\lambda}$ ($\lambda\in Q$) subject to the following relations:

\noindent{\bf(Toric relations)}$$e^0=1, \quad e^{\lambda}\cdot e^{\mu}=e^{\lambda+\mu}  \quad \forall \lambda,\mu\in Q;$$
\noindent{\bf (Hecke relations)} 
$$(T_i+1)(T_i-q_2)=0 \quad \forall 1\le i< d, \quad \text{and} \quad (T_d+1)(T_d+q_0q_1)=0;$$
\noindent{\bf (braid relations)}\begin{align*}
    &T_iT_j=T_jT_i,\quad \text{if} \  |i-j|>1,\\
    &T_iT_{i+1}T_i=T_{i+1}T_iT_{i+1}, \quad \text{if} \  1\le i<d-1,\\
    &(T_{d-1}T_d)^2=(T_dT_{d-1})^2;
\end{align*}
\noindent{\bf (Bernstein-Lusztig relations)}
$$T_i e^{\lambda}-e^{s_i\lambda}T_i=\left\{\begin{array}{ll}
(1-q_2)\frac{e^{\lambda}-e^{s_i\lambda}}{e^{\alpha_i}-1}, &\text{if} \  i\neq d,\\
\frac{(1+q_0q_1)-(q_0+q_1)e^{\epsilon_d}}{e^{\alpha_d}-1}(e^{\lambda}-e^{s_i\lambda}), &\text{if} \  i=d.
\end{array}\right.$$
\begin{rem}\label{rem1}
    The affine Hecke algebra $\widetilde{\mathbb{H}}$ coincides with the one studied in \cite[\S 2.3]{FLLLWW20} by the following correspondence: our $q_0\leftrightarrow$ their $q_0$, our $q_1\leftrightarrow$ their $-q_1$, our $q_2\leftrightarrow$ their $q^2$, our $T_{i}\leftrightarrow$ their $-qT_{d-i}$ for $1\le i<d$, our $T_d\leftrightarrow$ their $-q_0T_0$.
\end{rem}

Suppose $w=s_{i_1}\cdots s_{i_k}$ is a reduced form of $w\in W$. Define $T_{w}=T_{i_1}\cdots T_{i_k}$ and $q_{w}=q_{s_{i_1}}\cdots q_{s_{i_k}}$, where $q_{s_i}=-q_2\ (1\le i<d)$ or $q_0q_1\ (i=d)$. It is known that both $T_w$ and $q_w$ are independent of the choice of a reduced form of $w$.

Let $\mathbb{H}$ be the $\mathcal{A}$-subalgebra of $\widetilde{\mathbb{H}}$ generated by $\{T_i\}_{1\le i\le d}$, which is just the Hecke algebra of finite type $C_d$ (with multi-parameter). Denote by $\mathcal{A}[Q]$ the $\mathcal{A}$-subalgebra of $\widetilde{\mathbb{H}}$ generated by $e^\lambda$ ($\lambda\in Q$).
Set
$$x_{\mathbf{v}}=\sum_{w\in W_{\mathbf{v}}}q_{w}^{-1}T_{w}\in\mathbb{H}\quad \text{and} \quad \widetilde{\mathbb{T}}_{\ff}=\bigoplus_{\mathbf{v}\in\Lambda_{\ff}}x_\mathbf{v}\widetilde{\mathbb{H}}.$$
The following lemma is standard.
\begin{lem}\label{rk-Tf}
    Each summand $x_\mathbf{v}\widetilde{\mathbb{H}}$ of $\widetilde{\mathbb{T}}_{\ff}$ is a free $\mathcal{A}[Q]$-module with a basis $\{x_\mathbf{v} T_{w}\mid w\in\D_{\mathbf{v}}\}$. Therefore, $\widetilde{\mathbb{T}}_{\ff}$ is a free $\mathcal{A}[Q]$-module of rank $\sum_{\mathbf{v}\in\Lambda_\ff}|\D_\mathbf{v}|$.
\end{lem}

\begin{lem}\label{rann}
 For $\alpha_i\in\Delta_\mathbf{v}$, we always have $x_{\mathbf{v}}T_{i}=-x_\mathbf{v}$. Moreover, the right annihilator $r_{\widetilde{\mathbb{H}}}(x_\mathbf{v})$ of $x_\mathbf{v}$ in $\widetilde{\mathbb{H}}$ is the right ideal $I_\mathbf{v}$ generated by $\{1+T_i\mid \alpha_i\in\Delta_\mathbf{v},i=1,\cdots,d\}$.
\end{lem}
\begin{proof}
The first assertion follows from \cite[Lemma 3.1]{FLLLWW20} and Remark \ref{rem1}. Let us prove the second assertion.
   Note that $\widetilde{\mathbb{H}}$ has an $\mathcal{A}$-basis $\{T_yT_we^{\lambda}\mid y\in W_\mathbf{v},w\in\D_\mathbf{v},\lambda\in X^*(T)\}$. Suppose $\tilde{h}\in r_{\widetilde{\mathbb{H}}}(x_\mathbf{v})$. Thanks to Lemma~\ref{rk-Tf}, we assume $\widetilde{h}=hT_we^\lambda$,
   $h\in\mathbb{H}_\mathbf{v}$, $w\in\D_\mathbf{v}$ and $x_\mathbf{v} h=0$, where $\mathbb{H}_\mathbf{v}$ is the $\mathcal{A}$-subalgebra generated by $T_i$ ($\alpha_i\in\Delta_\mathbf{v}$). For $h=\sum_{y\in W_\mathbf{v}} a_yT_y$ ($a_y\in\mathcal{A}$), set $L_h=\{y\in W_\mathbf{v}\mid \text{$\ell(y)$ is maximal with $a_y \neq 0$}\}$. Define the length of $h$ by $\ell(h)=\ell(y), y\in L_h$. Now we use induction on $\ell(h)$. For $\ell(h)=0$, there is nothing to prove. Suppose that it holds for all elements in $\widetilde{\mathbb{H}}_\mathbf{v}$ with length less than $\ell(h)$. For $y\in L_h$, let $s_{i_y}$ be a simple reflection in $W_\mathbf{v}$ such that $\ell(s_{i_y}y)=\ell(y)-1$.
   Consider $$h=\sum_{y\in W_\mathbf{v}\backslash L_h} a_yT_y+\sum_{y\in L_h}-a_yT_{s_{i_y}y}+\sum\limits_{y\in L_h}a_y(T_{i_y}T_{s_{i_y}y}+T_{s_{i_y}y})=h'+h'',$$ where $h'=\sum_{y\in W_\mathbf{v}\backslash L_h} a_yT_y+\sum_{y\in L_h}-a_yT_{s_{i_y}y}$ and $h''=\sum_{y\in L_h}a_y(T_{i_y}T_{s_{i_y}y}+T_{s_{i_y}y})$. Clearly, we have $x_\mathbf{v} h'=0$ and $h''\in I_\mathbf{v}$. Since $\ell(h')\le \ell(h)-1$, we have $h'\in I_\mathbf{v}$ by the induction assumption, which forces $h\in I_\mathbf{v}$. The proof is complete.
\end{proof}

The affine quantum Schur algebra $\widetilde{\mathbb{S}}_{\ff}$ of type $\widetilde{C}_d$ (associated with $\Lambda_\mathbf{f}$) with three parameters is an $\mathcal{A}$-algebra defined as $$\widetilde{\mathbb{S}}_{\ff}=\mathrm{End}_{\widetilde{\mathbb{H}}}(\widetilde{\mathbb{T}}_{\ff}).$$
\begin{example}\label{ex:Sch}
   If we take $\Lambda_\ff=\Lambda_{n,d}^\jmath$ (resp., $\Lambda_{n,d}^\imath$) as in \eqref{eq:lambdaj} (resp., \eqref{eq:lambdai}), then the affine quantum Schur algebra $\widetilde{\mathbb{S}}_{\ff}$ matches the Schur algebra $\mathbb{S}_{\mathfrak{n},d}^{\imath\jmath}$ (resp., $\mathbb{S}_{\eta,d}^{\imath\imath}$) defined in \cite[\S 4]{FLLLWW20}.
\end{example}

%%===================================
\section{Construction of affine quantum Schur algebras}\label{sec:conssch}
%====================================

In this section, we shall provide a realization of affine quantum Schur algebras $\mathbb{S}_\mathbf{f}$ via the equivariant K-theory. To facilitate reading, we have placed the foundational content on equivariant K-theory required for this paper in Appendix~\ref{Apped:B}.

%------------------------------------
\subsection{Representation rings and localization}
For an algebraic group $H$, let $R(H)=K^H(pt)$ denote the rational representation ring of $H$. 
There are natural ring isomorphisms for representation rings of $G$ and its maximal torus $T$ (see \cite[\S6.1]{CG97}):
$$R(T)\simeq\Z[Q] \quad\text{and}\quad R(G)\simeq R(T)^W.$$ The latter isomorphism is a `group' analog of the Chevalley restriction theorem. In fact, it can be further generalized to the following result, obtained in \cite[Theorems 1.2 \& 2.2]{St75}.
\begin{lem}\label{lem:PW}
Let $P$ be a parabolic subgroup of $G$ containing $T$, and let $W'$ be the Weyl group of the Levi part of $\mathrm{Lie}(P)$. Then $R(P)\simeq R(T)^{W'}$ is a free $R(G)$-module with rank $|W/W'|$.
\end{lem}
By \cite[Proposition 4.7]{BT65}, $(P^w_{\mathbf{v}\mathbf{w}})U_\mathbf{v}$ is a parabolic subgroup, where $U_\mathbf{v}$ is the unipotent radical of $P_\mathbf{v}$. The following result follows by above lemma.
\begin{cor}\label{P^w}
We have $R(P^w_{\mathbf{v}\mathbf{w}})\simeq R(T)^{W^w_{\mathbf{v}\mathbf{w}}}$ is a free $R(G)$-module with rank $|W/W^w_{\mathbf{v}\mathbf{w}}|$.
\end{cor}

In order to indicate the three parameters, we introduce an extended group $\breve{H}=H\times (\C^\times)^3$ for any algebraic group $H$ so that $$R(\breve{H})\simeq R(H)[q_0^{\pm1},q_1^{\pm1},q_2^{\pm1}],$$ where $q_i$ ($i=0,1,2$) means the tautological representation of $(i+1)$-th factor of $(\C^\times)^3$. In particular, we define a $\breve{G}$-action on $\mathbb{V}$ by $$(g,z_0,z_1,z_2)\cdot(v_0,v_1,v_2)\mapsto (z^{-1}_0gv_0,z^{-1}_1gv_1,z^{-1}_2gv_2),$$ where $v_0,v_1\in V$ and $v_2\in \wedge^2V$.

Let $R(\breve{G})_{loc}$ be the localization of $R(\breve{G})$ at $(0)$. For an $R(\breve{G})$-module $M$, denote $M_{loc}=M \otimes_{R(\breve{G})} R(\breve{G})_{loc}$.

Let $\Fm_{\mathbf{v}}\stackrel{p_\mathbf{v}}\longleftarrow\Ob_{\mathbf{v},w,\mathbf{w}}\stackrel{p_{\mathbf{w}}}\longrightarrow\Fm_{\mathbf{w}}$ be the two canonical projections. 
Denote $$\mathscr{R}_{\mathbf{v}\mathbf{w}}^w=R(\breve{P}_\mathbf{v})R(w\breve{P}_\mathbf{w} w^{-1}),$$
where $w\breve{P}_\mathbf{w} w^{-1}$ means $wP_\mathbf{w} w^{-1}\times(\C^\times)^3$.
Note that $R(\breve{P}_\mathbf{v})R(w\breve{P}_\mathbf{w} w^{-1})$ is just the subring of $R(\breve{P}_{\mathbf{v}\mathbf{w}}^w)$ generated by the elements $p^*_{\mathbf{v}}(\chi')$ and $p^*_{\mathbf{w}}(\chi'')$, where $\chi'\in R(\breve{P}_\mathbf{v})$
and $\chi''\in R(\breve{P}_\mathbf{w})$.
As the same as \cite[Lemma~5.19]{LXY26}, we have
\begin{align} \label{loc gal}
\mathscr{R}_{\mathbf{v}\mathbf{w},loc}^w=R(\breve{P}^w_{\mathbf{v}\mathbf{w}})_{loc}.
\end{align}

\begin{rem}
    It may be true that $\mathscr{R}_{\mathbf{v}\mathbf{w}}^w$ equals $R(\breve{P}^w_{\mathbf{v}\mathbf{w}})$. But we are not able to prove this stronger statement. Anyway, the result in the above lemma is enough for our usage.
\end{rem}
%------------------------------------

\subsection{A geometric filtration of $K^{\breve{G}}(\mathbf{Z}_{\mathbf{v}\mathbf{w}})$}

Let ``$\le$" be the Bruhat order on $W$. We fix a total order ``$\preceq$" on $W$ compatible with the Bruhat order. Usually, the choice of ``$\preceq$" is not unique.
For $w \in W$ and $\mathbf{v},\mathbf{w}\in \Lambda_{n,d}^\jmath$ or $\Lambda_{n,d}^\imath$, recall $w_{\mathbf{v}\mathbf{w}}$ is the unique longest element in $W_\mathbf{v} wW_\mathbf{w}$. Set
\begin{align*}
    &\mathbf{Z}_{\mathbf{v}\mathbf{w}}^{\le w}=\bigsqcup\limits_{y\le w_{\mathbf{v}\mathbf{w}},y\in\D_{\mathbf{v}\mathbf{w}}}\mathbf{T}_{\Ob_{\mathbf{v},y,\mathbf{w}}},\quad
    \mathbf{Z}_{\mathbf{v}\mathbf{w}}^{\preceq w}=\bigsqcup\limits_{y\preceq w_{\mathbf{v}\mathbf{w}},y\in\D_{\mathbf{v}\mathbf{w}}}\mathbf{T}_{\Ob_{\mathbf{v},y,\mathbf{w}}} \quad \text{and}\\
    &\mathbf{Z}_{\mathbf{v}\mathbf{w}}^{\prec w}=\bigsqcup\limits_{y\prec w_{\mathbf{v}\mathbf{w}},y\in\D_{\mathbf{v}\mathbf{w}}}\mathbf{T}_{\Ob_{\mathbf{v},y,\mathbf{w}}},
\end{align*}
where $w_{\mathbf{v}\mathbf{w}}$ is the unique longest element in $W_\mathbf{v}wW_\mathbf{w}$.
Note that for $w \in \D_{\mathbf{v}\mathbf{w}}$, $\mathbf{T}_{\Ob_{\mathbf{v},w,\mathbf{w}}}$ is an open $\breve{G}$-subvariety of both $\mathbf{Z}_{\mathbf{v}\mathbf{w}}^{\le w}$ and $\mathbf{Z}_{\mathbf{v}\mathbf{w}}^{\preceq w}$, and $\mathbf{Z}_{\mathbf{v}\mathbf{w}}^{\prec w}$ is a closed $\breve{G}$-subvariety of $\mathbf{Z}_{\mathbf{v}\mathbf{w}}^{\preceq w}$.
\begin{lem}\label{exsq. 1}
    For $w \in \D_{\mathbf{v}\mathbf{w}}$, we have the following commutative diagram whose rows are exact:
    \begin{equation}\label{diag:exsq}
        \begin{tikzcd}
            &                                                 & K^{\breve{G}}(\mathbf{Z}^{\le w}_{\mathbf{v}\mathbf{w}}) \arrow[d] \arrow[r] & {K^{\breve{G}}(\mathbf{T}_{\Ob_{\mathbf{v},w,\mathbf{w}}})} \arrow[r] \arrow[d, equal] & 0 \\
0 \arrow[r] & K^{\breve{G}}(\mathbf{Z}^{\prec w}_{\mathbf{v}\mathbf{w}}) \arrow[r] & K^{\breve{G}}(\mathbf{Z}^{\preceq w}_{\mathbf{v}\mathbf{w}}) \arrow[r]       & {K^{\breve{G}}(\mathbf{T}_{\Ob_{\mathbf{v},w,\mathbf{w}}})} \arrow[r]                                & 0
\end{tikzcd}.
    \end{equation}
     Moreover, we have a filtration
     \begin{equation}\label{filtra}
         K^{\breve{G}}(\mathbf{Z}_{\mathbf{v}\mathbf{w}})\supsetneq\cdots\supsetneq K^{\breve{G}}(\mathbf{Z}_{\mathbf{v}\mathbf{w}}^{\preceq w})\supsetneq K^{\breve{G}}(\mathbf{Z}^{\prec w}_{\mathbf{v}\mathbf{w}})\supsetneq\cdots \supsetneq K^{\breve{G}}(\mathbf{Z}^{\preceq\id}_{\mathbf{v}\mathbf{w}})=K^{\breve{G}}(\mathbf{T}_{\mathbf{v},\id,\mathbf{w}}).
     \end{equation}
\end{lem}
\begin{proof}
We only need to prove that the bottom row of \eqref{diag:exsq} is left exact, which would lead to the whole lemma.
    Note that $\mathbf{Z}_{\mathbf{v}\mathbf{w}}$ owns a $\breve{T}$-equivariant cellular fibration over a single point. We can fit $\mathbf{Z}^{\prec w}_{\mathbf{v}\mathbf{w}}$ and $\mathbf{Z}^{\preceq w}_{\mathbf{v}\mathbf{w}}$ into this cellular fibration. Hence we have an injective map $K^{\breve{T}}(\mathbf{Z}^{\prec w}_{\mathbf{v}\mathbf{w}})\to K^{\breve{T}}(\mathbf{Z}^{\preceq w}_{\mathbf{v}\mathbf{w}})$ by Lemma \ref{cfl}. Thanks to \cite[6.1.22(a)]{CG97}, we have a commutative diagram with an injective bottom map:
    $$\begin{tikzcd}
K^{\breve{T}}(\mathbf{Z}_{\mathbf{v}\mathbf{w}}^{\prec w}) \arrow[r] \arrow[d, "\simeq"] & K^{\breve{T}}(\mathbf{Z}_{\mathbf{v}\mathbf{w}}^{\preceq w}) \arrow[d, "\simeq"] \\
R(\breve{T})\otimes_{R(\breve{G})}K^{\breve{G}}(\mathbf{Z}_{\mathbf{v}\mathbf{w}}^{\prec w}) \arrow[r]     & R(\breve{T})\otimes_{R(\breve{G})}K^{\breve{G}}(\mathbf{Z}^{\preceq w}_{\mathbf{v}\mathbf{w}})
\end{tikzcd}.$$ Note that $R(\breve{T})$ is free over $R(\breve{G})$ of finite rank by \cite[Theorem~1.2\&1.3]{St75}. In particular, it is a faithfully flat $R(\breve{G})$-algebra. Hence, $K^{\breve{G}}(\mathbf{Z}^{\prec w}_{\mathbf{v}\mathbf{w}})\to K^{\breve{G}}(\mathbf{Z}^{\preceq w}_{\mathbf{v}\mathbf{w}})$ is also injective. The proof is complete.
\end{proof}

%------------------------------
\subsection{Basis of $K^{\breve{G}}(F_\mathbf{v}\times_{\mathcal{N}}F)$}
%------------------------------
Recall $\mathbf{d}\in\Lambda$ is a regular $W$-orbit. We have a canonical cellular $\breve{G}$-fibration
\begin{equation}\label{cfb}
    \begin{tikzcd}
F_\mathbf{v}\times_{\mathcal{N}}F \arrow[d] & \cdots \arrow[l, hook'] & \mathbf{Z}^{\preceq w}_{\mathbf{v}\mathbf{d}} \arrow[l, hook'] \arrow[lld] & \mathbf{Z}^{\prec w}_{\mathbf{v}\mathbf{d}} \arrow[l, hook'] \arrow[llld] & \cdots \arrow[l, hook'] & {\mathbf{T}_{\mathbf{v},\id }} \arrow[l, hook'] \arrow[llllld] \\
\Bm                                      &                         &                                                                   &                                                                  &                         &
\end{tikzcd}.
\end{equation}
The following proposition is obtained by the cellular fibration lemma (see Lemma~\ref{cfl}).
\begin{prop}\label{bas1}
 The equivariant K-group $K^{\breve{G}}(F_\mathbf{v}\times_{\mathcal{N}}F)$ is a free $R(\breve{T})$-module with a basis$$\{[\mathcal{O}_{\mathbf{T}_{\mathbf{v},w}}]\mid w\in\D_{\mathbf{v}}\}.$$
 In particular, its $R(\breve{T})$-rank is $|\D_{\mathbf{v}}|$.
\end{prop}

For each weight $\lambda\in \breve{T}^*$, let $e^\lambda\in R(\breve{T})\simeq K^{\breve{G}}(\mathfrak{B})$ be the virtual character of $\breve{T}$ associated with $\lambda$. The $R(\breve{T})$-actions on $K^{\breve{G}}(\mathbf{Z}_{\mathbf{v}\mathbf{d}})$ and $K^{\breve{G}}(\mathbf{Z})$, mentioned in the above two corollaries, are given by $e^\lambda\cdot-=\pi^*(e^\lambda)\otimes^{\mathbb{L}}-$, where $\pi$ is the cellular fibration \eqref{cfb} and $\otimes^{\mathbb{L}}$ is the derived tensor product.
We write 
\begin{equation*}%\label{notation:elambda}
e_{\mathbf{v},w}^\lambda:=e^\lambda \cdot [\mathcal{O}_{\mathbf{T}_{\mathbf{v},w}}]\in K^{\breve{G}}(\mathbf{Z}_{\mathbf{v} \mathbf{d}}), \quad\mbox{(particularly, $e^\lambda_{w}:=e^\lambda\cdot[\mathcal{O}_{\mathbf{T}_{w}}]\in K^{\breve{G}}(\mathbf{Z})$)}.
\end{equation*}

The equivariant K-group $K^{\breve{G}}(F_\mathbf{v}\times_{\mathcal{N}}F)$ carries a natural $K^{\breve{G}}(\mathbf{Z})$-module structure by means of the right convolution
$$*:K^{\breve{G}}(F_\mathbf{v}\times_{\mathcal{N}}F)\times K^{\breve{G}}(\mathbf{Z})\to K^{\breve{G}}(F_\mathbf{v}\times_{\mathcal{N}}F).$$ Note that $K^{\breve{G}}(\Delta F)\simeq R(\breve{T})$ is a subalgebra of $K^{\breve{G}}(\mathbf{Z})$, where $\Delta F$ is the image of the diagonal embedding $F\hookrightarrow F\times F$.
\begin{prop}\label{tm coin}
    The $R(\breve{T})$-module structure on $K^{\breve{G}}(F_\mathbf{v}\times_{\mathcal{N}}F)$ induced by $$*:K^{\breve{G}}(F_\mathbf{v}\times_{\mathcal{N}}F)\times K^{\breve{G}}(\Delta F)\to K^{\breve{G}}(F_\mathbf{v}\times_{\mathcal{N}}F)$$ coincides with the usual one under the natural isomorphism $K^{\breve{G}}(\Delta F)\simeq R(\breve{T})\simeq K^{\breve{G}}(\Bm)$.
\end{prop}
\begin{proof}
    The proof is similar to \cite[Lemma 5.9]{LXY26}
\end{proof}

\subsection{Basis of $K^{\breve{G}}(\mathbf{Z}_{\mathbf{v}\mathbf{w}})$}

We choose a split $R(\breve{G})$-homomorphism $$\iota_w: K^{\breve{G}}(\mathbf{T}_{\Ob_{\mathbf{v},w,\mathbf{w}}})\to K^{\breve{G}}(\mathbf{Z}_{\mathbf{v}\mathbf{w}}^{\preceq w})$$ that factors through $K^{\breve{G}}(Z^{\le w}_{\mathbf{v}\mathbf{w}})$ in \eqref{diag:exsq}. Let $\mathcal{B}(P^{w}_{\mathbf{v}\mathbf{w}})$ be an $R(\breve{G})$-basis of $K^{\breve{G}}(\mathbf{T}_{\Ob_{\mathbf{v},w,\mathbf{w}}})\simeq R(\breve{P}^{w}_{\mathbf{v}\mathbf{w}})$ (the choice of $\mathcal{B}(P^{w}_{\mathbf{v}\mathbf{w}})$ is very flexible). For $[\mathcal{O}_{\mathbf{T}_{\Ob_{\mathbf{v},w,\mathbf{w}}}}(\chi)]\in K^{\breve{G}}(\mathbf{T}_{\Ob_{\mathbf{v},w,\mathbf{w}}})$, set $\chi^{w}_{\mathbf{v}\mathbf{w}}=\iota_w([\mathcal{O}_{\mathbf{T}_{\Ob_{\mathbf{v},w,\mathbf{w}}}}(\chi)])$. Clearly, $\chi^w_{\mathbf{v}\mathbf{w}}$ can be represented by a complex on $\mathbf{Z}_{\mathbf{v}\mathbf{w}}$ with support on $\mathbf{Z}_{\mathbf{v}\mathbf{w}}^{\le w}$ by our construction.

\begin{prop}\label{bas2}
    The equivariant K-group $K^{\breve{G}}(\mathbf{Z}^{\preceq w}_{\mathbf{v}\mathbf{w}})$ is a free $R(\breve{G})$-module with a basis
    $$\{\chi^y_{\mathbf{v}\mathbf{w}}\mid \chi\in\mathcal{B}(P^y_{\mathbf{v}\mathbf{w}}),\ y\preceq w,\ y\in\D_{\mathbf{v}\mathbf{w}}\}.$$ In particular,
    $K^{\breve{G}}(\mathbf{Z}_{\mathbf{v}\mathbf{w}})$ is a free $R(\breve{G})$-module of rank $|\D_{\mathbf{v}}|\cdot|\D_{\mathbf{w}}|$ with a basis
    $$\{\chi^w_{\mathbf{v}\mathbf{w}}\mid \chi\in\mathcal{B}(P^w_{\mathbf{v}\mathbf{w}}),\ w\in\D_{\mathbf{v}\mathbf{w}}\}.$$
\end{prop}
\begin{proof}
By our construction, $$\mathrm{Span}_{R(\breve{G})}\{\chi^w_{\mathbf{v}\mathbf{w}}\mid\chi\in\mathcal{B}(P^w_{\mathbf{v}\mathbf{w}})\}\oplus K^{\breve{G}}(\mathbf{Z}^{\prec w}_{\mathbf{v}\mathbf{w}})=K^{\breve{G}}(\mathbf{Z}^{\preceq w}_{\mathbf{v}\mathbf{w}}).$$ 
Now, the first part of the lemma follows from the filtration \eqref{filtra} and induction on $w$.

In the following, we shall determine the $R(\breve{G})$-rank of $K^{\breve{G}}(\mathbf{Z}_{\mathbf{v}\mathbf{w}})$.
    Thanks to Corollary~\ref{P^w}, $$K^{\breve{G}}(\mathbf{T}_{\Ob_{\mathbf{v},w,\mathbf{w}}}) \simeq K^{\breve{G}}(\Ob_{\mathbf{v},w,\mathbf{w}}) \simeq R(\breve{P}_{\mathbf{v}\mathbf{w}}^w)\simeq R(\breve{T})^{W_{\mathbf{v}\mathbf{w}}^w}$$ is free over $R(\breve{G})$ with rank $|D_\mathbf{v}|\cdot|W_\mathbf{v} w W_\mathbf{w}/W_\mathbf{w}|$.
Therefore,
    \begin{align*} &\mathrm{rank}_{R(\breve{G})}K^{\breve{G}}(\mathbf{Z}_{\mathbf{v}\mathbf{w}})=\mathrm{rank}_{R(\breve{G})}(\bigoplus_{w \in \D_{\mathbf{v}\mathbf{w}}}K^{\breve{G}}(\mathbf{T}_{\Ob_{\mathbf{v},w,\mathbf{w}}}))=\mathrm{rank}_{R(\breve{G})}(\bigoplus_{w \in \D_{\mathbf{v}\mathbf{w}}}K^{\breve{G}}(\Ob_{\mathbf{v},w,\mathbf{w}}))\\&=\mathrm{rank}_{R(\breve{G})}(\bigoplus_{w \in \D_{\mathbf{v}\mathbf{w}}}K^{\breve{G}}(G/P^{w}_{\mathbf{v}\mathbf{w}}))=\mathrm{rank}_{R(\breve{G})}(\bigoplus_{w \in \D_{\mathbf{v}\mathbf{w}}} R(\breve{T})^{W_{\mathbf{v}\mathbf{w}}^w})\\
        &=\sum_{w \in \D_{\mathbf{v}\mathbf{w}}} |\D_\mathbf{v}|\cdot|W_\mathbf{v} w W_\mathbf{w}/W_\mathbf{w}|=|\D_\mathbf{v}|\cdot|\D_\mathbf{w}|.
    \end{align*}
    The proof is complete.
\end{proof}
 We have a direct consequence as follows.
\begin{cor}\label{rk-KZ_f}
    The equivariant K-group $K^{\breve{G}}(\mathbf{Z}_{\ff})$ is a free $R(\breve{G})$-module with rank $(\sum_{\mathbf{v}\in\Lambda_\ff}|\D_\mathbf{v}|)^2.$
\end{cor}

\subsection{Realization of affine Hecke algebras}
In this subsection, we recall the K-theoretic realization of affine Hecke algebras with three parameters in \cite{Ka09}.
Let $\C_{a,b,c,\lambda}$ be a one-dimensional module of $\breve{T}$ with weight $(a,b,c,\lambda)$, i.e., $$(t,z_0,z_1,z_2)x=z^a_0\cdot z^b_1\cdot z^c_2\cdot\lambda(t)\cdot x$$ for $(t,z_0,z_1,z_2)\in \breve{T}$ and $x\in\C_{a,b,c,\lambda}$ where $a,b,c\in\Z$. It is also a one-dimensional module of $\breve{B}$ if we trivially extend the action for the unipotent radical of $B$.

We let $\mathbf{e}^{\lambda}$ be the pullback of the line bundle $G\times^{B}\C_{0,0,0,-\lambda}$ to $\Delta F$.
For each $i=1,2,\cdots,d$, denote $\widetilde{T}_{i}=[\mathcal{O}_{\mathbf{T}_{s_i}}]$.

By abuse of notation, we also let $q_i$ be the tautological representation corresponding to the ($i+1$)-th factor in $(\C^\times)^3$ ($i=0,1,2$).
We have the following identification 
\begin{equation*}%\label{eq:KGB=AQ}
K^{\breve{G}}(\Bm)\simeq R(\breve{T})\simeq R(T)[q_0^{\pm 1},q_1^{\pm 1},q_2^{\pm 1}]\simeq\mathcal{A}[Q].
\end{equation*}

\begin{thm}(\cite[Theorem 2.8]{Ka09})\label{exo-H}
  There is an algebra isomorphism $\widetilde{\mathbb{H}}\stackrel{\simeq}{\longrightarrow}K^{\breve{G}}(\mathbf{Z})$ given by $$\vartheta:e^{\lambda}\mapsto \mathbf{e}^{\lambda},\quad T_i\mapsto\begin{cases}
\widetilde{T}_{i}-(1-q_2(\mathbf{e}^{\alpha_i}+1)),&\text{if} \  1\le i<d,\\
\widetilde{T}_{i}+(q_0+q_1)\mathbf{e}^{\epsilon_d}-(1+q_0q_1(\mathbf{e}^{\alpha_d}+1)), &\text{if} \  i=d.
\end{cases}$$
\end{thm}
\begin{rem}
Although \cite[Theorem 2.8]{Ka09} is stated after
complexification, the same construction yields an isomorphism
over $\mathcal A=\mathbb Z[q_0^{\pm1},q_1^{\pm1},q_2^{\pm1}]$.
Indeed, the explicit formulas for the basic
(anti-spherical) representation given in
\cite[Theorem 2.7]{Ka09}
preserve $\mathcal A[Q]$.
This representation is faithful over $\mathcal A$,
by its faithfulness after complexification and the PBW basis
of $\widetilde{\mathbb{H}}$.
On the geometric side, \cite[Theorem 2.5]{Ka09} gives a faithful
convolution action of $K^{\breve G}(\mathbf Z)$ on
$K^{\breve G}(F)\cong\mathcal A[Q]$.
The calculation in \cite[Proposition 2.6]{Ka09} identifies
the corrected geometric generators with the operators
of the basic representation.
Moreover, \cite[Theorem 2.4]{Ka09} gives integral generation
by the line-bundle classes and $\widetilde T_i$.
Since each corrected generator differs from $\widetilde T_i$
by an element of $\mathcal A[Q]$, the two faithful images
coincide, proving the asserted integral isomorphism.
\end{rem}
\begin{rem}
    The above isomorphism also fits the following commutative diagram
    $$\begin{tikzcd}
	{R(\breve{G})} & {K^{\breve{G}}(\Delta F)} & {K^{\breve{G}}(\mathbf{Z})} \\
	{Z(\widetilde{\mathbb{H}})} & {\mathcal{A}[Q]} & {\widetilde{\mathbb{H}}}
	\arrow[hook, from=1-1, to=1-2]
	\arrow["\simeq", from=1-1, to=2-1]
	\arrow[hook, from=1-2, to=1-3]
	\arrow["\simeq", from=1-2, to=2-2]
	\arrow["\simeq", from=1-3, to=2-3]
	\arrow[hook, from=2-1, to=2-2]
	\arrow[hook, from=2-2, to=2-3]
\end{tikzcd}$$ where  $Z(\widetilde{\mathbb{H}})$ is the center of $\widetilde{\mathbb{H}}$, and the first isomorphism $R(\breve(G))\simeq Z(\widetilde{\mathbb{H}})$ is due to Bernstein (cf. \cite[Theorem~2.9]{Ka09} and \cite[(4.2.10)]{Mc03}).
\end{rem}

\begin{lem}
  There is an algebra isomorphism  $\widetilde{\mathbb{H}}_{loc}\simeq\End_{R(\breve{G})_{loc}}(\mathcal{A}[Q]_{loc})$.
\end{lem}
\begin{proof}
    It is known (cf. \cite[Theorem~2.7]{Ka09} \& \cite[(4.3.10)]{Mc03}) that there is a faithful representation of $\widetilde{\mathbb{H}}$ on $\mathcal{A}[Q]$ over $\mathcal{A}$, by which we have the embedding $\widetilde{\mathbb{H}}\hookrightarrow\End_{\mathcal{A}}(\mathcal{A}[Q])$. Using Bernstein's theorem $R(\breve(G))\simeq Z(\widetilde{\mathbb{H}})$, we have $\widetilde{\mathbb{H}}\hookrightarrow\End_{R(\breve{G})}(\mathcal{A}[Q])$. Taking the localization, we get $\widetilde{\mathbb{H}}_{loc}\simeq\End_{R(\breve{G})_{loc}}(\mathcal{A}[Q]_{loc})$, since $\dim_{R(\breve{G})_{loc}}(\widetilde{\mathbb{H}}_{loc})=\dim_{R(\breve{G})_{loc}}(\End_{R(\breve{G})_{loc}}(\mathcal{A}[Q]_{loc})=|W|^2.$ The proof is complete.
\end{proof}
\begin{prop}\label{rk-Sloc}
   The localization $\widetilde{\mathbb{S}}_{\ff,loc}$ is a matrix algebra over $R(\breve{G})_{loc}$ with dimension $(\sum_{\mathbf{v}\in\Lambda_\ff}|\D_\mathbf{v}|)^2$. Moreover, the center of $\widetilde{\mathbb{S}}_\ff$ is $R(\breve{G})$.
\end{prop}
\begin{proof}
    By the above lemma, $\widetilde{\mathbb{H}}_{loc}$ is a matrix algebra over $R(\breve{G})_{loc}$. It has a unique simple (right) module $M$ of dimension $|W|$. So $\widetilde{\mathbb{T}}_{\ff,loc}\simeq M^{\oplus n}$ as right $\widetilde{\mathbb{H}}_{loc}$-modules, where $n=\sum_{\mathbf{v}\in\Lambda_\ff}|\D_\mathbf{v}|.$
    We have \begin{align*}
        \widetilde{\mathbb{S}}_{\ff,loc}&=(\End_{\widetilde{\mathbb{H}}}\widetilde{\mathbb{T}}_{\ff})_{loc}=\End_{\widetilde{\mathbb{H}}_{loc}}\widetilde{\mathbb{T}}_{\ff,loc}\simeq\End_{\widetilde{\mathbb{H}}}(M^{\oplus n})\\&=\End_{\widetilde{\mathbb{H}}_{loc}}(M)\otimes_{R(\breve{G})_{loc}}\Mat_{n}(R(\breve{G})_{loc})=\Mat_{n}(R(\breve{G})_{loc}).
    \end{align*}
Therefore, the center of $\widetilde{\mathbb{S}}_{\ff,loc}$ is $R(\breve{G})_{loc}$. Since $\widetilde{\mathbb{S}}_\ff$ generates the entire algebra $\widetilde{\mathbb{S}}_{\ff,loc}$, the centralizer of it in $\widetilde{\mathbb{S}}_{\ff,loc}$ is $R(\breve{G})_{loc}$. Hence, the center of $\widetilde{\mathbb{S}}_{\ff}$ is $R(\breve{G})_{loc}\cap\widetilde{\mathbb{S}}_{\ff}=R(\breve{G})$.
\end{proof}

\subsection{Distinguished open subvarieties}
For $1\leq i\leq d$, let $P^{-}_{i}=\overline{B^{-}s_iB^{-}}$ be the parabolic subgroup opposite to $P_{i}$ and $P^-_\mathbf{v}$ be the parabolic subgroup opposite to $P_\mathbf{v}$. Let $U^-_i$ (resp., $U^-_{\mathbf{v}}$) be unipotent radical of $P^-_i$ (resp., $P^-_\mathbf{v}$). Let $L_{i}=P^{-}_{i}\cap P_{i}$ be the standard Levi factor. Consider $P^{-}_{i}B/B\simeq U^{-}_{i}\times P_{i}/B$, which is an open subvariety of $\Bm$. Since the canonical projection $p:\mathfrak{B}\to\mathfrak{F}_\mathbf{v}$ is faithfully flat, $P^-_iP_\mathbf{v}/P_\mathbf{v}=p(P^-_iB/B)$ is an open subvariety of $\mathfrak{F}_\mathbf{v}$.
Denote $\mathbb{V}^{+}_{i}=\mathbb{V}^{+}\cap s_{i}(\mathbb{V}^{+})$ and $B_i=L_{i}\cap B$. 
\begin{lem}\label{lem:pi=ULV}
    We have an $\breve{L}_i$-equivariant isomorphism of varieties
    $$\pi^{-1}(P^{-}_{i}B/B)\simeq U^{-}_{i}\times (L_{i}\times^{B_i}\mathbb{V}^{+}/\mathbb{V}^{+}_{i})\times\mathbb{V}^{+}_{i}.$$
\end{lem}
 \begin{proof}
     We have $\breve{L}_i$-equivariant isomorphisms of varieties: $$\pi^{-1}(P^{-}_{i}B/B)\simeq U^{-}_{i}\times P_{i}\times^{B}\mathbb{V}^{+}\simeq U^-_i\times L_i\times^{B_i}\mathbb{V}^+.$$ Since $\mathbb{V}^+\simeq\mathbb{V}^+/\mathbb{V}^+_i\oplus\mathbb{V}_i^+$ as $B_i$-modules, we have an $\breve{L}_i$-equivariant isomorphism of vector bundles over $L_i/B_i$: $$L_i\times^{B_i}\mathbb{V}^+\simeq L_i\times^{B_i}(\mathbb{V}^+/\mathbb{V}^+_i \oplus \mathbb{V}_i^+).$$
     Noting that $\mathbb{V}^{+}_{i}$ is a $L_{i}$-module, we have $L_{i}\times^{B_i}\mathbb{V}^{+}_{i}$ is a trivial vector bundle, i.e., $L_{i}\times^{B_i}\mathbb{V}^{+}_{i}\simeq L_{i}/B\times\mathbb{V}^{+}_{i}$. Hence, as $\breve{L}_i$-varieties, $L_{i}\times^{B_i}\mathbb{V}^{+}\simeq (L_{i}\times^{B_i}\mathbb{V}^{+}/\mathbb{V}^{+}_{i})\times\mathbb{V}^{+}_{i}$. The lemma follows.
 \end{proof}

\begin{lem}
    For $\alpha_i\in\Delta_{\mathbf{v}}$, we have an $\breve{L}_i$-equivariant isomorphism $$\pi_{\mathbf{v}}^{-1}(P^{-}_{i}P_{\mathbf{v}}/P_{\mathbf{v}})\simeq U^{-}_{\mathbf{v}}\times\mathbb{V}^{+}_{\mathbf{v}}.$$
\end{lem}
\begin{proof}
As $\breve{L}_i$-varieties, we have $\pi^{-1}_{\mathbf{v}}(P^{-}_iP_\mathbf{v})/P_\mathbf{v}\simeq U^{-}_\mathbf{v}\times P_\mathbf{v}\times^{P_\mathbf{v}}\mathbb{V}^+_\mathbf{v}\simeq U^{-}_\mathbf{v}\times\mathbb{V}^+_\mathbf{v}$.
\end{proof}
Let $M_1=P^{-}_{i}P_{\mathbf{v}}/P_{\mathbf{v}}$, $M_2=M_3=P^{-}_{i}B/B$ and $M_{ij}=M_{i}\times M_{j}$ ($1\leq i,j\leq3$). Let $Z_{12}=Z_{13}=\Ob_{\mathbf{v},\id }\cap M_{12}$, $Z_{23}=\overline{\Ob}_{s_i}\cap M_{23}$ and $\widetilde{Z}_{12}=\widetilde{Z}_{13}=\pi^{-1}_{\mathbf{v},\id}(Z_{12})=\pi^{-1}_{\mathbf{v},\id}(Z_{13})$, $\widetilde{Z}_{23}=\pi^{-1}(Z_{23})$.
Let $f_1:U^{-}_{i}\to U^{-}_{\mathbf{v}}$ be the canonical projection, and $f_2: \mathbb{V}^{+}_{\mathbf{v}}\to\mathbb{V}^{+}$ the canonical embedding.
\begin{lem}\label{L_i decomp}
   We have the following $\breve{L}_i$-equivariant isomorphisms of varieties:
$$\widetilde{Z}_{12}=\widetilde{Z}_{13}\simeq\mathrm{graph}(f_1)\times pt\times L_{i}/B_{i}\times\mathrm{graph}(f_2),$$
   $${\widetilde{Z}_{23}}\simeq\Delta U^{-}_{i}\times (L_{i}/B_{i}\times L_{i}/B_{i}\cup\Delta(L_i\times^{B_i}{\mathbb{V}^+/\mathbb{V}^+_i}))\times\Delta\mathbb{V}_{i}^{+}.$$
   Moreover, the Zariski closure of $\pi^{-1}(\Ob_{s_i}\cap M_{23})$ in $\widetilde{Z}_{23}$ is isomorphic to $$\Delta U^{-}_{i}\times (L_{i}/B_{i}\times L_{i}/B_{i})\times\Delta\mathbb{V}_{i}^{+}.$$
\end{lem}
  \begin{proof}
 Firstly, we have \begin{align*}
\widetilde{Z}_{12}&\simeq(P^{-}_{i}\times B)\times^{B}\mathrm{graph}(f_2)\simeq\mathrm{graph}(f_1)\times P_{i}\times^{B}\mathrm{graph}(f_2)\\&\simeq\mathrm{graph}(f_1)\times P_{i}\cdot(P_{\mathbf{v}},B)\times\mathrm{graph}(f_2)\simeq\mathrm{graph}(f_1)\times pt\times L_{i}/B_{i}\times\mathrm{graph}(f_2).
  \end{align*}
Secondly,
     $\pi^{-1}(\Ob_{s_i}\cap M_{23})\simeq (P^{-}_{i}B)\times^{B\cap s_{i}B}\Delta\mathbb{V}_{i}^{+}\simeq \Delta U^{-}_{i}\times P_{i}\times^{B\cap s_{i}B}\Delta\mathbb{V}^{+}_{i}$. Noting that $\mathbb{V}^{+}_{i}$ is a $P_{i}$-module, we have
     \begin{align*}
        P_{i}\times^{B\cap s_{i}B}\Delta\mathbb{V}^{+}_{i}&\simeq P_{i}/(B\cap s_{i}B)\times\Delta\mathbb{V}^{+}_{i}\simeq L_{i}\cdot (B_{i},s_{i}B_{i})\times\Delta\mathbb{V}^{+}_{i}\\&=(L_{i}/B_{i}\times L_{i}/B_{i}\backslash\Delta(L_{i}/B_{i}))\times\Delta\mathbb{V}^{+}_{i}.
     \end{align*}
      Hence $\pi^{-1}(\Ob_{s_i}\cap M_{23})\simeq\Delta U^{-}_{i}\times (L_{i}/B_{i}\times L_{i}/B_{i}\backslash \Delta (L_i/B_i))\times\Delta\mathbb{V}_{i}^{+}$. Also, $$\pi^{-1}(\Ob_\id\cap M_{23})=\Delta(P^{-}_iB/B)\simeq \Delta U^{-}_{i}\times\Delta(L_i\times^{B_i}{\mathbb{V}^+/\mathbb{V}^+_i}))\times\Delta\mathbb{V}_{i}^{+}.$$
     Hence $$\widetilde{Z}_{23}=\pi^{-1}(\Ob_{s_i}\cap Z_{23})\cup\pi^{-1}(\Ob_{\id}\cap Z_{23})\simeq\Delta U^{-}_{i}\times (L_{i}/B_{i}\times L_{i}/B_{i}\cup \Delta(L_i\times^{B_i}{\mathbb{V}^+/\mathbb{V}^+_i}))\times\Delta\mathbb{V}_{i}^{+}.$$ Combining the above, the last assertion is clear.
 \end{proof}

 \begin{cor}\label{restriction}  As $\breve{L}_i$-varieties, $\mathbf{T}_{s_i}\cap\widetilde{Z}_{23}\simeq\Delta U^{-}_{i}\times (L_{i}/B_{i}\times L_{i}/B_{i})\times\Delta\mathbb{V}_{i}^{+}$.
 \end{cor}
 \begin{proof}
     Since $\widetilde{Z}_{23}$ is an open subvariety of $\mathbf{T}_i$, we have
     \begin{align*}      \mathbf{T}_{s_i}\cap\widetilde{Z}_{23}=\overline{\pi^{-1}(\Ob_{s_i}\cap M_{23})}\simeq\Delta U^{-}_{i}\times (L_{i}/B_{i}\times L_{i}/B_{i})\times\Delta\mathbb{V}_{i}^{+},
     \end{align*} where the second and third closures are taken in $\widetilde{Z}_{23}$.
 \end{proof}

 Recall the ambient spaces of $\widetilde{Z}_{12}$ and $\widetilde{Z}_{23}$ are $$\pi_{\mathbf{v}}^{-1}(M_1)\times\pi^{-1}(M_2)\simeq U^{-}_{\mathbf{v}}\times U^{-}_{i}\times pt\times(L_i\times^{B_i}\mathbb{V}^+/\mathbb{V}^+_i)\times \mathbb{V}^+_\mathbf{v}\times\mathbb{V}^+_i$$ and $$\pi^{-1}(M_2)\times\pi^{-1}(M_3)\simeq U^{-}_i\times U^{-}_i\times(L_i\times^{B_i}\mathbb{V}^+/\mathbb{V}^+_i)\times(L_i\times^{B_i}\mathbb{V}^+/\mathbb{V}^+_i)\times\mathbb{V}^+_i\times\mathbb{V}^+_i,$$ respectively.
\begin{cor}
We have $\widetilde{Z}_{12}\circ\widetilde{Z}_{23}=\widetilde{Z}_{13}$. As a consequence, there is a well-defined convolution product
$$K^{\breve{L}_i}(\widetilde{Z}_{12})\times K^{\breve{L}_i}(\widetilde{Z}_{23})\stackrel{*}{\longrightarrow}K^{\breve{L}_i}(\widetilde{Z}_{13}).$$
\end{cor}
\begin{proof}
Set 
\begin{align*}%\label{def:M}
\mathcal{M}=(L_i/B_i\times L_i/B_i)\cup\Delta(L_i\times^{B_i}\mathbb{V}^{+}/\mathbb{V}^+_i).
\end{align*}
    By Lemma \ref{L_i decomp} and using the set-theoretic K\"unneth formula for convolution product, we have \begin{align*}
\widetilde{Z}_{12}\circ\widetilde{Z}_{23}&\simeq(\mathrm{graph}(f_1)\times pt\times L_{i}/B_{i}\times\mathrm{graph}(f_2))\circ(\Delta U^{-}_{i}\times \mathcal{M}\times\Delta\mathbb{V}_{i}^{+})\\&\simeq (\mathrm{graph}(f_1)\circ\Delta U^{-}_{i})\times((pt\times L_{i}/B_i)\circ \mathcal{M})\times(\mathrm{graph}(f_2)\circ\Delta\mathbb{V}^{+}_{i})\\&=\mathrm{graph}(f_1)\times pt\times L_{i}/B_{i}\times\mathrm{graph}(f_2)\simeq\widetilde{Z}_{13}.
    \end{align*} The corollary follows.
\end{proof}

%----------------------------------
\subsection{Key multiplication formulas}
%----------------------------------
There is a natural map $\Phi:K^{\breve{G}}(\mathbf{T}_{\mathbf{v},\id})\to K^{\breve{L}_i}(\widetilde{Z}_{13})$, which is the composition of the forgetful map $K^{\breve{G}}(\mathbf{T}_{\mathbf{v},\id})\to K^{\breve{L}_i}(\mathbf{T}_{\mathbf{v},\id})$ and the restriction map $K^{\breve{L}_i}(\mathbf{T}_{\mathbf{v},\id})\to K^{\breve{L}_i}(\widetilde{Z}_{13})$.

\begin{lem}\label{Phi}
   The map $\Phi$ is an isomorphism of $R(\breve{G})$-modules.
\end{lem}

\begin{proof}
Noting that $Z_{13}\simeq P^{-}_{i}B\cdot (P_{\mathbf{v}},B)\simeq U^{-}_{i}\times P_{i}\cdot(P_{\mathbf{v}},B)\simeq U^{-}_{i}\times pt\times L_{i}/B_{i}$, the third projection $Z_{13}\to L_{i}/B_{i}$ is an affine bundle.
We have the following commutative diagram
$$
\begin{tikzcd}
{K^{\breve{G}}(\mathbf{T}_{\mathbf{v},\id})} \arrow[dd,"\Phi"] &  & {K^{\breve{G}}(\Ob_{\mathbf{v},\id })\simeq K^{\breve{G}}(\Bm)} \arrow[d, "f"] \arrow[ll, "Thom"'] \arrow[r, "\simeq"] & R(\breve{T}) \arrow[dd, "\simeq"]                                                  \\
                                                            &  & {K^{\breve{L}_i}(\Ob_{\mathbf{v},\id })} \arrow[d, "Res"]                                &                                                                                                      \\
K^{\breve{L}_i}(\widetilde{Z}_{13})               &  & K^{\breve{L}_i}(Z_{13}) \arrow[ll, "Thom"']                                         & K^{\breve{L}_i}(L_{i}/B_{i})\simeq R(\breve{T}) \arrow[l, "Thom"']
\end{tikzcd}$$
where $Thom$ are Thom isomorphisms, $f$ is the forgetful map, and $Res$ is the restriction map. The lemma follows.
\end{proof}
For a principal $B$-bundle (resp., $B_i$-bundle) $M\to M/B$ (resp., $M\to M/B_i$), let $\mathcal{V}\to M/B$ (resp., $\mathcal{V}\to M/B_i$) be a vector bundle. 
Let $\C_{\lambda}$ be the one-dimensional $T$-module with weight $\lambda\in Q$.
Denote the pullback of the sheaf of regular sections of $M\times^{B}\C_{\lambda}$ (resp., $M\times^{B_i}\C_{\lambda}$) to $\mathcal{V}$ by $\mathcal{O}_{\mathcal{V}}(\lambda)$.

Recall $\widetilde{T}_{i}=[\mathcal{O}_{\mathbf{T}_{s_i}}]\in K^{\breve{G}}(\mathbf{Z})$. Moreover, we denote 
\begin{equation}\label{eq:evlambda}
e_{\mathbf{v}}^\lambda:=[\mathcal{O}_{\mathbf{T}_{\mathbf{v},\id}}(-\lambda)]\in K^{\breve{G}}(\mathbf{T}_{\mathbf{v},\id}).
\end{equation}
\begin{prop}\label{formula2}
For $e_{\mathbf{v}}^{\lambda}\in K^{\breve{G}}(\mathbf{T}_{\mathbf{v},\id})$ and $\alpha_i\in\Delta_\mathbf{v}$, we have
    $$e_{\mathbf{v}}^{\lambda}*\widetilde{T}_{i}=\begin{cases}
\frac{e_{\mathbf{v}}^{\lambda}-e_{\mathbf{v}}^{s_{i}(\lambda)-\alpha_i}}{1-e_{\mathbf{v}}^{-\alpha_i}}-q_2\frac{e_{\mathbf{v}}^{\lambda+\alpha_i}-e_{\mathbf{v}}^{s_{i}(\lambda)-2\alpha_i}}{1-e_{\mathbf{v}}^{-\alpha_i}}, & \text{if} \  i\neq d, \\
\frac{e_{\mathbf{v}}^{\lambda}-e_{\mathbf{v}}^{s_{d}(\lambda)-\alpha_d}}{1-e_{\mathbf{v}}^{-\alpha_d}}-(q_0+q_1)\frac{e_{\mathbf{v}}^{\lambda+\epsilon_d}-e_{\mathbf{v}}^{s_{d}(\lambda)-\epsilon_d-\alpha_d}}{1-e_{\mathbf{v}}^{-\alpha_d}}+q_0q_1\frac{e_{\mathbf{v}}^{\lambda+\alpha_d}-e_{\mathbf{v}}^{s_{d}(\lambda)-2\alpha_d}}{1-e_{\mathbf{v}}^{-\alpha_d}}, & \text{if} \  i=d.
\end{cases}$$
\end{prop}
\begin{proof}
     We have a diagram as follows:
   \begin{equation}\label{com diag}
  \begin{tikzcd}
{K^{\breve{G}}(\mathbf{T}_{\mathbf{v},\id})\times K^{\breve{G}}(\mathbf{Z}^{\le s_i})} \arrow[rr, "*"] \arrow[d] &  & {K^{\breve{G}}(\mathbf{T}_{\mathbf{v},\id})} \arrow[d, "\simeq"] \\
K^{\breve{L}_i}(\widetilde{Z}_{12})\times K^{\breve{L}_i}(\widetilde{Z}_{23}) \arrow[rr, "*"]      &  & K^{\breve{L}_i}(\widetilde{Z}_{13}) .
\end{tikzcd}\end{equation}
The left vertical arrow is a composition of the forgetful map $$K^{\breve{G}}(\mathbf{T}_{\mathbf{v},\id})\times K^{\breve{G}}(\mathbf{Z}^{\le s_i})\longrightarrow K^{\breve{L}_i}(\mathbf{T}_{\mathbf{v},\id})\times K^{\breve{L}_i}(\mathbf{Z}^{\le s_i})$$ and the restriction map $$K^{\breve{L}_i}(\mathbf{T}_{\mathbf{v},\id})\times K^{\breve{L}_i}(\mathbf{Z}^{\le s_i})\longrightarrow K^{\breve{L}_i}(\widetilde{Z}_{12})\times K^{\breve{L}_i}(\widetilde{Z}_{23}).     
    $$
The right vertical map is due to Lemma~\ref{Phi}. The above diagram is commutative, due to the exactness of the forgetful functor and the flat base change. Then we can use the bottom arrow of the diagram to calculate the convolution product.

Thanks to Lemma \ref{L_i decomp}, we have two K\"unneth isomorphisms by external tensor product:
$$K^{\breve{L}_i}(\mathrm{graph}(f_1))\otimes_{R(\breve{L}_i)} K^{\breve{L}_i}(pt\times L_{i}/B_{i})\otimes_{R(\breve{L}_i)} K^{\breve{L}_i}(\mathrm{graph}(f_2))\to K^{\breve{L}_i}(\widetilde{Z}_{12})=K^{\breve{L}_i}(\widetilde{Z}_{13}),$$
$$K^{\breve{L}_i}(\Delta U^{-}_{i})\otimes_{R(\breve{L}_i)} K^{\breve{L}_i}(\mathcal{M})\otimes_{R(\breve{L}_i)}K^{\breve{L}_i}(\Delta\mathbb{V}^{+}_{i})\to K^{\breve{L}_i}(\widetilde{Z}_{23})$$ via $(\mathcal{L}_1,\mathcal{L}_2,\mathcal{L}_3)\mapsto\mathcal{L}_1\boxtimes\mathcal{L}_2\boxtimes\mathcal{L}_3$ and $(\mathcal{L}'_1,\mathcal{L}'_2,\mathcal{L}'_3)\mapsto\mathcal{L}'_1\boxtimes\mathcal{L}'_2\boxtimes\mathcal{L}'_3$, respectively. By a K-theoretic version K\"unneth formula for convolution product, we have $$(\mathcal{L}_1\boxtimes\mathcal{L}_2\boxtimes\mathcal{L}_{3})*(\mathcal{L}'_1\boxtimes\mathcal{L}'_2\boxtimes\mathcal{L}'_{3})=(\mathcal{L}_1*\mathcal{L}'_1)\boxtimes(\mathcal{L}_2*\mathcal{L}'_2)\boxtimes(\mathcal{L}_{3}*\mathcal{L}'_{3}).$$ 
Here we recall that the ambient spaces of $\mathrm{graph}(f_1),\ pt\times L_i/B_i,\ \mathrm{graph}(f_2),\ \Delta U^-_i,\ M,\ \Delta\mathbb{V}^+_i$ are $U^-_\mathbf{v}\times U^-_i,\ pt\times\widetilde{\mathbb{V}}_i,\ \mathbb{V}^+_\mathbf{v}\times\mathbb{V}^+_i,\ U^-_i\times U^-_i,\ \widetilde{\mathbb{V}}_i\times\widetilde{\mathbb{V}}_i,\ \mathbb{V}^+_i\times\mathbb{V}^+_i$, respectively, where we simply denote $L_{i}\times^{B_{i}}\mathbb{V}^{+}/\mathbb{V}^{+}_{i}$ by $\widetilde{\mathbb{V}}_{i}$.

It follows from Lemma~\ref{L_i decomp} and Corollary~\ref{restriction} that 
\begin{align*}
    e_{\mathbf{v}}^{\lambda}\mid_{\widetilde{Z}_{12}}&=[\mathcal{O}_{\mathrm{graph}(f_1)}\boxtimes\mathcal{O}_{L_i/B_i}(-\lambda)\boxtimes\mathcal{O}_{\mathrm{graph}(f_2)}], \\
    \widetilde{T}_i\mid_{\widetilde{Z}_{23}}&=[\mathcal{O}_{\Delta U^{-}_{i}}\boxtimes(\mathcal{O}_{L_i/B_i}\boxtimes\mathcal{O}_{L_i/B_i})\boxtimes\mathcal{O}_{\Delta\mathbb{V}^{+}_{i}}].
\end{align*}
Using commutative diagram \eqref{com diag}, we directly compute \begin{align*}
    &e_{\mathbf{v}}^{\lambda}*\widetilde{T}_{i}=e_{\mathbf{v}}^{\lambda}\mid_{\widetilde{Z}_{12}}*\widetilde{T}_{i}\mid_{\widetilde{Z}_{23}}\\
    &=([\mathcal{O}_{\mathrm{graph}(f_1)}\boxtimes\mathcal{O}_{L_i/B_i}(-\lambda)\boxtimes\mathcal{O}_{\mathrm{graph}(f_2)}])*([\mathcal{O}_{\Delta U^{-}_{i}}\boxtimes(\mathcal{O}_{L_i/B_i}\boxtimes\mathcal{O}_{L_i/B_i})\boxtimes\mathcal{O}_{\Delta\mathbb{V}^{+}_{i}}])\\
    &=([\mathcal{O}_{\mathrm{graph}(f_1)}]*[\mathcal{O}_{\Delta U^{-}_{i}}])\boxtimes([\mathcal{O}_{L_i/B_i}(-\lambda)]*[(\mathcal{O}_{L_i/B_i}\boxtimes\mathcal{O}_{L_i/B_i}))]\boxtimes([\mathcal{O}_{\mathrm{graph}(f_2)}]*[\mathcal{O}_{\Delta\mathbb{V}^{+}_{i}}])\\
    &=[\mathcal{O}_{\mathrm{graph}(f_1)}]\boxtimes p_{2*}([\mathcal{O}_{L_i/B_i}(-\lambda)\boxtimes\mathcal{O}_{\widetilde{\mathbb{V}}_{i}}] \otimes_{\widetilde{\mathbb{V}}_{i} \times \widetilde{\mathbb{V}}_{i}} [\mathcal{O}_{L_i/B_i}\boxtimes\mathcal{O}_{L_i/B_i}])\boxtimes[\mathcal{O}_{\mathrm{graph}(f_2)}]\\
    &=[\mathcal{O}_{\mathrm{graph}(f_1)}]\boxtimes (R\Gamma([\mathcal{O}_{L_i/B_i}(-\lambda)] \otimes_{\widetilde{\mathbb{V}}_{i}}[\mathcal{O}_{L_i/B_i}])\cdot[\mathcal{O}_{L_i/B_i}])\boxtimes[\mathcal{O}_{\mathrm{graph}(f_2)}].
\end{align*}

Let us compute $R\Gamma([\mathcal{O}_{L_i/B_i}(-\lambda)] \otimes_{\widetilde{\mathbb{V}}_{i}} [\mathcal{O}_{L_i/B_i}])$. Noting that
$$L_{i}\times^{B_{i}}(\mathbb{V}^{+}/\mathbb{V}^{+}_{i})^{\vee}=
\begin{cases}
    L_{i}\times^{B_{i}}\C_{0,0,1,-\alpha_i}, &\text{if} \  i\neq d,\\
    L_{d}\times^{B_d}(\C_{1,0,0,-\epsilon_d}\oplus \C_{0,1,0,-\epsilon_d}), &\text{if} \  i=d,
\end{cases}$$
we have $A'$-equivariant Koszul complexes $$0\longrightarrow q_{2}\mathcal{O}_{\widetilde{\mathbb{V}}_{i}}(-\alpha_{i})\longrightarrow\mathcal{O}_{\widetilde{\mathbb{V}}_{i}}\longrightarrow\mathcal{O}_{L_{i}/B_{i}}\longrightarrow0\quad (i\neq d),$$
$$0\longrightarrow q_{0}q_{1}\mathcal{O}_{\widetilde{\mathbb{V}}_{d}}(-\alpha_{d})\longrightarrow q_{0}\mathcal{O}_{\widetilde{\mathbb{V}}_{d}}(-\epsilon_d)\oplus q_{1}\mathcal{O}_{\widetilde{\mathbb{V}}_{d}}(-\epsilon_{d})\longrightarrow\mathcal{O}_{\widetilde{\mathbb{V}}_{d}}\longrightarrow\mathcal{O}_{L_{d}/B_{d}}\longrightarrow0\quad (i=d),$$
which leads to
$$[\mathcal{O}_{L_{i}/B_{i}}]=
\begin{cases}
    [\mathcal{O}_{\widetilde{\mathbb{V}}_{i}}]-q_2[\mathcal{O}_{\widetilde{\mathbb{V}}_{i}}(-\alpha_{i})], &\text{if} \  i\neq d,\\
    [\mathcal{O}_{\widetilde{\mathbb{V}}_{d}}]-(q_0+q_1)[\mathcal{O}_{\widetilde{\mathbb{V}}_d}(-\epsilon_d)]+q_0q_1[\mathcal{O}_{\widetilde{\mathbb{V}}_{i}}(-\alpha_{d})], &\text{if} \  i=d.
\end{cases}$$
Therefore,
\begin{align*}
    &R\Gamma([\mathcal{O}_{L_i/B_i}(-\lambda)] \otimes_{\widetilde{\mathbb{V}}_{i}} [\mathcal{O}_{L_i/B_i}])\\
    =&
\small{\begin{cases}
    R\Gamma([\mathcal{O}_{L_i/B_i}(-\lambda)])-q_2 R\Gamma([\mathcal{O}_{L_i/B_i}(-\lambda-\alpha_i)]), &\text{if} \  i\neq d,\\
    R\Gamma([\mathcal{O}_{L_d/B_d}(-\lambda)])-(q_0+q_1)R\Gamma([\mathcal{O}_{L_d/B_d}(-\lambda-\epsilon_d)])+q_0 q_1 R\Gamma([\mathcal{O}_{L_d/B_d}(-\lambda-\alpha_d)]), &\text{if} \  i=d.
\end{cases}}
\end{align*}
By abuse of notation, we simply write $[\mathcal{O}_{L_{i}/B_i}(-\alpha)]$ by $e^{\alpha}$, then $$R\Gamma([\mathcal{O}_{L_{i}/B_i}(-\alpha)])\cdot[\mathcal{O}_{L_i/B_i}]=\frac{e^{\alpha}-e^{s_{i}(\alpha)-\alpha_i}}{1-e^{-\alpha_i}}.$$
Finally, we get
\begin{align*}
&R\Gamma([\mathcal{O}_{L_i/B_i}(-\lambda)]\otimes_{\widetilde{\mathbb{V}}_{i}}[\mathcal{O}_{L_i/B_i}])\cdot[\mathcal{O}_{L_i/B_i}]\\
=&\begin{cases}
\frac{e^{\lambda}-e^{s_{i}(\lambda)-\alpha_i}}{1-e^{-\alpha_i}}-q_2\frac{e^{\lambda+\alpha_i}-e^{s_{i}(\lambda)-2\alpha_i}}{1-e^{-\alpha_i}}, &\text{if} \  i\neq d,\\
\frac{e^{\lambda}-e^{s_{d}(\lambda)-\alpha_d}}{1-e^{-\alpha_d}}-(q_0+q_1)\frac{e^{\lambda+\epsilon_d}-e^{s_{d}(\lambda)-\epsilon_d-\alpha_d}}{1-e^{-\alpha_d}}+q_0q_1\frac{e^{\lambda+\alpha_d}-e^{s_{d}(\lambda)-2\alpha_d}}{1-e^{-\alpha_d}}, &\text{if} \  i=d.
\end{cases}
\end{align*}The proposition follows.
\end{proof}
Let $\rho_{\mathbf{v}} \in Q$ be the Weyl vector associated with $\Pi_\mathbf{v}$, i.e. $\alpha_i^\vee(\rho_\mathbf{v})=\begin{cases}
    1,& \text{if} \  \alpha_i \in \Delta_\mathbf{v},\\
    0,& \text{if} \  \alpha_i \notin \Delta_\mathbf{v}.
\end{cases}$
\begin{cor}\label{cor-mul}
 For $\alpha_i\in\Delta_{\mathbf{v}}$, we have $e_{\mathbf{v}}^{-\rho_{\mathbf{v}}}*\vartheta(T_{i})=-e_{\mathbf{v}}^{-\rho_{\mathbf{v}}}$.
\end{cor}
\begin{proof}
    This is a consequence of a direct calculation using Proposition \ref{tm coin} and Proposition \ref{formula2} as follows. In this proof, we shall always write $e^\lambda$ instead of $e^\lambda_{\mathbf{v}}$ for short.

    For $i\neq d$,
    \begin{align*}
       &e^{-\rho_{\mathbf{v}}}*\vartheta(T_{i})\\=&\frac{e^{-\rho_{\mathbf{v}}}-e^{-s_i(\rho_{\mathbf{v}})-\alpha_i}}{1-e^{-\alpha_i}}-q_2\frac{e^{-\rho_{\mathbf{v}}+\alpha_i}-e^{-s_{i}(\rho_{\mathbf{v}})-2\alpha_i}}{1-e^{-\alpha_i}}+q_2(e^{-\rho_{\mathbf{v}}+\alpha_i}+e^{-\rho_\mathbf{v}})-e^{-\rho_{\mathbf{v}}}\\=&-q_2\frac{1-e^{-2\alpha_i}}{1-e^{-\alpha_i}}e^{-\rho_{\mathbf{v}}+\alpha_i}+q_2(e^{-\rho_{\mathbf{v}}+\alpha_i}+e^{-\rho_\mathbf{v}})-e^{-\rho_{\mathbf{v}}}\\=&-e^{-\rho_{\mathbf{v}}}.
    \end{align*}

    For $i=d$,
    \begin{align*}
        e^{-\rho_{\mathbf{v}}}*\vartheta(T_{d})=&\frac{e^{-\rho_{\mathbf{v}}}-e^{-s_{d}(\rho_{\mathbf{v}})-\alpha_d}}{1-e^{-\alpha_d}}-(q_0+q_1)\frac{e^{-\rho_\mathbf{v}+\epsilon_d}-e^{-s_{d}(\rho_\mathbf{v})-\epsilon_d-\alpha_d}}{1-e^{-\alpha_d}}\\&+q_0q_1\frac{e^{-\rho_\mathbf{v}+\alpha_d}-e^{-s_{d}(\rho_{\mathbf{v}})-2\alpha_d}}{1-e^{-\alpha_d}}+(q_0+q_1)e^{-\rho_\mathbf{v}+\epsilon_d}\\&-q_0q_1(e^{-\rho_{\mathbf{v}}+\alpha_d}+e^{-\rho_{\mathbf{v}}})-e^{-\rho_{\mathbf{v}}}\\=&
        -(q_0+q_1)\frac{1-e^{-2\epsilon_d}}{1-e^{-\alpha_d}}e^{-\rho_\mathbf{v}+\epsilon_d}+q_0q_1\frac{1-e^{-2\alpha_d}}{1-e^{-\alpha_d}}e^{-\rho_\mathbf{v}+\alpha_d}\\&+(q_0+q_1)e^{-\rho_{\mathbf{v}}+\epsilon_d}-q_0q_1(e^{-\rho_{\mathbf{v}}+\alpha_d}+e^{-\rho_\mathbf{v}})-e^{-\rho_\mathbf{v}}\\=&-e^{-\rho_\mathbf{v}}.
    \end{align*}
\end{proof}
\subsection{Restriction formulas}
Let $X$ be a closed $\breve{G}$-subvariety of $Y$. We identify $K^{\breve{G}}(X)$ with $K^{\breve{G}}(Y,X)$. Let $Y'$ be an open $\breve{G}$-subvariety of $Y$, then $X'=X\cap Y'$ becomes an open $\breve{G}$-subvariety of $X$. We identify the restriction morphism $\cdot|_{X'}:K^{\breve{G}}(X)\to K^{\breve{G}}(X')$ with $\cdot|_{Y'}:K^{\breve{G}}(Y,X)\to K^{\breve{G}}(Y',X')$. If $Y''$ is another open $\breve{G}$-subvariety of $Y$ such that $Y''\cap X=Y'\cap X=X'$, then $\cdot|_{Y'}$ and $\cdot|_{Y''}$ determine the same restriction map $\cdot|_{X'}$ under the identification $K^{\breve{G}}(Y,X)=K^{\breve{G}}(X)$.

\begin{lem}\label{res1}
    Let $w$ be the unique shortest element in $W_\mathbf{v}w$ for $\mathbf{v}\in\Lambda_{\ff}$. One has
    $$[\mathcal{O}_{\mathbf{T}_{\mathbf{v},\id}}]*[\mathcal{O}_{\mathbf{T}_w}]\in K^{\breve{G}}(\mathbf{Z}_{\mathbf{v}\mathbf{d}}^{\le w}).$$ 
    Moreover, $$([\mathcal{O}_{\mathbf{T}_{\mathbf{v},\id}}]*[\mathcal{O}_{\mathbf{T}_w}])|_{\mathbf{T}_{\Ob_{\mathbf{v},w}}}=[\mathcal{O}_{\mathbf{T}_{\Ob_{\mathbf{v},w}}}].$$
\end{lem}
\begin{proof}
    Noting that $\mathbf{T}_{\mathbf{v},\id}\circ\mathbf{T}_w\subseteq \mathbf{T}_{\mathbf{v},\id}\circ\mathbf{Z}^{\le w}\subseteq\mathbf{Z}_{\mathbf{v}\mathbf{d}}^{\le w}$, the first statement follows. We regard $[\mathcal{O}_{\mathbf{T}_{\mathbf{v},\id}}]\in K^{\breve{G}}(F_\mathbf{v}\times F,\mathbf{T}_{\mathbf{v},\id})$ and $[\mathcal{O}_{\mathbf{T}_w}]\in K^{\breve{G}}(F\times F,\mathbf{Z}^{\le w})$. Set $S=\bigsqcup_{y\ge w,y\in\D_{\mathbf{v}}}{\Ob_{\mathbf{v},y}}$ and $\widetilde{S}=pr^{-1}_{13}(\pi^{-1}_{\mathbf{v}}(S))$.
    Note that $\pi^{-1}_{\mathbf{v}}(S)\cap\mathbf{Z}^{\le w}_{\mathbf{v}\mathbf{d}}=\mathbf{T}_{\Ob_{\mathbf{v},w}}$.
    By a proper base change,
    \begin{align*}
        ([\mathcal{O}_{\mathbf{T}_{\mathbf{v},\id}}]*[\mathcal{O}_{\mathbf{T}_w}])|_{\mathbf{T}_{\Ob_{\mathbf{v},w}}}&=([\mathcal{O}_{\mathbf{T}_{\mathbf{v},\id}}]*[\mathcal{O}_{\mathbf{T}_w}])|_{\pi^{-1}_{\mathbf{v}}(S)}\\=&pr_{13*}((pr_{12}^{*}[\mathcal{O}_{\mathbf{T}_{\mathbf{v},\id}}])|_{\widetilde{S}}\otimes_{\widetilde{S}} (pr_{23}^*[\mathcal{O}_{\mathbf{T}_w}])|_{\widetilde{S}})=pr_{13*}([\mathcal{O}_{X}]\otimes_{\widetilde{S}} [\mathcal{O}_Y]),
    \end{align*}
    where $X=\widetilde{S}\cap pr^{-1}_{12}(\mathbf{T}_{\mathbf{v},\id})$ and $Y=\widetilde{S}\cap pr^{-1}_{23}(\mathbf{T}_w)$.
    By $X\cap Y=pr^{-1}_{12}(\mathbf{T}_{\mathbf{v},\id})\cap pr^{-1}_{23}(\mathbf{T}_{\Ob_w})$ and Lemma~\ref{intr-2}, we know that $X$ and $Y$ intersect transversely in $\widetilde{S}$. Hence, $$pr_{13*}([\mathcal{O}_X]\otimes_{\widetilde{S}} [\mathcal{O}_Y])=pr_{13*}[\mathcal{O}_{X\cap Y}]=[\mathcal{O}_{\mathbf{T}_{\Ob_{\mathbf{v},w}}}].$$ The proof is complete.
\end{proof}

We denote $e^\lambda_{\mathbf{w},\id,\mathbf{u}}=[\mathcal{O}_{\mathbf{T}_{\mathbf{w},\id,\mathbf{u}}}(-\lambda)]$. 
\begin{lem}\label{res2}
    For $\mathbf{v},\mathbf{w},\mathbf{u}\in\Lambda_\ff$ with $P_{\mathbf{u}}\subseteq P_\mathbf{w}$ and $w\in\D_{\mathbf{v}\mathbf{w}}$, we have $$\chi^w_{\mathbf{v}\mathbf{w}}*e^0_{\mathbf{w},\id,\mathbf{u}}\in K^{\breve{G}}(\mathbf{Z}^{\preceq w}_{\mathbf{v}\mathbf{u}}).$$ Moreover, $$(\chi^w_{\mathbf{v}\mathbf{w}}*e^0_{\mathbf{w},\id,\mathbf{u}})|_{\mathbf{T}_{\Ob_{\mathbf{v},w,\mathbf{u}}}}=[\mathcal{O}_{\mathbf{T}_{\Ob_{\mathbf{v},w,\mathbf{u}}}}(\chi)].$$
\end{lem}
\begin{proof}
Since $\mathbf{Z}^{\le w}_{\mathbf{v}\mathbf{w}}\circ\mathbf{T}_{\mathbf{w},\id,\mathbf{u}}\subseteq\mathbf{Z}_{\mathbf{v}\mathbf{u}}^{\le w}\subseteq\mathbf{Z}^{\preceq w}_{\mathbf{v}\mathbf{u}}$, the first statement follows.
We regard $\chi^w_{\mathbf{v}\mathbf{w}}\in K^{\breve{G}}(F_{\mathbf{v}\mathbf{w}},\mathbf{Z}^{\le w}_{\mathbf{v}\mathbf{w}})$ and $e^0_{\mathbf{w},\id,\mathbf{u}}\in K^{\breve{G}}(F_{\mathbf{w}\mathbf{u}},\mathbf{T}_{\mathbf{w},\id,\mathbf{u}})$.
We first assume $\chi\in R(\breve{P}_\mathbf{v})R(w\breve{P}_{\mathbf{w}}w^{-1})$. 
Without loss of generality, we let $\chi=\chi'\cdot w(\chi'')$, where $\chi'\in R(\breve{P}_\mathbf{v})$ and $\chi''\in R(\breve{P}_\mathbf{w})$. Then $\chi^w_{\mathbf{v}\mathbf{w}}|_{\mathbf{T}_{\Ob_{\mathbf{v},w,\mathbf{w}}}}=p^*_1(\chi')\otimes p^*_2(\chi'')\otimes [\mathcal{O}_{\mathbf{T}_{\Ob_{\mathbf{v},w,\mathbf{w}}}}],$
where $p_1:\mathbf{T}_{\Ob_{\mathbf{v},w,\mathbf{w}}}\to\Fm_{\mathbf{v}}$ and $p_2:\mathbf{T}_{\Ob_{\mathbf{v},w,\mathbf{w}}}\to\Fm_{\mathbf{w}}$ are canonical projections.
Set $S=\bigsqcup_{y_{\mathbf{v}\mathbf{u}}^+\ge w,y\in\D_{\mathbf{v}\mathbf{u}}}\Ob_{\mathbf{v},y,\mathbf{u}}$ and let $\widetilde{S}=pr_{13}^{-1}(\pi^{-1}_{\mathbf{v}\mathbf{u}}(S))$.
Set $Z=pr_{12}^{-1}(\mathbf{Z}^{\le w}_{\mathbf{v}\mathbf{w}})\cap pr_{23}^{-1}(\mathbf{T}_{\mathbf{w},\id,\mathbf{u}})$ and $Z'=Z\cap\widetilde{S}$.
By a proper base change,
\begin{align*}(\chi^w_{\mathbf{v}\mathbf{w}}*e^0_{\mathbf{w},\id,\mathbf{u}})|_{\mathbf{T}_{\Ob_{\mathbf{v},w,\mathbf{u}}}}=(\chi_{\mathbf{v}\mathbf{w}}^{w} \star e_{\mathbf{w}, \id}^0)|_{\pi^{-1}_{\mathbf{v}\mathbf{u}}(S)}=pr'_{13*}((pr_{12}^{*}(\chi_{\mathbf{v}\mathbf{w}}^{w})\otimes pr_{23}^{*}(e^0_{\mathbf{w},\id,\mathbf{u}}))|_{\widetilde{S}}).
\end{align*} 
Here $pr'_{13}:Z'\to\mathbf{T}_{\Ob_{\mathbf{v},w,\mathbf{u}}}$ is the restriction of $pr_{13}$, which is an isomorphism thanks to Lemma~\ref{inter trans}.
 Let $S'=\bigsqcup_{y\ge w,y\in\D_{\mathbf{v}\mathbf{w}}}\Ob_{\mathbf{v},y,\mathbf{w}}$ and $\widetilde{S}'=pr_{12}^{-1}(\pi_{\mathbf{v}\mathbf{w}}^{-1}(S'))$.
Note that $pr_{12}^{*}(\chi_{\mathbf{v}\mathbf{w}}^{w})\otimes pr_{23}^{*}(e^0_{\mathbf{w},\id,\mathbf{u}})\in K^{\breve{G}}(F_{\mathbf{v}\mathbf{w}\mathbf{u}},Z)=K^{\breve{G}}(Z)$ and $Z\cap\widetilde{S}=Z\cap\widetilde{S}\cap\widetilde{S}'=Z'$. We have \begin{align*}
    (pr_{12}^{*}(\chi_{\mathbf{v}\mathbf{w}}^{w})\otimes_{F_{\mathbf{vwu}}} pr_{23}^{*}(e^0_{\mathbf{w},\id,\mathbf{u}}))|_{\widetilde{S}}&=(pr_{12}^{*}(\chi_{\mathbf{v}\mathbf{w}}^{w})\otimes_{F_{\mathbf{vwu}}} pr_{23}^{*}(e^0_{\mathbf{w},\id,\mathbf{u}}))|_{\widetilde{S}\cap\widetilde{S}'}\\&=pr_{12}^{*}(\chi_{\mathbf{v}\mathbf{w}}^{w})|_{\widetilde{S}\cap\widetilde{S}'}\otimes_{\widetilde{S}\cap\widetilde{S}'} pr_{23}^{*}(e^0_{\mathbf{w},\id,\mathbf{u}})|_{\widetilde{S}\cap\widetilde{S}'}\in K^{\breve{G}}(Z').
\end{align*}

    Denote $X=pr^{-1}_{12}(\mathbf{Z}^{\le w}_{\mathbf{v}\mathbf{w}})\cap\widetilde{S}\cap\widetilde{S}'$ and $Y=pr^{-1}_{23}(\mathbf{T}_{\mathbf{w},\id,\mathbf{u}})\cap\widetilde{S}\cap\widetilde{S}'$, which are both smooth. They intersect transversely in $\widetilde{S}\cap\widetilde{S}'$ by Lemma \ref{inter trans}, and their intersection is $X\cap Y=Z'$.
There come the following commutative diagrams
$$\begin{tikzcd}
X \arrow[d, "p"'] \arrow[r, "j", hook]             & \widetilde{S}\cap\widetilde{S}' \arrow[d, "p'"] \\
{\mathbf{T}_{\Ob_{\mathbf{v},w,\mathbf{w}}}} \arrow[r, "p_1"] & \Fm_\mathbf{v}
\end{tikzcd},
\ \begin{tikzcd}
X \arrow[d, "p"'] \arrow[r, "j", hook]             & \widetilde{S}\cap\widetilde{S}' \arrow[d, "p''"] \\
{\mathbf{T}_{\Ob_{\mathbf{v},w,\mathbf{w}}}} \arrow[r, "p_2"] & \Fm_\mathbf{w}
\end{tikzcd},\ \begin{tikzcd}
X \arrow[d, "p"'] \arrow[r, hook]                 & pr^{-1}_{12}(\mathbf{Z}_{\mathbf{v}\mathbf{w}}^{\le w}) \arrow[d, "pr_{12}"] \\
{\mathbf{T}_{\Ob_{\mathbf{v},w,\mathbf{w}}}} \arrow[r, hook] & \mathbf{Z}_{\mathbf{v}\mathbf{w}}^{\le w}
\end{tikzcd},$$
by which we have \begin{align*}
    pr^{*}_{12}(\chi_{\mathbf{v}\mathbf{w}}^{w})|_{\widetilde{S}\cap\widetilde{S}'}&=j_{*}(p^*(p^*_1(\chi')\otimes_{\mathbf{T}_{\Ob_{\mathbf{v},w,\mathbf{w}}}} p^*_2(\chi'')))=j_*j^*({p'}^*(\chi')\otimes_{\widetilde{S}\cap\widetilde{S}'} {p''}^*(\chi''))\\&={p'}^*(\chi')\otimes_{\widetilde{S}\cap\widetilde{S}'} {p''}^*(\chi'')\otimes_{\widetilde{S}\cap\widetilde{S}'} [\mathcal{O}_X]\in K^{\breve{G}}(\widetilde{S}\cap\widetilde{S}',X)
\end{align*}
and $pr_{23}^*(e^0_{\mathbf{w},\id,\mathbf{u}})|_{\widetilde{S}\cap\widetilde{S}'}=[\mathcal{O}_{Y}]\in K^{\breve{G}}(\widetilde{S}\cap\widetilde{S}',Y)$. Moreover, we have a commutative diagram
$$\begin{tikzcd}
                                                                       & Z' \arrow[ld, "\pi_1"'] \arrow[d, "\pi_2"] \arrow[rd, "\pi_3"] \arrow[ldd, "pr'_{13}"] &                           \\
\Fm_\mathbf{v}                                                             & \Fm_\mathbf{w}                                                                                & \Fm_\mathbf{u} \arrow[l, "\pi"'] \\
{\mathbf{T}_{\Ob_{\mathbf{v},w,\mathbf{u}}}} \arrow[u, "p_1"] \arrow[rru, "p_3"'] &                                                                                        &
\end{tikzcd}$$ where $\pi:\Fm_\mathbf{u}\to\Fm_\mathbf{w}$ is the canonical map induced by $P_\mathbf{u}\subseteq P_\mathbf{w}$, which implies an inclusion $R(\breve{P}_\mathbf{w})\hookrightarrow R(\breve{P}_\mathbf{u})$, and the other arrows are all canonical.
Hence \begin{align*}
(\chi^w_{\mathbf{v}\mathbf{w}}*e^0_{\mathbf{w},\id,\mathbf{u}})|_{\mathbf{T}_{\Ob_{\mathbf{v},w,\mathbf{u}}}}&={pr'}_{13*}(pr^{*}_{12}(\chi_{\mathbf{v}\mathbf{w}}^{w})|_{\widetilde{S}\cap\widetilde{S}'}\otimes_{\widetilde{S}\cap\widetilde{S}'} pr_{23}^*(e^0_{\mathbf{w},\id,\mathbf{u}})|_{\widetilde{S}\cap\widetilde{S}'})\\&=pr'_{13*}({p'}^*(\chi')\otimes_{\widetilde{S}\cap\widetilde{S}'} {p''}^*(\chi'')\otimes_{\widetilde{S}\cap\widetilde{S}'} [\mathcal{O}_X]\otimes_{\widetilde{S}\cap\widetilde{S}'} [\mathcal{O}_{Y}])\\&=
pr'_{13*}(\pi_1^*(\chi')\otimes_{Z'} \pi_2^*(\chi''))\\&=pr'_{13*}(\pi_1^*(\chi')\otimes_{Z'} \pi_3^*(\chi''))\\&=pr'_{13*}{pr'}_{13}^*(p^*_1(\chi')\otimes_{\mathbf{T}_{\Ob_{\mathbf{v},w,\mathbf{u}}}} p^*_3(\chi''))\\&=p^*_1(\chi')\otimes_{\mathbf{T}_{\Ob_{\mathbf{v},w,\mathbf{u}}}} p^*_3(\chi'')=[\mathcal{O}_{\mathbf{T}_{\Ob_{\mathbf{v},w,\mathbf{u}}}}(\chi)].
\end{align*}

Now let $\chi\in R(\breve{P}^w_{\mathbf{v}\mathbf{w}})$. By \eqref{loc gal}, there exists a nonzero $\chi'''\in R(\breve{G})$ such that $\chi'''\cdot\chi\in\mathscr{R}_{\mathbf{v}\mathbf{w}}^w$. By the above discussion, we have $$((\chi'''\cdot\chi)^w_{\mathbf{v}\mathbf{w}}*e^0_{\mathbf{w},\id,\mathbf{u}})|_{\mathbf{T}_{\Ob_{\mathbf{v},w,\mathbf{u}}}}=[\mathcal{O}_{\mathbf{T}_{\Ob_{\mathbf{v},w,\mathbf{u}}}}(\chi'''\cdot\chi)].$$ Since the maps here are all $R(\breve{G})$-linear, we have $$\chi'''\cdot(\chi^w_{\mathbf{v}\mathbf{w}}*e^0_{\mathbf{w},\id,\mathbf{u}})|_{\mathbf{T}_{\Ob_{\mathbf{v},w,\mathbf{u}}}}=\chi'''\cdot[\mathcal{O}_{\mathbf{T}_{\Ob_{\mathbf{v},w,\mathbf{u}}}}(\chi)].$$ Note that $K^{\breve{G}}(\mathbf{T}_{\Ob_{\mathbf{v},w,\mathbf{u}}})$ is a free $R(\breve{G})$-module; in particular, it has no torsion. Hence $$(\chi^w_{\mathbf{v}\mathbf{w}}*e^0_{\mathbf{w},\id,\mathbf{u}})|_{\mathbf{T}_{\Ob_{\mathbf{v},w,\mathbf{u}}}}=[\mathcal{O}_{\mathbf{T}_{\Ob_{\mathbf{v},w,\mathbf{u}}}}(\chi)].$$ The proof is complete.
\end{proof}

%------------------------------------------------
\subsection{Realization of $\widetilde{\mathbb{S}}_\mathbf{f}$}
%----------------------------------------------
\begin{prop}\label{fock}
\
    There exists an $\widetilde{\mathbb{H}}$-module isomorphism $$\widetilde{\phi}:\widetilde{\mathbb{T}}_{\ff}\simeq K^{\breve{G}}(F_{\ff}\times_{\mathcal{N}}F),\ x_{\mathbf{v}}\mapsto e_{\mathbf{v}}^{-\rho_{\mathbf{v}}}.$$
\end{prop}
\begin{proof}
    By Lemma \ref{rann}, Theorem \ref{exo-H} and Corollary \ref{cor-mul}, one has $r_{\widetilde{\mathbb{H}}}(x_\mathbf{v})\subseteq r_{\widetilde{\mathbb{H}}}(e^{-\rho_\mathbf{v}}_{\mathbf{v},\id })$. Hence, there exists a homomorphism of (right) $\widetilde{\mathbb{H}}$-modules $\widetilde{\phi}:\widetilde{\mathbb{T}}_{\ff}\to K^{\breve{G}}(F_{\ff}\times_{\mathcal{N}}F)$ via $x_{\mathbf{v}}\mapsto e_{\mathbf{v}}^{-\rho_{\mathbf{v}}}$. To show the surjectivity, it is enough to show that $K^{\breve{G}}(F_\mathbf{v}\times_{\mathcal{N}}F)$ is generated by $e^{-\rho_\mathbf{v}}_{\mathbf{v},\id }$ as a (right) $\widetilde{\mathbb{H}}$-module. By Proposition \ref{tm coin}, we have $[\mathcal{O}_{\mathbf{T}_{\mathbf{v},w}}]*\mathbf{e}^{\lambda}=e^\lambda\cdot[\mathcal{O}_{\mathbf{T}_{\mathbf{v},w}}]$. In particular, $e_{\mathbf{v}}^0=e^{-\rho_\mathbf{v}}_{\mathbf{v},\id }*\mathbf{e}^{\rho_\mathbf{v}}$ is in the image $\mathrm{Im}\widetilde{\phi}$. It remains to show $[\mathcal{O}_{\mathbf{T}_{\mathbf{v},w}}]\in\mathrm{Im}\widetilde{\phi}$. 
    By Lemma \ref{res1} and Proposition \ref{bas1}, we know for any $w\in\D_{\mathbf{v}}$ and $w'$ the unique shortest element in ${W}_\mathbf{v}w$, 
    $$e_{\mathbf{v}}^0*[\mathcal{O}_{\mathbf{T}_{w'}}]=[\mathcal{O}_{\mathbf{T}_{\mathbf{v},w}}]+\sum\limits_{y<w,y\in\D_{\mathbf{v}}}a_y[\mathcal{O}_{\mathbf{T}_{\mathbf{v},y}}],\quad \text{where } a_y\in\mathcal{A}[Q],$$ 
    which shows $[\mathcal{O}_{\mathbf{T}_{\mathbf{v},w}}]\in \mathrm{Im}\widetilde{\varphi}$ by a standard argument about the upper triangular matrices. Now we have a surjective morphism $\widetilde{\phi}$ of $\mathcal{A}[Q]$-modules. By Lemma \ref{rk-Tf} and Proposition \ref{bas1}, $\widetilde{\mathbb{T}}_\ff$ and $K^{\breve{G}}(F_\ff\times_{\mathcal{N}}F)$ are free $\mathcal{A}[Q]$-modules with the same rank $\sum_{\mathbf{v}\in\Lambda_{\ff}}|\D_{\mathbf{v}}|$. Thus, the surjectivity of $\widetilde{\phi}$ also implies its injectivity.
\end{proof}

\begin{rem}
If we take $\Lambda_\ff=\{\mathbf{v}\}$, a single orbit, then the above proposition gives an equivariant K-theoretical realization of the anti-spherical module (associated with $W_\mathbf{v}$) of $\widetilde{\mathbb{H}}$ of type $\widetilde{C}_d$ with three parameters. 
\end{rem}

Now, we have a commutative diagram of algebra homomorphisms:
\begin{equation*}%\label{diag:main}
   \begin{tikzcd}
\widetilde{\mathbb{S}}_{\ff} \arrow[r, equal]                    & \End_{\widetilde{\mathbb{H}}}(\widetilde{\mathbb{T}}_{\ff})                                     \\
K^{\breve{G}}(\mathbf{Z}_\ff) \arrow[u, "\widetilde{\psi}"', dotted] \arrow[r] & \End_{\widetilde{\mathbb{H}}}(K^{\breve{G}}(F_{\ff}\times_{\mathcal{N}}F)) \arrow[u, "\simeq"']
\end{tikzcd}
\end{equation*}
where the vertical map on the right is due to Proposition \ref{fock}, the horizon map on the bottom is due to the associativity of convolution, and  $\widetilde{\psi}$ is the unique algebra homomorphism that makes the diagram commute. Precisely, $\widetilde{\psi}:K^{\breve{G}}(\mathbf{Z}_\ff) \to \widetilde{\mathbb{S}}_\ff$ is uniquely determined by $\widetilde{\psi}(a)\cdot b=a*b$ for $a\in K^{\breve{G}}(\mathbf{Z}_\ff)$ and $b\in K^{\breve{G}}(F_\ff \times_\mathcal{N}F)$, where we identify $K^{\breve{G}}(F_\ff \times_\mathcal{N}F)$ with $\widetilde{\mathbb{T}}_\ff$ under the isomorphism $\widetilde{\phi}$.

\begin{thm}\label{main}
The homomorphism $\widetilde{\psi}$ is an isomorphism between $K^{\breve{G}}(\mathbf{Z}_{\ff})$ and $\widetilde{\mathbb{S}}_{\ff}$.
\end{thm}
\begin{proof}
Firstly, let us show the injectivity of $\widetilde{\psi}$.
Lemma~\ref{res2} implies $\chi^w_{\mathbf{v}\mathbf{w}}*e^0_{\mathbf{w},\id}\in\chi^w_{\mathbf{v}}+K^{\breve{G}}(\mathbf{Z}^{\prec w}_{\mathbf{v}})$, for any $w\in\D_{\mathbf{v}\mathbf{w}}$ and $\chi\in R(\breve{P}^w_{\mathbf{v}\mathbf{w}})$. Take any nonzero element $\mathcal{M}\in K^{\breve{G}}(\mathbf{Z}_\ff)$. Without loss of generality, we may assume $\mathcal{M}\in K^{\breve{G}}(\mathbf{Z}^{\preceq w}_{\mathbf{v}\mathbf{w}})\backslash K^{\breve{G}}(\mathbf{Z}^{\prec w}_{\mathbf{v}\mathbf{w}})$, i.e. $\mathcal{M}\in\chi^w_{\mathbf{v}\mathbf{w}}+K^{\breve{G}}(\mathbf{Z}^{\prec w}_{\mathbf{v}\mathbf{w}})$ for $\chi\neq0$. By Lemma \ref{res2} again, we have $(\widetilde{\psi}(\mathcal{M})\cdot e^0_{\mathbf{w},\id})|_{\mathbf{T}_{\Ob_{\mathbf{v},w}}}=(\mathcal{M}*e^0_{\mathbf{w},\id})|_{\mathbf{T}_{\Ob_{\mathbf{v},w}}}=\chi\neq0$. Hence $\mathrm{Ker}\widetilde{\psi}=0$ and $\widetilde{\psi}$ is injective.

    Since $\mathrm{rank}_{R(\breve{G})}(K^{\breve{G}}(\mathbf{Z}_\ff))=\dim_{R(\breve{G})_{loc}} \widetilde{\mathbb{S}}_{\ff,loc}$ by Corollary~\ref{rk-KZ_f} and Proposition~\ref{rk-Sloc}, we obtain an isomorphism $\widetilde{\psi}_{loc}:K^{\breve{G}}(\mathbf{Z}_\ff)_{loc}\simeq\widetilde{\mathbb{S}}_{\ff,loc}$. Consider the following commutative diagram
    \begin{equation*}
  \begin{tikzcd}
K^{\breve{G}}(\mathbf{Z}_\ff) \arrow[rr, "\widetilde{\psi}", hook] \arrow[d, hook]             &  & \widetilde{\mathbb{S}}_\ff \arrow[d, hook] \arrow[r, hook]         & \End_{R(\breve{G})}(\widetilde{\mathbb{T}}_\ff) \arrow[d, hook]         \\
K^{\breve{G}}(\mathbf{Z}_\ff)_{loc} \arrow[rr, phantom] \arrow[rr, "\simeq"] &  & \widetilde{\mathbb{S}}_{\ff,loc} \arrow[r, hook] & \End_{R(\breve{G})}(\widetilde{\mathbb{T}}_\ff)_{loc}
\end{tikzcd}
\end{equation*}
where the right square is clearly a Cartesian diagram.
We will show that the biggest square is also Cartesian, which forces the left square to be Cartesian as well, and hence $\widetilde{\psi}$ is an isomorphism.
If not, then there exists $\mathcal{M} \in (\End_{R(\breve{G})}(\widetilde{\mathbb{T}}_\ff)\cap K^{\breve{G}}(\mathbf{Z}_\ff)_{loc})\backslash K^{\breve{G}}(\mathbf{Z}_{\ff})$. Without loss of generality, we may assume $\mathcal{M}\in K^{\breve{G}}(\mathbf{Z}^{\preceq w}_{\mathbf{v}\mathbf{w}})_{loc} \backslash K^{\breve{G}}(\mathbf{Z}^{\prec w}_{\mathbf{v}\mathbf{w}})_{loc}$ for some $w\in\mathcal{D}_{\mathbf{v}\mathbf{w}}$,
then $\mathcal{M} \in \sum\limits_{\chi\in\mathcal{B}(P^w_{\mathbf{v}\mathbf{w}})} a_\chi\chi_{\mathbf{v}\mathbf{w}}^w+K^{\breve{G}}(\mathbf{Z}^{\prec w}_{\mathbf{v}\mathbf{w}})_{loc}$, where $a_\chi\in R(\breve{G})_{loc}$ such that $a_\chi\in R(\breve{G})_{loc}\backslash R(\breve{G})$ for at least one $\chi$. By \cite[Theorem 2.2]{St75} and Corollary \ref{P^w}, $R(\breve{P}^w_{\mathbf{v}})\simeq R(\breve{T})\simeq K^{\breve{G}}(\mathbf{T}_{\Ob_{\mathbf{v},w}})$ is a free module over $R(\breve{P}^w_{\mathbf{v}\mathbf{w}})\simeq K^{\breve{G}}(\mathbf{T}_{\Ob_{\mathbf{v},w,\mathbf{w}}})$. There exists an $R(\breve{G})$-basis of $R(\breve{P}^w_{\mathbf{v}\mathbf{w}})$
which can be extended to an $R(\breve{G})$-basis of $R(\breve{P}^w_{\mathbf{v}})$. Hence we may assume $\mathcal{B}(P^w_{\mathbf{v}\mathbf{w}})\subseteq\mathcal{B}(P^w_{\mathbf{v}})$. Thanks to Lemma~\ref{res2}, $$(\mathcal{M}*e^0_{\mathbf{w},\id})|_{\mathbf{T}_{\Ob_{\mathbf{v},w}}}=\sum\limits_{\chi\in\mathcal{B}(P^w_{\mathbf{v}\mathbf{w}})\subseteq\mathcal{B}(P^w_{\mathbf{v}})}a_\chi\chi\notin K^{\breve{G}}(\mathbf{T}_{\Ob_{\mathbf{v},w}}).$$ Hence, $\mathcal{M}*e^0_{\mathbf{w},\id}\notin\widetilde{\mathbb{T}}_\ff$, which forces $\mathcal{M}\notin\End_{R(\breve{G})}(\widetilde{\mathbb{T}}_\ff)$, a contradiction. So we have $\End_{R(\breve{G})}(\widetilde{\mathbb{T}}_\ff)\cap K^{\breve{G}}(\mathbf{Z}_\ff)_{loc}=K^{\breve{G}}(\mathbf{Z}_{\ff})$, which implies that the largest square is Cartesian and hence $\widetilde{\psi}$ is an isomorphism.
\end{proof}

\begin{cor}
For $\mathbf{v},\mathbf{w}\in\Lambda_\ff$, we have$$\Hom_{\widetilde{\mathbb{H}}}(x_\mathbf{w}\widetilde{\mathbb{H}},x_\mathbf{v}\widetilde{\mathbb{H}})\simeq K^{\breve{G}}(\mathbf{Z}_{\mathbf{v}\mathbf{w}}).$$
\end{cor}
\begin{proof}
    By Theorem \ref{main}, we have $\End_{\widetilde{\mathbb{H}}}(x_\mathbf{v}\widetilde{\mathbb{H}})\simeq K^{\breve{G}}(\mathbf{Z}_{\mathbf{v}\mathbf{v}})$ by taking $\Lambda_\ff=\{\mathbf{v}\}$. Let $\id_\mathbf{v}\in\End_{\widetilde{\mathbb{H}}}(x_\mathbf{v}\widetilde{\mathbb{H}})$ be the identity, which corresponds to $[\mathcal{O}_{\mathbf{v},\id,\mathbf{v}}]$ in $K^{\breve{G}}(\mathbf{Z}_{\mathbf{v}\mathbf{v}})$. Recall $$\mathbb{S}_\ff=\End_{\widetilde{\mathbb{H}}}(\bigoplus\limits_{\mathbf{v}\in\Lambda_\ff}x_\mathbf{v}\widetilde{\mathbb{H}})=\bigoplus\limits_{\mathbf{v},\mathbf{w}\in\Lambda_\ff}\Hom_{\widetilde{\mathbb{H}}}(x_\mathbf{w}\widetilde{\mathbb{H}},x_\mathbf{v}\widetilde{\mathbb{H}}).$$ We have \begin{align*}
        \Hom_{\widetilde{\mathbb{H}}}(x_\mathbf{w}\widetilde{\mathbb{H}},x_\mathbf{v}\widetilde{\mathbb{H}})=\id_\mathbf{v}\cdot\mathbb{S}_\ff\cdot\id_\mathbf{w}\simeq[\mathcal{O}_{\mathbf{v},\id,\mathbf{v}}]*K^{\breve{G}}(\mathbf{Z}_\ff)*[\mathcal{O}_{\mathbf{w},\id,\mathbf{w}}]=K^{\breve{G}}(\mathbf{Z}_{\mathbf{v}\mathbf{w}})
    \end{align*} as desired.
\end{proof}

\subsection{Geometric Howe duality}
%----------------------------------------
Let $Q_\g\subseteq Q$ be another finite $W$-invariant subset. We use the subscript $\g$ to indicate the notion associated with $Q_\g$, such as $\widetilde{\mathbb{S}}_\g$, $F_\g$, etc. Denote $\widetilde{\mathbb{T}}_{\ff\g}=K^{\breve{G}}(F_\ff\times_\mathcal{N} F_\g)$, which admits a left $\widetilde{\mathbb{S}}_\ff$-action and a right $\widetilde{\mathbb{S}}_\g$-action under the convolution product.

\begin{lem}\label{lem:howe}
    If for each $\mathbf{w} \in \Lambda_\ff$ there exists an orbit $\mathbf{u} \in \Lambda_\g$ such that $P_\mathbf{u} \subseteq P_\mathbf{w}$, then $$\widetilde{\mathbb{S}}_\ff \simeq \End_{\widetilde{\mathbb{S}}_\g}(\widetilde{\mathbb{T}}_{\ff\g}).$$
\end{lem}

\begin{proof}
    Similarly to the proof of Theorem \ref{main}, by Lemma \ref{res2}, we see that if $P_\mathbf{u} \subseteq P_\mathbf{w}$, then
    \begin{align*}
        \chi_{\mathbf{v}\mathbf{w}}^{w} \star e_{\mathbf{w}, \id, \mathbf{u}}^0\in\chi^{w}_{\mathbf{v}\mathbf{u}}+K^{\breve{G}}(\mathbf{Z}^{\prec w}_{\mathbf{v}\mathbf{u}}), \quad \forall w \in \D_{\mathbf{v}\mathbf{w}} \ \text{and} \ \chi \in R(\breve{P}_{\mathbf{v}\mathbf{w}}^w).
    \end{align*} Therefore,  $\widetilde{\mathbb{S}}_\ff \to \End_{\widetilde{\mathbb{S}}_\g}(\widetilde{\mathbb{T}}_{\ff\g})$ is injective. Similarly to the arguments for Theorem \ref{main}, we have $$\End_{\widetilde{\mathbb{S}}_\g}(\widetilde{\mathbb{T}}_{\ff\g})=\End_{R(\breve{G})}(\widetilde{\mathbb{T}}_{\ff\g}) \cap \widetilde{\mathbb{S}}_{\ff,loc}=\widetilde{\mathbb{S}}_\ff.$$
\end{proof}

Thanks to Lemma~\ref{lem:howe}, we immediately obtain the following double centralizer property, which can be regarded as a quantized Howe duality of affine type at the Schur algebra level.
\begin{thm}\label{thm:howe}
If the subsets of minimal parabolic subgroups of $\{P_\mathbf{v}~|~\mathbf{v}\in\Lambda_\ff\}$ and of $\{P_\mathbf{v}~|~\mathbf{v}\in\Lambda_\g\}$ coincide, then $\widetilde{\mathbb{S}}_\ff$ and $\widetilde{\mathbb{S}}_\g$ admit a double centralizer property on $\widetilde{\mathbb{T}}_{\ff\g}$, saying
\begin{align*}
\widetilde{\mathbb{S}}_\ff \simeq \End_{\widetilde{\mathbb{S}}_\g}(\widetilde{\mathbb{T}}_{\ff\g}),\quad
\End_{\widetilde{\mathbb{S}}_\ff}(\widetilde{\mathbb{T}}_{\ff\g}) \simeq \widetilde{\mathbb{S}}_\g.
\end{align*}
\end{thm}

%===================================
\section{Construction of affine iquantum groups}\label{sec:consiquan}
%===================================
In this section, we realize affine iquantum groups $\mathbb{U}^{\imath\imath}$ and $\mathbb{U}^{\imath\jmath}$ using the equivariant K-theory. To adapt to the notation used in the prior literature, we use the parameter $q$ instead of $q_2$ with $q^2=q_2$ and take $\K=\C(q_0,q_1,q)$.

%------------------------------------------------------------
\subsection{Affine iquantum groups of type AIII}
%------------------------------------------------------------
The universal quasi-split affine iquantum groups $\mathbb{U}^{\imath\imath}$ and $\mathbb{U}^{\imath\jmath}$ of type AIII over the field $\K$ are defined, respectively, via the following two Satake diagrams. Each Satake diagram is a Dynkin diagram of affine type A with an involution $\tau$ on the set $\mathbb{I}$ of nodes. 

\begin{center}
$\mathbb{U}^{\imath\imath}=\mathbb{U}^{\imath\imath}_{2n}$ (type $\mathrm{AIII}_{2n-1}^{\imath\imath}$):
\begin{tikzpicture}[scale=.2, baseline=0]

\node at (4,0.85) {\tiny $1$};
\node at (12,0.75) {\tiny $n-1$};

\node at (4,-3.65) {\tiny $2n-1$};
\node at (12,-3.65) {\tiny $n+1$};
    
\node at (0,-0.75) {\tiny $0$};
\node at (16,-0.75) {\tiny $n$};
       
\node at (8,0) {$\dots$};
\node at (8,-3) {$\dots$};
    
\foreach \x in {2,6} 
{\draw[thick,xshift=\x cm] (\x, 0) circle (0.3); 
\draw[thick,xshift=\x cm] (\x, -3) circle (0.3);}
    
\foreach \x in {0,8} 
\draw[thick,xshift=\x cm] (\x, -1.5) circle (0.3); 
   
\foreach \x in 0
{\draw[thick,xshift=\x cm] (\x,-1.2) ++(0.5,0) -- +(3,1.2);
\draw[thick,xshift=\x cm] (\x,-1.8) ++(0.5,0) -- +(3,-1.2);}
   
\foreach \x in {2,4.5}
{\draw[thick,xshift=\x cm] (\x,0) ++(0.5,0) -- +(2,0);
\draw[thick,xshift=\x cm] (\x,-3) ++(0.5,0) -- +(2,0);}

\foreach \x in {6}
{\draw[thick,xshift=\x cm] (\x,0) ++(0.5,0) -- +(3,-1.2);
\draw[thick,xshift=\x cm] (\x,-3) ++(0.5,0) -- +(3,1.2);}
    
\foreach \x in {4} 
\draw[thick,<->, blue, bend right=50] (\x+0.3,-2.5) to (\x+0.3,-0.5);
\foreach \x in {12} 
\draw[thick,<->, blue, bend left=50] (\x-0.3,-2.5) to (\x-0.3,-0.5);
    
\foreach \x in 0 
\draw[thick,<->, blue, bend left=100,looseness=10] (\x-0.4,-1.7) to (\x-0.4,-1.3);
\foreach \x in {16} 
\draw[thick,<->, blue, bend right=100,looseness=10] (\x+0.4,-1.7) to (\x+0.4,-1.3);

\end{tikzpicture}
\\[12pt]
$\mathbb{U}^{\imath\jmath}=\mathbb{U}^{\imath\jmath}_{2n+1}$ (type $\mathrm{AIII}_{2n}^{\imath\jmath}$):
\begin{tikzpicture}[scale=.2, baseline=0]

\node at (4,0.85) {\tiny $1$};
\node at (12,0.75) {\tiny $n-1$};
\node at (16,0.75) {\tiny $n$};
    
\node at (4,-3.85) {\tiny $2n$};
\node at (12,-3.85) {\tiny $n+2$};
\node at (16,-3.85) {\tiny $n+1$};
    
\node at (0,-0.6) {\tiny $0$};
       
\node at (8,0) {$\dots$};
\node at (8,-3) {$\dots$};
    
\foreach \x in {2,6,8} 
{\draw[thick,xshift=\x cm] (\x, 0) circle (0.3); 
\draw[thick,xshift=\x cm] (\x, -3) circle (0.3);}
    
\foreach \x in 0 
\draw[thick,xshift=\x cm] (\x, -1.5) circle (0.3); 
   
\foreach \x in 0
{\draw[thick,xshift=\x cm] (\x,-1.2) ++(0.5,0) -- +(3,1.2);
\draw[thick,xshift=\x cm] (\x,-1.8) ++(0.5,0) -- +(3,-1.2);}
   
\foreach \x in {2,4.5}
{\draw[thick,xshift=\x cm] (\x,0) ++(0.5,0) -- +(2,0);
\draw[thick,xshift=\x cm] (\x,-3) ++(0.5,0) -- +(2,0);}

\foreach \x in {6}
{\draw[thick,xshift=\x cm] (\x,0) ++(0.5,0) -- +(3,0);
\draw[thick,xshift=\x cm] (\x,-3) ++(0.5,0) -- +(3,0);}
    
\foreach \x in {8}
\draw[thick,xshift=\x cm] (\x,-0.5) -- +(0,-2);
    
\foreach \x in {4} 
\draw[thick,<->, blue, bend right=50] (\x+0.3,-2.5) to (\x+0.3,-0.5);
\foreach \x in {12,16} 
\draw[thick,<->, blue, bend left=50] (\x-0.3,-2.5) to (\x-0.3,-0.5);
    
\foreach \x in 0 
\draw[thick,<->, blue, bend left=100,looseness=10] (\x-0.4,-1.7) to (\x-0.4,-1.3);

\end{tikzpicture}
\end{center}

In \cite{LWZ24}, it was shown that $\mathbb{U}^{\imath\imath}$ admits a Drinfeld new presentation, which was rewritten in \cite{SW26} with
generators
\begin{align*}
B_{i,l}, \quad \Theta_{i,m}, \quad \K_i^{\pm1} \quad \text{and}\quad C^{\pm1} \qquad (1 \leq i \leq 2n-1, l \in \Z, m \in \N)
\end{align*}
 subject to
\begin{align}\label{eq:1}
&\text{$C$ is central}, \quad [\K_i,\K_j]=[\K_i,\Theta_{j}(w)]=0,\\
&\Theta_{i}(z)\Theta_{j}(w)=\Theta_{j}(w)\Theta_i(z),\\
&\K_i\mathbf{B}_j(w)=q^{c_{\tau i,j}-c_{ij}}\mathbf{B}_j(w)\K_i,\\
&\mathbf{B}_j(w)\Theta_i(z)=\frac{(1-q^{c_{ij}}zw^{-1})(1-q^{-c_{\tau i,j}}zwC)}{(1-q^{-c_{ij}}zw^{-1})(1-q^{c_{\tau i,j}}zwC)}\Theta_i(z)\mathbf{B}_j(w),
\\
&[\mathbf{B}_i(z),\mathbf{B}_{\tau i}(w)]=\frac{\mathbf{\Delta}(zw)}{q-q^{-1}}(\K_{\tau i}\Theta_i(z)-\K_{i}\Theta_{\tau i}(w)), \quad \text{if} \  c_{i, \tau i}=0, 
\\
&\mathbf{B}_i(z)\mathbf{B}_{j}(w)=\mathbf{B}_j(w)\mathbf{B}_{i}(z), \quad \text{if} \  c_{ij}=0, \tau i\neq j,
\\\label{eq:6}
&(q^{c_{ij}}z-w)\mathbf{B}_i(z)\mathbf{B}_j(w)+(q^{c_{ij}}w-z)\mathbf{B}_j(w)\mathbf{B}_i(z)=0, \quad\text{if} \  i\neq \tau j,\\
&(q^2z-w)\mathbf{B}_n(z)\mathbf{B}_n(w)+(q^2w-z)\mathbf{B}_n(w)\mathbf{B}_n(z)\\
\nonumber&=\frac{q^{-2}\mathbf{\Delta}(zw)\K_n}{q-q^{-1}} \frac{(z-q^2w)(w-q^2z)}{z-w}(\Theta_n(z)-\Theta_n(w)),
\end{align}
and the Serre relations
\begin{align}\label{eq:7}
&\mathbb{S}_{i,j}(w_1,w_2|z)=0, \quad \text{if} \  c_{ij}=-1, \ i\neq \tau j \ \text{or} \ \tau i,
\end{align}
\begin{align}
&(z-qw_1)(z-qw_2)\mathbb{S}_{n,n\pm1}(w_1,w_2|z)=\\
\nonumber&\K_n\mathbf{\Delta}(w_1w_2)\frac{z(qw_1-q^{-1}w_2)(q^{-1}w_1-qw_2)}{w_1-w_2}\big(\Theta_n(w_2)\mathbf{B}_{n\pm1}(z)-\Theta_n(w_1)\mathbf{B}_{n\pm1}(z)\big),
\end{align}
where 
\begin{align*}
\mathbf{B}_i(z)=\sum_{l\in\Z}B_{i,l}z^l,\quad
\mathbf{\Theta}_i(z)=1+\sum_{m\in\N}(q-q^{-1})\Theta_{i,m}z^m,\quad
\mathbf{\Delta}(z)=\sum_{k\in\Z}C^kz^k
\end{align*}
and
\begin{align*}
&\mathbb{S}_{ij}(w_1,w_2|z)\\
&=\mathbf{B}_j(z)\mathbf{B}_i(w_1)\mathbf{B}_i(w_2)-[2]\mathbf{B}_i(w_1)\mathbf{B}_j(z)\mathbf{B}_i(w_2)+\mathbf{B}_i(w_1)\mathbf{B}_i(w_2)\mathbf{B}_j(z)+\{w_1 \leftrightarrow w_2\}.
\end{align*}

A Drinfeld new presentation for $\mathbb{U}^{\imath\jmath}$ was recently obtained in \cite{LPWZ25}, saying that $\mathbb{U}^{\imath\jmath}$ is generated by 
\begin{align*}
B_{i,l}, \quad \Theta_{i,m}, \quad \K_i^{\pm1} \quad \text{and}\quad C^{\pm1} \qquad (1 \leq i \leq 2n, l \in \Z, m \in \N)
\end{align*}
subject to \eqref{eq:1}-\eqref{eq:6}, \eqref{eq:7} and 
\begin{align}
    &(q^{-1}z-w)\mathbf{B}_n(z) \mathbf{B}_{n+1}(w)+(q^{-1}w-z)\mathbf{B}_{n+1}(w) \mathbf{B}_n(z)\\
    \nonumber=&\mathbf{\Delta}(wz)\frac{(z-q^{-1}w)\K_{n+1}\mathbf{\Theta}_n(z)+(w-q^{-1}z)\K_n\mathbf{\Theta}_{n+1}(w)}{q-q^{-1}},
    \end{align}
    \begin{align}
    \label{eq:8}&\mathbb{S}_{i,\tau i}(w_1,w_2|z)=\\
    \nonumber&[2]\mathbf{\Delta}(w_2z)(\frac{(qz-w_2)\K_i \mathbf{\Theta}_{\tau i}(z)\mathbf{B}_i(w_1)}{(qw_1-q^{-1}w_2)(qw_1^{-1}z-1)}+\frac{(qw_2-z)\mathbf{B}_i(w_1)\K_{\tau i} \mathbf{\Theta}_i(w_2)}{(q^{-1}w_1-qw_2)(w_1^{-1}z-q)})+\{w_1 \leftrightarrow w_2\},\\
\nonumber&\qquad\qquad\qquad\qquad\qquad\qquad\qquad\qquad\qquad\qquad\qquad\qquad
    \mbox{ if } i=n \ \text{or} \ n+1,
\end{align}
where $\eqref{eq:8}$ is a reformulation of \cite[(4.36)]{LPWZ25} by \cite[(4.30)]{LPWZ25}.

\begin{rem}
Note that for {\em universal} iquantum groups, all the generation relations listed above involve only a single parameter $q$. This stands in contrast to iquantum groups with parameters, which are reproduced from universal ones via a central reduction and carry additional parameters. Nevertheless, as we emphasize, the other two parameters will appear naturally in the geometric construction given below.
\end{rem}

%-----------------------------------------
\subsection{Polynomial representation}
%-----------------------------------------

Let $\mathbf{R}$ be the ring of $\K$-valued regular functions on $(\C^\times)^d$, which can be identified as $$\mathbf{R}=\K \otimes_\mathcal{A} K^{\breve{G}}(\Bm)$$ via $e^{\epsilon_i}(x)=x_i$ for any $x=(x_1,\ldots, x_d) \in (\C^\times)^d$. Define a $W$-action on $\{1,\ldots,2d\}$ via $$s_i(j)=\begin{cases}
    j+1, &\text{if} \  j=i \text{ or } 2d-i,\\
    j-1, &\text{if} \  j=i+1 \text{ or } 2d-i+1,\\
    j, &\text{otherwise}.
\end{cases}$$
Then the $W$-action on $(\C^\times)^d$, induced by the one on $\mathbf{R}$, satisfies
$$w(x_1,\ldots,x_d)=(x_{w^{-1}(1)},\ldots,x_{w^{-1}(d)})$$
where $x_i=x_{2d-i+1}^{-1}$ ($\forall i=1,2,\ldots, 2d$). For $\mathbf{v}\in\Lambda_{n,d}^\mathfrak{c}$, let $\mathbf{R}^\mathbf{v}$ denote the ring of $\K$-valued regular functions on $(\C^\times)^d/W_\mathbf{v}$, which can be identified as $$\mathbf{R}^{W_\mathbf{v}} \simeq \K \otimes_\mathcal{A} K^{\breve{G}}(\Fm_\mathbf{v}).$$ Denote
$$\mathbf{P}^\mathfrak{c}=\bigoplus_{\mathbf{v}\in\Lambda_{n,d}^\mathfrak{c}}\mathbf{R}^\mathbf{v}\simeq \K \otimes_\mathcal{A} K^{\breve{G}}(\bigsqcup_{\mathbf{v}\in\Lambda_{n,d}^\mathfrak{c}} \Fm_\mathbf{v}) \qquad(\mathfrak{c}=\imath,\jmath).$$
In particular, we use $(2d)$ to represent the unique element in $\Lambda_{0,d}^\jmath$, i.e., the unique $1$-step partition.

Each $\mathbf{v}\in\Lambda_{n,d}^\jmath$ (resp., $\Lambda_{n,d}^\imath$) determines a set partition $[\mathbf{v}]=([\mathbf{v}]_1,[\mathbf{v}]_2,\ldots,[\mathbf{v}]_{2n+1})$ (resp., $[\mathbf{v}]=([\mathbf{v}]_1,[\mathbf{v}]_2,\ldots,[\mathbf{v}]_{2n})$) of $\{1,2,\ldots,2d\}$, where $[\mathbf{v}]_i=\{\bar{v}_{i-1}+1, \bar{v}_{i-1}+2,\ldots,\bar{v}_i\}$ (recall $\bar{v}_i$ in \eqref{def:tildevi}). Take $w \in W$ and $x \in (\C^\times)^d$. Write $I=w[\mathbf{v}]=(I_1,\ldots,I_{2n+1})$. We denote by $x_I$ the $W_\mathbf{v}$-orbit containing
$(x_{w(1)},\ldots,x_{w(d)})$, which is well defined since any $w'\in W$ with $w'[\mathbf{v}]=I$ yields the same orbit $x_I$. For $1 \leq r \leq 2d$, let $\tau_r^\pm I$ be the set partition with $r$ shifted from $I_s$ to $I_{s\pm1}$ and $2d+1-r$ from $I_{2n+2-s}$ (resp., $I_{2n+1-s}$) to $I_{2n+2\mp1-s}$ (resp., $I_{2n+1\mp1-s}$). For $1\leq j\leq d$, let $\iota_jI$ be the set partition obtained from $I$ by swapping the elements $j$ and $2d+1-j$.

%---------------------
\subsection{Realization of $\mathbb{U}^{\imath\imath}$}
%----------------------
In this subsection, we take $\mathfrak{c}=\imath$.

For $1\leq i\leq 2n-1$ but $i\neq n$ and $\mathbf{v} \in \Lambda_{n,d-1}^{\imath}$,
let $p_1:\Ob_{\mathbf{v}+\mathbf{e}_i^\theta,\id,\mathbf{v}+\mathbf{e}_{i+1}^\theta}\to \Fm_{\mathbf{v}+\mathbf{e}_i^\theta}$ and $\pi:\mathbf{T}_{\mathbf{v}+\mathbf{e}_i^\theta,\id,\mathbf{v}+\mathbf{e}_{i+1}^\theta}\to \Ob_{\mathbf{v}+\mathbf{e}_i^\theta,\id,\mathbf{v}+\mathbf{e}_{i+1}^\theta}$ be projections where $\mathbf{e}_i^\theta=\mathbf{e}_i +\mathbf{e}_{2n+1-i}$.
Define 
$$\mathscr{B}_{i,\mathbf{v},k}'=e^{k\epsilon_{\bar{v}_i+1}} \pi^\ast(\mathrm{Det}(T_{p_1}^\ast)) \in K^{\breve{G}}(\mathbf{T}_{\mathbf{v}+\mathbf{e}_i^\theta,\id,\mathbf{v}+\mathbf{e}_{i+1}^\theta})$$
and 
 \begin{align*}
 \mathscr{B}_{i,k}'=\sum_{\mathbf{v} \in \Lambda_{n,d-1}^\imath}(-q)^{-\rank \Omega_{p_1}}\mathscr{B}_{i,\mathbf{v},k}'=\sum_{\mathbf{v} \in \Lambda_{n,d-1}^\imath}(-q)^{-v_i}\mathscr{B}_{i,\mathbf{v},k}' \in K^{\breve{G}}(\mathbf{Z}_n^\imath),
 \end{align*}
 where $\Omega_{p_1}$ is the sheaf of relative $1$-forms along $p_1$.

For $i=n$ and $\mathbf{v} \in \Lambda_{n,d-1}^\imath$, let $p_1:\Ob_{\mathbf{v}+\mathbf{e}_n^\theta,s_d,\mathbf{v}+\mathbf{e}_n^\theta}\to \Fm_{\mathbf{v}+\mathbf{e}_n^\theta}$ and $\pi:\mathbf{T}_{\Ob_{\mathbf{v}+\mathbf{e}_n^\theta,s_d,\mathbf{v}+\mathbf{e}_n^\theta}} \to \Ob_{\mathbf{v}+\mathbf{e}_n^\theta,s_d,\mathbf{v}+\mathbf{e}_n^\theta}$ be projections.
However, the orbit $\Ob_{\mathbf{v}+\mathbf{e}_n^\theta,s_d,\mathbf{v}+\mathbf{e}_n^\theta}$ is not closed. Therefore, we have to introduce the localization as follows. For any $R(\breve{G})$-module (resp., $\mathbf{R}^{(2d)}$-module) $M$, 
let $$\underline{M}=\K \otimes_{\mathcal{A}} M[\prod_{\alpha \in \Pi}\frac1{1-e^\alpha}]\qquad
(\mbox{resp.}\quad \underline{M}=M[\prod_{\alpha \in \Pi}\frac1{1-e^\alpha}]),$$
e.g. $\underline{K}^{\breve{G}}(\mathbf{Z}^\mathfrak{c})$, $\underline{\mathbf{R}}^{\mathbf{v}}$, $\underline{\mathbf{P}}^\mathfrak{c}$, and so on. 

\begin{rem}
The localized spaces defined above are smaller than the ones introduced in \cite{SW26}, but are large enough. The following is an explanation. Let $f: X\to Y$ be a morphism between smooth $G$-varieties $X$ and $Y$ such that $f|_{X^T}$ is proper. Assume that there exists a smooth closed $T$-subvariety $X'\supseteq X^T$ such that $f|_{X'}$ is also proper and $\bigwedge^\bullet(T_C^*X) \in R(\breve{T})[\mathcal{O}_C]$ for each connected component $C$ of $X'$. Denote $i: X' \to X$ the embedding map. Note that $(f|_{X'})_*$ is a homomorphism of $R(\breve{T})$-modules. We define $f_*$ by 
\begin{align*}
    f_*[\mathcal{F}]=(f|_{X'})_*\big(\sum_{C}\frac{(i|_{C})^*[\mathcal{F}]}{\bigwedge^\bullet(T_C^*X)}\big),
\end{align*}
which coincides with the localization formula defined in \cite[(2.1)]{SW26}. In our case, let $\mathbf{v}\in \Lambda_{n,d}^\imath$. We take $X=\mathbf{T}_{\Ob_{\mathbf{v},s_d,\mathbf{v}}}$, $X'=\pi_{\mathbf{v}\mathbf{v}}^{-1}(\Ob_{\mathbf{v},s_d,\mathbf{v}}^T)$ and $Y=F_\mathbf{v}$. Since
\begin{align*}
    [\mathcal{O}_{X'}]=\big(\prod_{\alpha \in (\Pi^-\backslash \Pi_\mathbf{v}^-)\cup s_d(\Pi^-\backslash \Pi_\mathbf{v}^-)}(1-e^\alpha)\big)[\mathcal{O}_X],
\end{align*}
then $i^*(\prod_{\alpha \in (\Pi^-\backslash \Pi_\mathbf{v}^-)\cup s_d(\Pi^-\backslash \Pi_\mathbf{v}^-)}(1-e^\alpha))[\mathcal{F}]=\bigwedge^\bullet(T_{X'}^*X)i^*[\mathcal{F}]$.
Therefore, we have
\begin{align*}
    (\prod_{\alpha \in \Pi}(1-e^\alpha))(f_*[\mathcal{F}])&=f_*((\prod_{\alpha \in \Pi}(1-e^\alpha))[\mathcal{F}])\\
    &=(f|_{X'})_*i^*(\prod_{\alpha \in (\Pi^+\cup \Pi_\mathbf{v}^-)\cap s_d(\Pi^+ \cup \Pi_\mathbf{v}^-)}(1-e^\alpha)[\mathcal{F}]) \in K^{\breve{T}}(Y).
\end{align*}
The claim in this remark follows.
\end{rem}

After the localization, we can define 
$$\mathscr{B}_{n,\mathbf{v},k}'=e^{k\epsilon_d} \pi^*(\mathrm{Det}(T_{p_1}^*)) \in \underline{K}^{\breve{G}}(\mathbf{T}_{\mathbf{v}+\mathbf{e}_n^\theta,s_d,\mathbf{v}+\mathbf{e}_n^\theta})$$
and
$$\mathscr{B}_{n,k}'=\sum_{\mathbf{v} \in \Lambda_{n,d-1}^\imath}(-q)^{-\rank \Omega_{p_1}}\mathscr{B}_{n,\mathbf{v},k}'=\sum_{\mathbf{v} \in \Lambda_{n,d-1}^\imath}(-q)^{-v_n-1}\mathscr{B}_{n,\mathbf{v},k}'\in \underline{K}^{\breve{G}}(\mathbf{Z}_n^\imath).$$

For $1\leq i\leq 2n-1$, let $\mathscr{B}_i'(z)=\sum_{k \in \Z} \mathscr{B}_{i,k}'z^{-k}$.

For any subset $I\subset \{1,2,\ldots,2d\}$, we denote
$$\Phi_I(z)=\prod_{t\in I}\theta(z/x_t)\quad\text{where}\quad \theta(z)=\frac{qz-1}{z-q}.$$
For $1 \leq i \leq 2n-1$, define 
$\hat{B}_{i,k}\in\bigoplus_{\mathbf{v}\in\Lambda_{n,d}^\imath}\mathrm{Hom}_{\underline{\mathbf{R}}^{(2d)}}(\underline{\mathbf{R}}^{\mathbf{v}-\mathbf{e}_i^\theta+\mathbf{e}_{i+1}^\theta},\underline{\mathbf{R}}^\mathbf{v})$ by
$$(\hat{B}_{i,k}f)(x_{[\mathbf{v}]})=\sum_{r\in [\mathbf{v}]_{i}}x_r^k\Phi_{(\tau_r^+[\mathbf{v}])_i}(qx_r)f(x_{\tau_r^+ [\mathbf{v}]})\in \underline{\mathbf{R}}^\mathbf{v}, \quad \forall f\in \underline{\mathbf{R}}^{\mathbf{v}-\mathbf{e}_i^\theta+\mathbf{e}_{i+1}^\theta},\ x \in (\C^\times)^d,$$
and denote $$\hat{B}_i(z)=\sum_{k\in\mathbb{Z}}\hat{B}_{i,k}q^{ki}z^k.$$

\begin{lem}\label{ope}
    For $1\leq i\leq 2n-1$, we have
    \begin{align*}
        \hat{B}_i(z)f=\begin{cases}
            \mathscr{B}_i'(q^{-i}z^{-1})f, & \mbox{if} \ i\neq n,\\
         \frac{1-q^{-2n+2}z^{-2}}{(1-q^{-n}q_0z^{-1})(1-q^{-n}q_1z^{-1})}\mathscr{B}_n'(q^{-n}z^{-1})f, & \mbox{if} \ i=n.
        \end{cases}
    \end{align*}
\end{lem}
\begin{proof}
We claim 
\begin{align*}
    (\mathscr{B}_{i,k}'f)(x_{[\mathbf{v}]})=\begin{cases}
        \sum_{r \in [\mathbf{v}]_i} x_r^k\Phi_{(\tau_r^+[\mathbf{v}])_i}(qx_r)(f(x_{\tau_r^+[\mathbf{v}]})), & \mbox{if} \ i\neq n,\\
     \sum_{r \in [\mathbf{v}]_n} x_r^k\frac{(1-q_0x_r)(1-q_1x_r)}{1-q^2x_r^2}\Phi_{(\tau_r^+[\mathbf{v}])_n}(qx_r)(f(x_{\tau_r^+[\mathbf{v}]})), & \mbox{if} \ i=n.
    \end{cases}
\end{align*}
For $i\neq n$, the calculation is the same as that in \cite[\S 6]{SW26}. For $i=n$, we only need to replace $\bigwedge_{q^2}T_{p_1}$ in the proof of \cite[Proposition 6.2]{SW26} with $(1-q_0e^{\epsilon_d})(1-q_1e^{\epsilon_d})\prod_{i \in [\mathbf{v}]_n\backslash\{d\}}(1-q^2e^{\epsilon_d-\epsilon_i})$.

Then comparing with the action of $\hat{B}_{i,k}$, the lemma follows.
\end{proof}

For $1\leq i\leq 2n-1$, the operators $\hat{\mathbb{K}}_i$ and $\hat{\Theta}_i(z)$ are defined as the action on $\underline{\mathbf{R}}^{\mathbf{v}}$ by multiplication with the scalar $(q_0q_1q^{-1})^{\delta_{ni}}q^{-v_i+v_{i+1}}$ and the rational function $q^{v_{i+1}-v_i}\Phi_{i,\mathbf{v}}(z^{-1})$ respectively, where
$$\Phi_{i,\mathbf{v}}(z):=
\Phi_{[\mathbf{v}]_i}(q^{1-i}z)\Phi_{[\mathbf{v}]_{i+1}}(q^{-1-i}z)^{-1}. 
$$
Write
\begin{align*}
    \hat{\Theta}_i'(z)&=\hat{\Theta}_i(z), \ \text{for} \ i \neq n,\\
    \hat{\Theta}_n'(z)
    &=\frac{(1-q^nq_0z)(1-q^nq_1z)(1-q^nq_0^{-1}z)(1-q^nq_1^{-1}z)}{(1-q^{2n+2}z^2)(1-q^{2n-2}z^2)}\hat{\Theta}_n(z).
\end{align*}

\begin{thm}
\label{thm:ii}
The assignment, for $1\leq i \leq 2n-1$,
\begin{align*}
 C \mapsto q^{2n}, \quad \mathbf{\Theta}_i(z) \mapsto \hat{\Theta}_i'(z),
    \quad \K_i \mapsto \hat{\mathbb{K}}_i,
    \quad \mathbf{B}_i(z) \mapsto  \mathscr{B}_i'(q^{-i}z^{-1}), 
\end{align*}
gives a $\K$-algebra homomorphism from $\mathbb{U}^{\imath\imath}$ to $\underline{K}^{\breve{G}}(\mathbf{Z}^\imath_{n})$.
\end{thm}
\begin{proof}
    Using \cite[Theorem 4.7]{SW26} and Lemma \ref{ope}, we can check the relations by representing $\mathscr{B}_i'(z), \hat{\Theta}_i'(z)$ with $\hat{B}_i(z), \hat{\Theta}_i(z)$.
\end{proof}

%---------------------
\subsection{Realization of $\mathbb{U}^{\imath\jmath}$}
%----------------------
In this subsection, we take $\mathfrak{c}=\jmath$.

For $1\leq i\leq 2n$ and $\mathbf{v} \in \Lambda_{n,d-1}^{\jmath}$,
let $p_1:\Ob_{\mathbf{v}+\mathbf{e}_i^\theta,\id,\mathbf{v}+\mathbf{e}_{i+1}^\theta}\to \Fm_{\mathbf{v}+\mathbf{e}_i^\theta}$ and $\pi:\mathbf{T}_{\mathbf{v}+\mathbf{e}_i^\theta,\id,\mathbf{v}+\mathbf{e}_{i+1}^\theta}\to \Ob_{\mathbf{v}+\mathbf{e}_i^\theta,\id,\mathbf{v}+\mathbf{e}_{i+1}^\theta}$ be projections where $\mathbf{e}_i^\theta=\mathbf{e}_i+\mathbf{e}_{2n+2-i}$.
Define 
$$\mathscr{B}_{i,\mathbf{v},k}'=e^{k\epsilon_{\bar{v}_i+1}} \pi^\ast(\mathrm{Det}(T_{p_1}^\ast)) \in K^{\breve{G}}(\mathbf{T}_{\mathbf{v}+\mathbf{e}_i^\theta,\id,\mathbf{v}+\mathbf{e}_{i+1}^\theta}).$$
Denote 
 \begin{align*}
 \mathscr{B}_i'(z)=\sum_{k \in \Z} \mathscr{B}_{i,k}'z^{-k} 
 \end{align*}
 where
 \begin{align*}
 \mathscr{B}_{i,k}'=\sum_{\mathbf{v} \in \Lambda_{n,d-1}^\jmath}(-q)^{-\rank \Omega_{p_1}}\mathscr{B}_{i,\mathbf{v},k}'=\sum_{\mathbf{v} \in \Lambda_{n,d-1}^\jmath}(-q)^{-v_i-\delta_{i,n+1}}\mathscr{B}_{i,\mathbf{v},k}' \in K^{\breve{G}}(\mathbf{Z}_n^\jmath).
 \end{align*}

For $1\leq i \leq 2n$, define $\hat{B}_{i,k}\in\bigoplus_{\mathbf{v}\in\Lambda_{n,d}^\jmath}\mathrm{Hom}_{\mathbf{R}^{(2d)}}(\mathbf{R}^{\mathbf{v}-\mathbf{e}_i^\theta+\mathbf{e}_{i+1}^\theta},\mathbf{R}^\mathbf{v})$ by
$$(\hat{B}_{i,k}f)(x_{[\mathbf{v}]})=\sum_{j\in [\mathbf{v}]_i}x_j^k\Phi_{[\mathbf{v}]_i\setminus\{j\}}(qx_j)f(x_{\tau_j^+ [\mathbf{v}]})\in \mathbf{R}^\mathbf{v}, \quad \forall f\in \mathbf{R}^{\mathbf{v}-\mathbf{e}_i^\theta+\mathbf{e}_{i+1}^\theta},\ x \in (\C^\times)^d.$$  
We denote
\begin{align*}
\hat{B}_i(z)&=\sum_{k\in\mathbb{Z}}\hat{B}_{i,k}q^{ki}z^k.
\end{align*}

Comparing the actions on $\mathbf{P}^\jmath$, we obtain the following lemma.
\begin{lem}\label{ope2}
    It holds
    \begin{align*}
\hat{B}_i(z)f&=(\frac{1-q^{-2n}z^{-2}}{(1-q^{-n-1}q_0z^{-1})(1-q^{-n-1}q_1z^{-1})})^{\delta_{i,n+1}} \mathscr{B}_i'(q^{-i}z^{-1})f.
    \end{align*}
\end{lem}
\begin{proof}
Imitating the calculation steps given in \cite{LSX26}, we have that
\begin{align*}
    (\mathscr{B}_{i,k}'f)(x_{[\mathbf{v}]})&=\sum_{r \in [\mathbf{v}]_i} x_r^k (\frac{(1-q_0x_r)(1-q_1x_r)}{1-q^2x_r^2})^{\delta_{i,n+1}}\Phi_{[\mathbf{v}]_i \backslash \{r\}}(qx_r)(f(x_{\tau_r^+[\mathbf{v}]})).
\end{align*}
Then comparing with the action of $\hat{B}_{i,k}$, the lemma follows.
\end{proof}

For $1\leq i\leq 2n$, the operators $\hat{\mathbb{K}}_i$ and $\hat{\Theta}_i(z)$  are defined as the action on $\mathbf{R}^{\mathbf{v}}$ by multiplication with the scalar $(-q^{-1}q_0 q_1)^{\delta_{in}}q^{-v_i+v_{i+1}}$ and the rational function $q^{v_{i+1}-v_i-\delta_{n+1,i}}\Phi_{i,\mathbf{v}}(z^{-1})$ respectively, where
$$\Phi_{i,\mathbf{v}}(z):=
(\theta(q^{-2n-1}z^2))^{\delta_{i,n+1}}\Phi_{[\mathbf{v}]_i}(q^{1-i}z)\Phi_{[\mathbf{v}]_{i+1}}(q^{-1-i}z)^{-1}. 
$$

Write
\begin{align*}
    \hat{\Theta}_i'(z)&=\hat{\Theta}_i(z), \ \text{for} \ i \neq n,n+1,\\
    \hat{\Theta}_n'(z)%&=1+(q-q^{-1})\sum_{k\in \N^+} \hat{\Theta}_{n,k}'z^k=(q-q^{-1})\sum_{k\in \N} \hat{\Theta}_{n,k}'z^k\\
    &=\frac{(1-q^nq_0z)(1-q^nq_1z)}{1-q^{2n+2}z^2}\hat{\Theta}_n(z),\\
    \hat{\Theta}_{n+1}'(z)%&=1+(q-q^{-1})\sum_{k\in \N^+} \hat{\Theta}_{n+1,k}'z^k=(q-q^{-1})\sum_{k\in \N} \hat{\Theta}_{n+1,k}'z^k\\
    &=\frac{(1-q^{n+1}q_0^{-1}z)(1-q^{n+1}q_1^{-1}z)}{1-q^{2n}z^2}\hat{\Theta}_{n+1}(z).
\end{align*}

\begin{thm}
\label{thm:ij}
The assignment, for $1\leq i \leq 2n$,
\begin{align*}
    C \mapsto q^{2n+1}, \quad \K_i \mapsto \hat{\mathbb{K}}_i,  \quad \mathbf{\Theta}_i(z) \mapsto \hat{\Theta}_i'(z), \quad \mathbf{B}_i(z) \mapsto \mathscr{B}_i'(q^{-i}z^{-1})
\end{align*}
gives a $\K$-algebra homomorphism from $\mathbb{U}^{\imath\jmath}$ to $K^{\breve{G}}(\mathbf{Z}^\jmath_{n})\otimes_{\mathcal{A}}{\mathbb{K}}$.
\end{thm}
\begin{proof}
    Using \cite[Theorem 4.4]{LSX26} and Lemma \ref{ope2}, we can check the relations by representing $\mathscr{B}_i'(z), \hat{\Theta}_i'(z)$ with $\hat{B}_i(z), \hat{\Theta}_i(z)$.
\end{proof}
%==================

\appendix
\section{Affine types B/C/F/G with two parameters}\label{Apped:A}

In this appendix, we provide an equivariant K-theoretic construction of affine quantum Schur algebras of types B/C/F/G with two parameters, following Antor's recent work \cite{An25}. Since the arguments for the two-parameter case are similar to the previous case of three parameters, we shall not provide proofs in this appendix.

%--------------------------------------------------------
\subsection{Exotic representation}
%--------------------------------------------------------
Let $G$ be a simply connected simple algebraic group of non simply-laced type defined over an algebraic closed field $\mathbf{k}$ with the special characteristic. Recall that the characteristic $p$ of $\mathbf{k}$ is called special for $G$ if it is equal to the ratio's square of the lengths of the long and short roots of $G$, i.e.
\begin{itemize}
    \item $p=2$ when $G$ is of type $B_n$, $C_n$ or $F_4$;
    \item $p=3$ when $G$ is of type $G_2$.
\end{itemize}
Let $\mathfrak{g}$ be the adjoint representation of $G$, so that the set of its nonzero weights is just the root system $\Pi$ of $G$.

It was showed in \cite{Ho82} and \cite{Hi84} that if $p$ is special for $G$ then there is a $G$-subspace $\mathfrak{g}_s\subseteq\mathfrak{g}$ whose nonzero weights are the short roots of $G$. Antor \cite{An25} introduced the following exotic representation of $G$: $$\mathbb{V}=\mathfrak{g}_s\oplus\mathfrak{g}/\mathfrak{g}_s,$$ which admits a canonical $G\times(\mathbf{k}^\times)^2$-action via
$$(g,z_1,z_2)\cdot(v_1,v_2)=(z^{-1}_1gv_1,z^{-1}_2gv_2)$$ where $(g,z_1,z_2)\in G\times(\mathbf{k}^\times)^2$ and $(v_1,v_2)\in \mathbb{V}$.

In this section, we reset $\mathcal{A}=\mathbb{Z}[q_1^{\pm1},q_2^{\pm1}]$.
Let $\Delta$ be the set of simple roots of $G$. Define a map $\mathbf{q}:\Delta\to\{q_1,q_2\}$ such that $\mathbf{q}(\alpha)=q_1$ (resp., $q_2$) if $\alpha$ is a short root (resp., long root).

%--------------------------------------------------------
\subsection{Affine Hecke and Schur algebras with two parameters}
%--------------------------------------------------------

As before, let $W$ be the Weyl group of $G$ associated with a fixed maximal torus $T\subseteq G$, and $Q$ the weight lattice.
    The affine Hecke algebra $\widetilde{\mathbb{H}}_{q_1,q_2}$ (associated with $G$) with unequal parameters is the $\mathcal{A}$-algebra with a basis $\{T_we^\lambda\mid w\in W,\ \lambda\in Q\}$ and the following relations:
    \begin{itemize}
        \item $(T_\alpha+1)(T_\alpha-\mathbf{q}(\alpha))=0$, $\alpha\in\Delta$;
        \item $T_wT_{w'}=T_{ww'}$, if $\ell(ww')=\ell(w)+\ell(w')$;
        \item $e^\lambda e^\mu=e^{\lambda+\mu}$, $\lambda,\mu\in Q$;
        \item $T_\alpha e^\lambda-e^{s_\alpha\lambda}T_\alpha=(1-\mathbf{q}(\alpha))\frac{e^\lambda-e^{s_\alpha\lambda}}{e^\alpha-1}$, $\alpha\in\Delta$.
    \end{itemize}

Suppose that $w=s_{i_1}\cdots s_{i_k}$ is a reduced form of $w\in W$. Define $q_{w}=q_{s_{i_1}}\cdots q_{s_{i_k}}$, where $q_{s_i}=-\mathbf{q}(\alpha_i)$. It is known that $q_w$ is independent of the choice of a reduced form of $w$.

Set
$$\mathbf{x}_{\mathbf{v}}=\sum_{w\in W_{\mathbf{v}}}q_{w}^{-1}T_{w}\in\widetilde{\mathbb{H}}_{q_1,q_2}\quad \text{and} \quad \widetilde{\mathbb{T}}_{\ff,q_1,q_2}=\bigoplus_{\mathbf{v}\in\Lambda_{\ff}}\mathbf{x}_\mathbf{v}\widetilde{\mathbb{H}}_{q_1,q_2}.$$
The two-parameter affine quantum Schur algebra $\widetilde{\mathbb{S}}_{\ff,q_1,q_2}$ is an $\mathcal{A}$-algebra defined as $$\widetilde{\mathbb{S}}_{\ff,q_1,q_2}=\mathrm{End}_{\widetilde{\mathbb{H}}_{q_1,q_2}}(\widetilde{\mathbb{T}}_{\ff,q_1,q_2}).$$

%--------------------------------------------------------
\subsection{Geometric construction}
%--------------------------------------------------------
Fix a Borel subgroup $B$ such that $T\subseteq B\subseteq G$. We use the same notations as in the main text for $\mathrm{Sp}_{2d}(\C)$ to denote corresponding notions for $G$ in this appendix, such as the set $\Lambda_\ff$ of $W$-orbit, the parabolic subgroup $P_\mathbf{v}$, the partial flag variety $\Fm_\mathbf{v}$, etc. 

For $\mathbf{v}\in\Lambda_\ff$, set $$\mathbb{V}^{+}_\mathbf{v}=\bigoplus\limits_{\lambda\in\Pi^+\backslash\Pi^+_\mathbf{v}}\mathbb{V}[\lambda],$$ which is a $P_\mathbf{v}$-module.
Define the exotic cotangent bundle of the partial flag variety $\Fm_\mathbf{v}$ as follows: $$F_\mathbf{v}=G\times^{P_\mathbf{v}}\mathbb{V}^+_\mathbf{v},$$ which is an analog of the cotangent bundle of $\Fm_\mathbf{v}$. Using the canonical projection $F_\mathbf{v}\to \mathbb{V}$, we 
define the exotic generalized Steinberg variety
$$\mathbf{Z}_\ff=\bigsqcup\limits_{\mathbf{v},\mathbf{w}\in\Lambda_\ff}\mathbf{Z}_{\mathbf{v}\mathbf{w}}\quad\text{with}\quad \mathbf{Z}_{\mathbf{v}\mathbf{w}}=F_\mathbf{v}\times_{\mathbb{V}}F_\mathbf{w}.$$

As in the main text, we fix a regular $W$-orbit $\mathbf{u}_0$. When $\mathbf{v},\mathbf{w}=\mathbf{u}_0$, we rewrite $F=F_{\mathbf{u}_0}$ and $\mathbf{Z}=\mathbf{Z}_{\mathbf{u}_0\mathbf{u}_0}=F\times_{\mathbb{V}}F$, which are just the objects studied in \cite{An25}.

We denote by $\mathbf{e}^\lambda$ the pull-back of $G\times^{B}\mathbf{k}_{-\lambda}$ to $F$. Let $\pi:\mathbf{Z}\to\Bm\times\Bm$ be the projection.
Recall $\Bm\times\Bm=\bigsqcup\limits_{w\in W}\Ob_{w}$. Denote $\mathbf{T}_{\Ob_w}=\pi^{-1}(\Ob_w)$, and let $\mathbf{T}_w$ be its Zariski closure in $\mathbf{Z}$.

In \cite{An25}, Antor provided a K-theoretic construction of affine Hecke algebra with two parameters using the variety $\mathbf{Z}$. Here we give a variation of Antor's isomorphism, to align with Kato's expression convention.
\begin{thm}[{cf. \cite[Theorem 2.19]{An25}}]
    There is an isomorphism of algebras
    $$\widetilde{\mathbb{H}}_{q_1,q_2}\simeq K^{G\times(\mathbf{k}^\times)^2}(\mathbf{Z}),$$
    via $$\vartheta:e^{\lambda}\mapsto \mathbf{e}^{\lambda},\quad T_\alpha\mapsto[\mathcal{O}_{\mathbf{T}_{s_\alpha}}]-(1-\mathbf{q}(\alpha)(\mathbf{e}^{\alpha}+1)).$$
\end{thm}

Similarly to \eqref{eq:evlambda}, let us define the element $e_{\mathbf{v}}^\lambda:=[\mathcal{O}_{\mathbf{T}_{\mathbf{v},\id}}(-\lambda)]\in K^{G\times(\mathbf{k}^\times)^2}(F_{\ff}\times_{\mathbb{V}}F)$. 

%In this case, Corollary~\ref{}
%\begin{lem}
%Let $\alpha\in\Delta_\mathbf{v},$ we have $e_{\mathbf{v}}^{-\rho_\mathbf{v}}*\vartheta(T_\alpha)=-e_{\mathbf{v}}^{-\rho_\mathbf{v}}$.
%\end{lem}

\begin{thm}[Geometric affine quantum Schur duality with two parameters]\label{main2}
\
\begin{itemize}
    \item[(1)] There exists an $\widetilde{\mathbb{H}}_{q_1,q_2}$-module isomorphism $\widetilde{\phi}:\widetilde{\mathbb{T}}_{\ff,q_1,q_2}\to K^{G\times(\mathbf{k}^\times)^2}(F_{\ff}\times_{\mathbb{V}}F)$ via $\mathbf{x}_{\mathbf{v}}\mapsto e_{\mathbf{v}}^{-\rho_{\mathbf{v}}}$, where $\rho_{\mathbf{v}} \in Q$ is the Weyl vector associated with $\Pi_\mathbf{v}$.
    \item[(2)] The isomorphism $\widetilde{\phi}$ induces an algebra isomorphism $\widetilde{\psi}:K^{G\times(\mathbf{k}^\times)^2}(\mathbf{Z}_{\ff})\to\widetilde{\mathbb{S}}_{\ff,q_1,q_2}$.
\end{itemize}
\end{thm}
\begin{cor}
    Let $\mathbf{v},\mathbf{w}\in\Lambda_\ff$, we have
    $$\Hom_{\widetilde{\mathbb{H}}_{q_1,q_2}}(\mathbf{x}_\mathbf{w}\widetilde{\mathbb{H}}_{q_1,q_2},\mathbf{x}_\mathbf{v}\widetilde{\mathbb{H}}_{q_1,q_2})\simeq K^{G\times(\mathbf{k}^\times)^2}(\mathbf{Z}_{\mathbf{v}\mathbf{w}}).$$
\end{cor}
Denote $\widetilde{\mathbb{T}}_{\ff\g,q_1,q_2}=K^{G\times(\mathbf{k}^\times)^2}(F_\ff\times_\mathbb{V} F_\g)$, which admits a left $\widetilde{\mathbb{S}}_{\ff,q_1,q_2}$-action and a right $\widetilde{\mathbb{S}}_{\g,q_1,q_2}$-action under the convolution product.
\begin{thm}
If the family of minimal parabolic subgroups of $\{P_\mathbf{v}~|~\mathbf{v}\in\Lambda_\ff\}$ and of $\{P_\mathbf{v}~|~\mathbf{v}\in\Lambda_\g\}$ coincide, then $\widetilde{\mathbb{S}}_{\ff,q_1,q_2}$ and $\widetilde{\mathbb{S}}_{\g,q_1,q_2}$ admit a double centralizer property on $\widetilde{\mathbb{T}}_{\ff\g,q_1,q_2}$.
\end{thm}

%=============================================
\section{Algebraic equivariant K-theory} \label{Apped:B}
This appendix summarizes some standard facts about algebraic equivariant K-theory. 

\subsection{Basic notions}

Let $X$ be a $G$-variety. Denote by $\mathrm{Coh}^G(X)$ the category of $G$-equivariant coherent sheaves on $X$. Let $K^G(X)$ be the Grothendieck group of $\mathrm{Coh}^G(X)$. The $K$-group carries a canonical $R(G)$-module structure. For $\mathcal{F}\in\mathrm{Coh}^G(X)$, let $[\mathcal{F}]$ denote its class in $K^G(X)$. Let us recall some properties of the equivariant K-theory.
\begin{itemize}
\item[(a)] Push-forward. \\
For any proper morphism $f:X\to Y$ between $G$-varieties, there is a direct image $f_*:K^G(X)\to K^G(Y)$ via $[\mathcal{F}]\mapsto \sum_i (-1)^i[R^if_*\mathcal{F}]$, which is a finite sum due to the properness. If $f$ is the structure morphism, then $f_*$ becomes the derived global sections functor $R\Gamma$.
\item[(b)] Pull-back. \\
If $f:X\to Y$ is flat (for instance, an open immersion) or Y is a smooth $G$-variety, there is an inverse image $f^*: K^G(Y)\to K^G(X)$.
\item[(c)] Proper base change. \\
If $G$-varieties $X_1,X_2,Y_1,Y_2$ admit a Cartesian square $$\begin{tikzcd}
X_1 \arrow[r, "f_1"] \arrow[d, "g"'] & Y_1 \arrow[d, "h"] \\
X_2 \arrow[r, "f_2"]                 & Y_2
\end{tikzcd}$$ with $f_2$ being flat and $h$ being proper, then $ g_*\circ f_{1}^*=f_{2}^*\circ h_*$.
\item[(d)] External tensor product. \\
Let $X, Y$ be $G$-varieties. There is an external tensor product $\boxtimes:K^G(X)\times K^G(Y)\to K^G(X\times Y)$ via $([\mathcal{F}],[\mathcal{F}'])\mapsto [\mathcal{F}\boxtimes\mathcal{F}']$.
\item[(e)] Derived tensor product.\\
Let $X$ be a smooth $G$-variety.
Let $\mathcal{F},\mathcal{F}'\in\mathrm{Coh}^G(X)$. 
The derived tensor product is $\otimes_X :K^G(X)\times K^G(X)\to K^G(X)$ via $([\mathcal{F}],[\mathcal{F}'])\mapsto\sum_i(-1)^i[\mathcal{T}or_i^{\mathcal{O}_X}(\mathcal{F},\mathcal{F}')],$ which makes $K^G(X)$ a commutative ring. 
Moreover, the derived tensor product commutes with the pull-back.
\item[(f1)] Let $i:X\to Y$ be a closed immersion of $G$-varieties. In this case, the higher direct image vanishing. Let $\mathcal{F}\in\mathrm{Coh}^G(X)$, then $i_*[\mathcal{F}]=[i_*\mathcal{F}]$, which is a class of coherent sheaves on $Y$ supported by $X$. Denote by $\mathrm{Coh}^G(Y,X)$ the full subcategory of $\mathrm{Coh}^G(Y)$ consisting of the coherent sheaves on $Y$ supported by $X$. Let $K^G(Y,X)$ be the Grothendieck group of $\mathrm{Coh}^G(Y,X)$, then $i_*$ induces a natural isomorphism $K^G(X)\simeq K^G(Y,X)$.
\item[(f2)] Let $X'$ be a closed $G$-subvariety of $X$. If $f:X\to Y$ is a $G$-morphism such that $f|_{X'}$ is proper, then there is a well-defined direct image $f_*:K^G(X,X')\to K^G(Y)$.
\item[(f3)] Let $X'$ and $X''$ be two closed $G$-subvarieties of a smooth $G$-variety $X$. Let $\mathcal{F}$ and $\mathcal{G}$ be sheaves in $\mathrm{Coh}^G(X,X')$ and $\mathrm{Coh}^G(X,X'')$, respectively. Then the Tor sheaves $\mathcal{T}or_i^{\mathcal{O}_X}(\mathcal{F}, \mathcal{G})$ is supported by $X'\cap X''$, hence there exists a well-defined map (by abuse of notation) $$\otimes_X :K^G(X')\times K^G(X'')\to K^G(X'\cap X'').$$
\item[(g)] Projection formula.\\
Let $f:X\to Y$ be a $G$-equivariant proper morphism of smooth $G$-varieties. Let $\mathcal{F}$ be an equivariant coherent sheaf on $X$, and $\mathcal{E}$ an equivariant locally free sheaf on $Y$. Then we have the following formula $$f_*([\mathcal{F}]\otimes f^*[\mathcal{E}])=f_*[\mathcal{F}]\otimes [\mathcal{E}].$$
\item[(h)] Thom isomorphism.\\
Let $\pi:E\to X$ be a $G$-equivariant vector bundle. There is a natural isomorphism $\pi^*: K^G(X)\simeq K^G(E)$.
\item[(i)] Koszul complex.\\
Let $\pi:E\to X$ be a $G$-equivariant vector bundle of rank $r$, and $E^\vee$ its dual bundle. Let $i:X\to E$ be the zero section. One has a $G$-equivariant locally free resolution of $i_*\mathcal{O}_X$:$$0 \rightarrow\pi^*(\wedge^r E^\vee)\rightarrow\cdots\rightarrow\pi^*(\wedge^2E^\vee)\rightarrow\pi^*(\wedge^1E^\vee)\rightarrow\mathcal{O}_E\rightarrow i_*\mathcal{O}_X.$$ 
\end{itemize}

%------------------------------------
\subsection{Convolution product}
%------------------------------------
Let $M_1,M_2$ and $M_3$ be smooth $G$-varieties. Let $$p_{ij}:M_1\times M_2\times M_3\to M_i\times M_j \quad (1\leq i<j\leq3)$$ be the projections along the factor not named, which is naturally $G$-equivariant. Let $Z_{12}\subseteq M_1\times M_2$ and $Z_{23}\subseteq M_2\times M_3$ be closed G-subvarieties. Assume that the restriction of $p_{13}$ to $p_{12}^{-1}(Z_{12})\cap p_{23}^{-1}(Z_{23})$ is proper and let $Z_{13}=Z_{12}\circ Z_{23}$ be its image. Define a convolution \begin{align*}
*:K^G(Z_{12})\times K^G(Z_{23})&\to K^G(Z_{13}),\\([\mathcal{F}],[\mathcal{G}])&\mapsto p_{13*}(p^*_{12}[\mathcal{F}]\otimes_{M_1 \times M_2 \times M_3} p^*_{23}[\mathcal{G}]),
\end{align*}
where we regard $p^*_{12}[\mathcal{F}]\in K^{G}(M_1\times M_2\times M_3,p_{12}^{-1}(Z_{12}))$ and $p^*_{23}[\mathcal{G}]\in K^{G}(M_1\times M_2\times M_3,p_{23}^{-1}(Z_{23}))$ respectively, the tensor product is defined as in (f3) and $p_{13*}$ is well-defined as in (f2), and finally we use the isomorphism in (f1).

%------------------------------------
\subsection{Cellular fibration lemma}
%------------------------------------
We recall the cellular fibration lemma in \cite[5.5.1]{CG97}.
Let $H$ be an algebraic group and let $\pi: F\to X$ be a morphism of $H$-varieties with $X$ being smooth. We call $F$ an $H$-cellular fibration over $X$ if $F$ is equipped with a finite decreasing filtration $F=F^n\supseteq F^{n-1}\supseteq\cdots\supseteq F^0=\emptyset$ such that for any $i=1,2,\cdots,n,$ the following holds:
\begin{itemize}
    \item[(a)] $F^i$ is a closed $H$-subvariety; furthermore, the restriction $\pi: F^i\to X$ is an $H$-equivariant locally trivial fibration.
    \item[(b)] The map $\pi_i: F^i\backslash F^{i-1}\to X$, the restriction of $\pi$, is an $H$-equivariant affine fibration.
\end{itemize}
Consider the following maps $X\stackrel{\widetilde{\pi}_i}\leftarrow\overline{F^i\backslash F^{i-1}}\stackrel{\iota_i}\hookrightarrow F$, where $\widetilde{\pi}_i$ is the restriction of $\pi$ on $\overline{F^i\backslash F^{i-1}}$ and $\iota_i$ is the natural embedding.
\begin{lem}[Cellular Fibration Lemma]\label{cfl}
    In the above setup, the following holds.
    \begin{itemize}
        \item[(a)] For each $i=1,\cdots,n$, there is a canonical short exact sequence $$0\longrightarrow K^H(F^{i-1})\longrightarrow K^H(F^i)\longrightarrow K^H(F^i\backslash F^{i-1})\longrightarrow0.$$
        \item[(b1)] All exact sequences in (a) are split as $K^H(X)$-modules, and $K^H(F)$ is a free $K^H(X)$-module with a basis $$\{[\mathcal{O}_{\overline{F^i\backslash F^{i-1}}}] ~|~ i=1,\cdots,n\}.$$
        \item[(b2)] If $K^H(X)$ is a free $R(H)$-module with basis $\mathcal{F}_1,\cdots,\mathcal{F}_{m_i}$ then all exact sequences in (a) are split as $R(H)$-modules. Moreover, $K^H(F)$ is a free $R(H)$-module with a basis $$\{\iota_{i*}\widetilde{\pi}_i^*[\mathcal{F}_j] ~|~ i=1,\cdots,n, \ j=1,\cdots,m_i\}.$$
        \item[(c)] Let $H'\subseteq H$ be a closed algebraic subgroup and suppose that the assumption of (b2) holds for both $H$ and $H'$. If the natural map $$R(H')\otimes_{R(H)}K^H(X)\to K^{H'}(X)$$ is an isomorphism, then so is the map $$R(H')\otimes_{R(H)}K^H(F)\to K^{H'}(F).$$
    \end{itemize}
\end{lem}

\vspace{1cm}

%\noindent
%{\bf Data Availability Statement:}
%This manuscript has no associated data.

%\vspace{0.5cm}
%\noindent
%{\bf Conflict of Interest Statement:}
%On behalf of all authors, the corresponding author states that there is no conflict of interest.

\end{document}